\documentclass{amsart}

\usepackage{graphicx,epsfig}

\usepackage{amsmath,amssymb,latexsym, amsfonts, amscd, amsthm, xy}

\usepackage{url}

\usepackage{hyperref}

\usepackage{arevmath}

\usepackage{mathrsfs}

\usepackage[latin1]{inputenc}

\input{xy}

\theoremstyle{plain}
\newtheorem{theorem}{Theorem}[section]

\newtheorem{proposition}[theorem]{Proposition}
\newtheorem{corollary}[theorem]{Corollary}

\newtheorem{lemma}[theorem]{Lemma}

\theoremstyle{definition}
\newtheorem{definition}[theorem]{Definition}
\newtheorem{remark}[theorem]{Remark}
\newtheorem{example}[theorem]{Example}

\def\bA{\mathbb{A}}
\def\bB{\mathbb{B}}
\def\bL{\mathbb{L}}
\def\bC{\mathbb{C}}
\def\bF{\mathbb{F}}

\def\bN{\mathbb{N}}

\def\bQ{\mathbb{Q}}
\def\bZ{\mathbb{Z}}

\def\P{\mathbf{P}}

\def\cD{\mathcal{D}}

\def\cF{\mathcal{F}}

\def\cM{\mathcal{M}}

\def\cO{\mathcal{O}}

\def\cP{\mathcal{P}}

\def\cR{\mathcal{R}}

\def\cU{\mathcal{U}}

\def\fp{\mathfrak{p}}

\def\deg{\mathbf{deg}}

\def\Gal{\mathrm{Gal}}

\def\Norm{\mathrm{Norm}}

\def\ulim{\mathbf{ulim}}

\def\alg{\mathbf{alg}}

\def\Im{\mathrm{Im}}

\def\card{\mathrm{card}}

\def\GF{\mathrm{GF}}

\def\icard{\mathrm{icard}}

\begin{document}

\title[Reciprocity Law and Analogue of the Grunwald--Wang theorem]{Higher reciprocity law and An analogue of the Grunwald--Wang theorem for the ring of polynomials over an ultra-finite field}

\author{Dong Quan Ngoc Nguyen}

\address{Department of Mathematics, George Washington University, DC 20052}

\date{October 1, 2023}

\email{\href{mailto:dongquan.ngoc.nguyen@gmail.com}{\tt dongquan.ngoc.nguyen@gmail.com}\\
\href{mailto:dongquan.nguyen@gwu.edu}{\tt dongquan.nguyen@gwu.edu}}

\maketitle

\tableofcontents

\begin{abstract}

In this paper, we establish an explicit higher reciprocity law for the polynomial ring over a nonprincipal ultraproduct of finite fields. Such an ultraproduct can be taken over the same finite field, which allows to recover the classical higher reciprocity law for the polynomial ring $\bF_q[t]$ over a finite field $\bF_q$ that is due to Dedekind, K\"uhne, Artin, and Schmidt. On the other hand, when the ultraproduct is taken over finite fields of unbounded cardinalities, we obtain an explicit higher reciprocity law for the polynomial ring over an infinite field in both characteristics $0$ and $p >0$ for some prime $p$. We then use the higher reciprocity law to prove an analogue of the Grunwald--Wang theorem for such a polynomial ring in both characteristics $0$ and $p > 0$ for some prime $p$.

\end{abstract}

\section{Introduction}
\label{sec-Introduction}

In this paper, we establish higher reciprocity laws for power residue symbols in the polynomial ring $\bF[t]$ and use them to prove an analogue of the Grunwald--Wang theorem for $\bF[t]$, where $\bF$ is a nonprincipal ultraproduct of finite fields of unbounded cardinalities. Such fields $\bF$ were introduced in a more general context by Ax (se \cite[Section 7]{ax-1968}). In fact Ax \cite{ax-1968} introduced \textit{pseudo-finite fields} that are infinite models of the first-order theory of finite fields, and thus inherit many similar properties of finite fields. Nonprincipal ultraproducts of finite fields of unbounded cardinalities which we call \textit{ultra-finite fields} in this paper (see Section \ref{sec-ultrafinite-fields}), are examples of pseudo-finite fields as shown by Ax \cite{ax-1968}.

A nonprincipal ultraproduct of fields is called an \textit{ultra-field} in this paper. The terminology ``\textit{ultra-field}" was first used in Schoutens \cite{schoutens}. Inspired by the use of such a terminology, throughout this paper, we always use the prefix ``\textit{ultra}" for any nonprincipal ultraproduct of algebraic objects. For example, a nonprincipal ultraproduct of groups is called an \textit{ultra-group}, a nonprincipal ultraproduct of rings is called an \textit{ultra-ring}, and so on. In order to establish higher reciprocity law in the polynomial ring $\bF[t]$ over an ultra-finite field $\bF$, we systematically develop the theory of ultra-fields--an analogue of the theory of fields. Closely related to such a theory is a notion of ultra-polynomials which are elements in a so-called \textit{ultra-hull}--a nonprincipal ultraproduct of polynomial rings. More precisely, if $F = \prod_{s\in S}F_s/\cD$ is an ultra-field with respect to a nonprincipal ultrafilter $\cD$ for some fields $F_s$, the ultra-hull of the polynomial ring $F[t]$, denoted by $\cU(F[t])$, is a nonprincipal ultraproduct of polynomial rings $F_s[t]$. Each element in $\cU(F[t])$ is an \textit{ultra-polynomial}. Finite fields are closely related to polynomials $x^q - x$ for some power $q$ of a prime. In order to carry out the theory of ultra-finite fields in the same way that the theory of finite fields is developed, we systematically develop the theory of ultra-polynomials and introduce several analogues of notions relating to polynomials such as irreducibility, roots, separability, and so on. We also introduce an analogue of field extensions called \textit{ultra-field extensions} that provides insight into the structure of zeros of ultra-polynomials. The theory of ultra-finite fields is developed as ultra-field extensions that contain all roots of ultra-polynomials $x^{\alpha} - x$, an analogue of $x^q - x$, where $\alpha$ is a hypernatural number which is the ultraproduct of powers of primes.

In 1857, Dedekind \cite{Dedekind} stated a quadratic reciprocity law in the polynomial ring $\bF_q[t]$ over a finite field $\bF_q$. The first published proof was given in 1902 in K\"uhne \cite{kuhne} in which he also described a \textit{higher reciprocity law} in $\bF_q[t]$. Artin \cite{Artin-1924}, in 1924, gave another proof of a quadratic reciprocity law in $\bF_q[t]$ without realizing that it was already proved by K\"uhne. In 1927, Schmidt \cite{schmidt} rediscovered the higher reciprocity law of K\"uhne in $\bF_q[t]$ and also gave an elementary proof. For an early history of the development of reciprocity law for $\bF_q[t]$, see Roquette \cite{roquette-2222-LNM}.

Without realizing the literature about the reciprocity law in $\bF_q[t]$, Carlitz \cite{Carlitz-1932, Carlitz-1935} proved the general reciprocity law for $\bF_q[t]$, and thought he was the first to do so, but Ore referred him  to the work of Schmidt \cite{schmidt}. In fact, in 1934,  Ore \cite{ore} himself gave a proof of a quadratic reciprocity law in $\bF_q[t]$ using resultant. The \textit{reciprocity law by resultant} of Ore seems to inspire Clark and Pollack \cite{CP} to search for a reciprocity law by resultant in $F[t]$ for other fields $F$ among which are \textit{quasi-finite fields}. Recall (see \cite[p.190]{serre-local-fields}) that a quasi-finite field $F$ is a perfect field such that there is a Galois element $\sigma \in \Gal(F^{\alg}/F)$ for which the map $\alpha \mapsto \sigma^{\alpha}$ is an isomorphism of $\widehat{\bZ}$ onto the Galois group $\Gal(F^{\alg}/F)$, where $F^{\alg}$ denotes an algebraic closure of $F$ and $\widehat{\bZ}$ is the profinite completion of integers with respect to its subgroups of finite index. We write $(F, \sigma)$ to indicate $F$ is a quasi-finite field and $\sigma$ is a topological generator of $\Gal(F^{\alg}/F)$.

Clark and Pollack \cite{CP} defines the $n$-th power residue symbol for a quasi-finite field $(F, \sigma)$ that contains the $n$-th roots of unity as follows. Let $\mu_n \subset F^{\times}$ be the group of $n$-th roots of unity in $F$. There is a homomorphism from $F^{\times}$ to $\mu_n$ defined by sending each $a \in F^{\times}$ to $\dfrac{\sigma(a^{1/n})}{a^{1/n}}$, where $a^{1/n}$ is an $n$-th root of $a$ in $F^{\alg}$. Such a homomorphism does not depend on the choice of $a^{1/n}$, and by passage to the quotient, it induces an injective homomorphism
\begin{align}
\label{eqn-Gamma-map-introduction-in-Serre}
\Gamma_{F, n} : F^{\times}/F^{\times n} \hookrightarrow \mu_n.
\end{align}
Note that the above homomorphism is already known in Serre \cite[Lemma, p.210]{serre-local-fields}.

Take any relatively prime polynomials $\alpha, P$ in $F[t]$ such that $P$ is irreducible of degree $d$. There is a unique extension of $F$, say $F(d)$ of degree $d$ inside $F^{\alg}$. The pair $(F(d), \sigma^d)$ is also a quasi-finite field, and thus there is an injective homomorphism $\Gamma_{F(d), n} : F(d)^{\times}/F(d)^{\times n} \hookrightarrow \mu_n$ defined in a similar way as above. Identifying $F[t]/PF[t]$ with $F(d)$, the \textit{$n$-th power residue symbol} in \cite{CP} is defined as
\begin{align*}
\left(\dfrac{\alpha}{P}\right)_{n, CP} := \Gamma_{F(d), n}(\alpha \pmod{P}).
\end{align*}

For any relatively prime polynomials $\alpha, \beta \in F[t]$, write $\beta = aP_1^{r_1}\cdots P_h^{r_h}$ for the unique factorization of $\beta$ in $F[t]$, where $a \in F^{\times}$, the $P_i$ are irreducible polynomials, and the $r_i$ are positive integers. Define
\begin{align*}
\left(\dfrac{\alpha}{\beta}\right)_{n, CP} := \prod_{i = 1}^h\left(\dfrac{\alpha}{P_i}\right)^{r_i}_{n, CP}.
\end{align*}

The main result of Clark and Pollack \cite{CP} is a reciprocity law by resultant in $F[t]$.

\begin{theorem}
(Clark--Pollack \cite{CP})
\label{thm-Clark-Pollack}

Let $n$ be a positive integer, and let $F$ be a quasi-finite field that contains the group $\mu_n$ of $n$-th roots of unity. Let $\alpha, \beta$ be relatively prime polynomials in $F[t]$. Then
\begin{itemize}

\item[(i)] if $\beta$ is monic, then
\begin{align*}
\left(\dfrac{\alpha}{\beta}\right)_{n, CP} = \Gamma_{F, n}(\mathrm{Res}(\beta, \alpha)),
\end{align*}
where $\mathrm{Res}(\beta, \alpha) \in F[t]$ is the resultant of $\beta$ and $\alpha$.

\item[(ii)] If $\alpha, \beta$ are monic, then
\begin{align*}
\left(\dfrac{\alpha}{\beta}\right)_{n, CP} = \Gamma_{F, n}(-1)^{\deg(\alpha)\deg(\beta)}\left(\dfrac{\beta}{\alpha}\right)_{n, CP}.
\end{align*}

\end{itemize}

\end{theorem}

The above notion of power residue symbol introduced in Clark--Pollack \cite{CP} does not provide congruence relations between power residue symbols and powers of elements in $F[t]$ as in the case of the polynomial ring over a finite field, which makes it difficult to study certain local-to-global questions in number theory of the following form: \textit{if a certain property holds modulo all but finitely many primes in $F[t]$, does it hold in $F[t]$?} An affirmative answer to such a question is known in the case that $F$ is a finite field (see Rosen \cite[Theorem 3.7]{rosen}), due to the fact that in $\bF_q[t]$, for any relatively prime polynomials $\alpha, P$ in $\bF_q[t]$ with $P$ being irreducible of degree $d$, the $n$-th power residue symbol $\left(\dfrac{\alpha}{P}\right)_n$ is the unique element of $\bF_q^{\times}$ that satisfies the congruence relation
\begin{align}
\label{eqn1-introduction-PRS}
a^{\frac{q^d - 1}{n}} \equiv \left(\dfrac{\alpha}{P}\right)_n \pmod{P}.
\end{align}

For a general quasi-finite field $F$, for example, $F = \bC((t))$, there is no meaningful way to \textit{count} the number of elements of $F$ as in the case of finite fields. The main purpose of this paper is to generalize (\ref{eqn1-introduction-PRS}) to $F[t]$ for ultra-finite fields $F$ which are nonprincipal ultraproducts of finite fields of unbounded cardinalities and so are also quasi-finite fields. For this special class of quasi-finite fields, we propose using \textit{internal cardinality} (see Section \ref{sec-basic-notions} for a brief introduction of this notion) to \textit{count} the number of elements of an ultra-finite field $F$ which can be represented by a hypernatural number. For an ultra-finite field $F$, there is a well-defined action of hypernatural numbers on elements in $F[t]$ which we call \textit{ultra-powers}, based on the notions of ultraproducts and ultra-hulls. This action allows to define an ultra-power $a^{\frac{\kappa^d - 1}{n}}$ in a similar way as in the left-hand side of (\ref{eqn1-introduction-PRS}), where $\kappa$ is the internal cardinality of $F$ and $n$ is a positive integer dividing $\kappa - 1$ in the ultrapower of $\bZ$. Such an ultra-power $a^{\frac{\kappa^d - 1}{n}}$ must belong in the ultra-hull $\cU(F[t])$ due to the ultra-power $\frac{\kappa^d - 1}{n}$, and thus we can prove (see Section \ref{sec-power-residue-symbol}) that there exists a unique element $\left(\dfrac{\alpha}{P}\right)_n$ in $F^{\times}$ that we call the \textit{$n$-th power residue symbol} such that the congruence
\begin{align}
\label{eqn2-introduction-PRS}
a^{\frac{\kappa^d - 1}{n}} \equiv \left(\dfrac{\alpha}{P}\right)_n \pmod{P\cU(F[t])}.
\end{align}
holds in $\cU(F[t])$ instead of $F[t]$ as in the finite-field case.

Our notion of power residue symbol introduced as above recovers that of power residue symbol in Clark--Pollack \cite{CP} for the case that $F$ is an ultra-finite field, and we recover Theorem \ref{thm-Clark-Pollack} using our approach. Due to the congruence (\ref{eqn2-introduction-PRS}), we can generalize the reciprocity law of Clark--Pollack in Theorem \ref{thm-Clark-Pollack} to the general case in which both $\alpha, \beta$ are not monic polynomials (see Theorem \ref{thm-General-Reciprocity-Law}). In the case that $\alpha, \beta$ are monic, the congruence (\ref{eqn2-introduction-PRS}) allows to obtain a simpler reciprocity law in $F[t]$ (see Theorem \ref{thm-Reciprocity-Law}) that is also more analogous to the case $\bF_q[t]$ in which the reciprocity equals the ultra-power of $-1$, say $(-1)^{\frac{\kappa - 1}{n}\deg(\alpha)\deg(\beta)}$--a hypernatural number taking values either $-1$ or $1$ instead of $\Gamma_{F, n}(-1)^{\deg(\alpha)\deg(\beta)}$ as in the reciprocity law of Clark--Pollack (see Theorem \ref{thm-Clark-Pollack}).

The congruence (\ref{eqn2-introduction-PRS}) can also be used to define the $n$-th power residue symbol for any \textit{hypernatural number $n$} dividing $\kappa - 1$, the internal cardinality of $F^{\times}$, and thus for $F[t]$ with ultra-finite fields $F$, it generalizes the notion of power residue symbol in Clark--Pollack \cite{CP} that can only be defined for positive integers $n$.

In Section \ref{sec-reciprocity-law}, we prove a higher reciprocity law for $F[t]$ with $F$ being an ultraproduct of finite fields that contains the reciprocity law for $\bF_q[t]$ as a special case, and generalize to polynomial rings over infinite fields in both characteristics $0$ and $p > 0$. A special case of the higher reciprocity law proved in this paper is the following.
\begin{theorem}
\label{thm-reciprocity-law-Nguyen-in-introduction}
(see Theorem \ref{thm-Reciprocity-Law})

Let $F$ be an ultra-finite field that is a ultraproduct of finite fields $\bF_{q_s}$ over $s \in S$ with respect to a nonprincipal ultrafilter $\cD$ whose internal cardinality is a positive hyperinteger $\kappa$. Let $n$ be a hyperinteger such that $n$ divides $\kappa-1$. Let $P$ and $Q$ be monic irreducible polynomials in $F[t]$ of degrees $d$ and $\ell$, respectively. Then
\begin{align*}
\left( \frac{Q}{P} \right)_n = (-1)^{\frac{\kappa-1}{n}d\ell}\left( \frac{P}{Q} \right)_n.
\end{align*}

\end{theorem}

\begin{remark}

If $q_s = q$ for a power $q$ of some prime $p$ almost all $s \in S$, i.e., $\{s \in S \; | \; q_s = q\}$ belongs in the ultrafilter $\cD$, then $F$ must be a finite field $\bF_q$ (see Bell--Slomson \cite[Section 3, Chapter 6]{bell-slomson}). So in this case, Theorem \ref{thm-reciprocity-law-Nguyen-in-introduction} recovers the classical higher reciprocity law in $\bF_q[t]$.

If the sequence $\{q_s : s\in S\}$ is unbounded and the $q_s$ are powers of the same prime $p$, then $F$ is an infinite extension of $\bF_p$. In fact the cardinality of $F$ is $2^{\aleph_0}$, and so $F$ has the cardinality of continuum. In this case, Theorem \ref{thm-reciprocity-law-Nguyen-in-introduction} generalizes the classical higher reciprocity law for $\bF_q[t]$. In fact one can choose an appropriate ultrafilter $\cD$ so that $F$ contains the algebraic closure $\overline{\bF}_p$ of the finite field $\bF_p$.

If the $\bF_{q_s}$ have distinct characteristics, $F$ has characteristic $0$. By Lefschetz principle (see \cite{schoutens}), $\bC$ is the ultraproduct of the algebraic closures $\overline{\bF}_{q_s}$, and thus $\bQ \subseteq F \subseteq \bC$. So Theorem \ref{thm-reciprocity-law-Nguyen-in-introduction} generalizes the classical reciprocity law for $\bF_q[t]$ to polynomial rings over certain infinite fields of characteristic $0$ between $\bQ$ and $\bC$.

In Nguyen \cite{nguyen-2023-2nd-paper}, the author uses the above higher reciprocity law to establish the existential definability of $F[t]$ in its fraction field $F(t)$ which in turn proves the undecidability of the full first-order theory of $F(t)$ for any ultraproduct of finite fields of characteristic $0$. Note that it is not known whether the full first-order theory of $\bC(t)$ is decidable or not.

\end{remark}

Based on the congruence (\ref{eqn2-introduction-PRS}) and the reciprocity law \ref{thm-reciprocity-law-Nguyen-in-introduction}, we can provide an affirmative answer to a local-to-global question regarding the property of being $n$-th powers which can be viewed as an analogue of the Grunwald--Wang theorem for $F[t]$ for ultra-finite fields $F$.
\begin{theorem}
\label{thm-analogue-GW-introduction}
(see Theorem \ref{thm-local-to-global-GW})

Let $F$ be an ultra-finite field. Let $\alpha$ be a polynomial in $F[t]$ of positive degree, and let $n$ be a positive integer such that either $F$ has characteristic $0$ or the characteristic of $F$ does not divide $n$. Then $x^n = \alpha$ is solvable in $F[t]$ if and only if $x^n = \alpha$ is solvable in $P$-adic fields $F(t)_P$ for all but finitely many monic primes $P$ of $F[t]$.

\end{theorem}

The structure of the paper is as follows. In Section \ref{sec-basic-notions}, we recall some basic notions regarding ultraproducts, hypernatural numbers, and hyperintegers as well as internal cardinality. In Section \ref{sec-ultrafinite-fields}, we introduce notions of ultra-polynomials, ultra-fields, and ultra-field extensions that are analogues of polynomials, fields, and field extensions in field theory. In the same section, we develop the theory of ultra-finite fields in a similar way that the theory of finite fields is developed. In Section \ref{sec-power-residue-symbol}, we introduce a notion of $n$-th power residue symbol in $F[t]$ for ultra-finite fields $F$, where $n$ is a hypernatural number. Based on our notion of power residue symbol, we establish a general reciprocity law for $F[t]$ for ultra-finite fields $F$ in Section \ref{sec-reciprocity-law}. In Section \ref{sec-analogue-of-GW-thm}, we prove an analogue of the Grunwald--Wang theorem for the polynomial ring over an ultra-finite field.

\section{Basic notions and notation}
\label{sec-basic-notions}

In this section, we recall some basic notions and fix some notation that we will use throughout the paper. The main references for this section are Bell--Slomson \cite{bell-slomson}, Chang--Keisler \cite{CK}, Rothmaler \cite{rothmaler}, and Schoutens \cite{schoutens}.

\subsection{Ultraproducts and \L{}o\'s' theorem}

Throughout this paper, we fix an infinite set $S$. A collection $\cD$ of infinite subsets of $S$ is called a \textit{nonprincipal ultrafilter on $S$} if $\cD$ is closed under finite intersection and for any subset $A \subseteq S$, either $A$ or its complement $S \setminus A$ belong in $\cD$. Since every set in $\cD$ must be infinite, we deduce that any cofinite subset of $S$ belongs in $\cD$. One can use Zorn's lemma to show the existence of nonprincipal ultrafilters on $S$ (see \cite[Lemma 4.3.1]{rothmaler} or Tarski \cite{Tarski}). For the rest of this paper, we always fix $\cD$ to be a nonprincipal ultrafilter on $S$.

For each $s \in S$, let $A_s$ be a set. Let $\widehat{A} = \prod_{s \in S}A_s$ be the Cartesian product of the $A_s$. We define an equivalence relation on $A$ as follows: for elements $\widehat{a} = (a_s)_{s\in S}$, $\widehat{b} = (b_s)_{s \in S}$ in $\widehat{A}$, $\widehat{a}$ is equivalent to $\widehat{b}$ if the set $\{s \in S \; : \; a_s = b_s\}$ belongs in $\cD$. Following Schoutens \cite{schoutens}, we will also denote the equivalence class of an element $(a_s)_{s \in S} \in \widehat{A}$ by $\ulim_{s \in S}a_s$. We also say that $\widehat{a} = \ulim_{s \in S}a_s$ is the \textit{ultra-limit of $a_s$}.

The set of all equivalence classes on $\widehat{A}$ is called the \textit{ultraproduct of the $A_s$} with respect to the ultrafilter $\cD$, and is denoted by $\prod_{s \in S}A_s/\cD$. In this paper, only ultraproducts with respect to a nonprincipal ultrafilter will be considered, and so whenever we consider \textit{ultraproducts}, we mean \textit{nonprincipal ultraproducts}.

If all the sets $A_s$ are equal to the same set $A$, we call the ultraproduct of the $A_s$ the \textit{ultrapower of $A$}, and use $A^{\#}$ to denote the ultrapower. There is a canonical map from $A$ to $A^{\#}$ that sends each element $a \in A$ to the equivalence class $\ulim_{s \in S}a \in A^{\#}$ in which all components are equal to $a$. We identify $A$ with its image in $A^{\#}$ under the canonical map, and consider $A$ as a subset of $A^{\#}$.

Throughout this paper, whenever we say a certain property (P) holds for $\cD$-almost all $s \in S$, we mean the set $\{s \in S\; : \; \text{(P) holds in $A_s$}\}$ belongs in the ultrafilter $\cD$. For example,  $a_s = H(b_s)$ holds for $\cD$-almost all $s \in S$ for some polynomial $H \in \bZ[x]$ if and only if the set $\{s \in S \; : \; a_s = H(b_s)\}$ belongs in $\cD$.

If all the sets $A_s$ have the same algebraic structure such as groups, rings, or fields, their ultraproduct $\bA = \prod_{s \in S}A_s/\cD$ inherits the same algebraic structure from its components as follows. Let $\star_s : A_s \times A_s \to A_s$ be a binary operation for $\cD$-almost all $s \in S$. We always write $a_s \star_s b_s$ instead of $\star_s(a_s, b_s)$. The map $\star : \bA \times \bA \to \bA$ defined by
\begin{align}
\label{e1-operation-ultraproduct}
(\ulim_{s \in S}a_s)\star (\ulim_{s \in S}b_s) = \ulim_{s \in S}(a_s\star_s b_s) \in \bA
\end{align}
is a binary operation on $\bA$. Note that the operation $\star$ does not depend on the choices of $a_s$ and $b_s$ for the equivalence classes $\ulim_{s \in S}a_s$, $\ulim_{s \in S}b_s$. One can verify that if $A_s$ are groups with multiplication $\star_s$ and identity element $1_s$, then $\bA$ is a group with multiplication $\star$ and identity element $1 = \ulim_{s \in S}1_s$. Furthermore if $A_s$ is abelian for $\cD$-almost all $s \in S$, then $\bA$ is also abelian.

We obtain the following.
\begin{proposition}
\label{prop-notion-section-algebraic-structures-of-ultraproducts}

We maintain the same notation as above.

\begin{itemize}

\item [(i)] If $A_s$ is a (commutative) ring (resp. field) with respect to addition $+_s$ and multiplication $\cdot_s$ for $\cD$-almost all $s \in S$, then $\bA$ is a (commutative) ring (resp. field) with addition $+$ and multiplication $\cdot$ that are defined by (\ref{e1-operation-ultraproduct}).

\item [(ii)] If $A_s$ is a ring and $I_s$ is an ideal in $A_s$ for $\cD$-almost all $s \in S$, then $\prod_{s \in S}I_s/\cD$ is an ideal in $\bA$. Furthermore $\prod_{s \in S}(A_s/I_s)/\cD$ is isomorphic to $\bA/\left(\prod_{s \in S}I_s/\cD\right)$. The ideal $\prod_{s\in S}I_s/\cD$ is generated by $h$ elements if $A_s$ is generated by $h$ elements for $\cD$-almost all $s\in S$.

\end{itemize}

\end{proposition}

\begin{proof}

See Schoutens \cite[2.1.4, p.10]{schoutens} for the proof.

\end{proof}

We state one of the most fundamental transfer principle for ultraproducts.

\begin{theorem}
(\L{}o\'s' theorem)
\label{thm-Los}

Let $\cR$ be a ring, let $A_s$ be an $\cR$-algebra for each $s \in S$, and let $\bA$ be their ultraproduct. Let $\psi(\zeta_1, \ldots, \zeta_m)$ be a formula with parameters from $\cR$ whose free variables are among $\zeta_1, \ldots, \zeta_m$. For each $s \in S$, let $\widehat{a}_s$ be an $m$-th tuple in $A_s$, and let $\widehat{a}$ be their ultraproduct which is an $m$-th tuple in $\bA$. Then $\psi(\widehat{a}_s)$ holds in $A_s$ for $\cD$-almost all $s \in S$ if and only if $\psi(\widehat{a})$ holds in $\bA$.

\end{theorem}

\begin{proof}

For the proof of \L{}o\'s' theorem, see \cite{bell-slomson} or \cite{rothmaler}.

\end{proof}

\subsection{Hyperintegers and hypernaturals}
\label{subsec-hyperintegers}

The ring of integers is denoted by $\bZ$. The set of natural numbers, i.e. positive integers $1, 2, \ldots$ is denoted by $\bN$.

Throughout this paper, we let $\bZ^{\#}$ be the ultrapower $\prod_{s \in S}\bZ/\cD$ whose elements are called \textbf{hyperintegers}. We also let $\bN^{\#}$ be the ultrapower $\prod_{s \in S}\bN/\cD$ whose elements are called \textbf{hypernatural numbers} or simply \textbf{hypernaturals}. Since $\bN \subset \bZ$, $\bN^{\#} \subset \bZ^{\#}$.

\subsubsection{} There is a total order on $\bZ^{\#}$. For any hyperintegers $a = \ulim_{s \in S}a_s$, $b = \ulim_{s \in S}b_s$, we define $a \le b$ if and only if $a_s \le b_s$ for $\cD$-almost all $s \in S$, i.e., $\{s \in S \; : \; a_s \le b_s\} \in \cD$. We further write $a < b$ if $a \le b$ and $a \ne b$.

Since $\cD$ is a nonprincipal ultrafilter, it is not difficult to check that ``$\le$" is a total order on $\bZ^{\#}$. Indeed, take any hyperintegers $a = \ulim_{s \in S}a_s$ and $b = \ulim_{s \in S}b_s$ for some integers $a_s, b_s \in \bZ$. Set
\begin{align*}
A &= \{s \in S \; : \; a_s \le b_s \},\\
B &= \{s \in S \; : \; a_s > b_s \}.
\end{align*}
Note that $S = A \cup B$, and for any $s \in S$, either $a_s \le b_s$ or $a_s > b_s$, and thus $s$ belongs in exactly one of $A$ and $B$. Since $\cD$ is a nonprincipal ultrafilter, either $A \in \cD$ or $B = S\setminus A \in \cD$, which implies that either $a \le b$ or $a > b$. Thus ``$\le$" is a total order on $\bZ^{\#}$.

A hyperinteger $a = \ulim_{s \in S}a_s$ with integers $a_s$ is \textbf{positive} if $a > 0$, i.e., $a_s > 0$ for $\cD$-almost all $s \in S$. A positive hyperinteger is also a hypernatural number. A hyperinteger $a$ is \textbf{negative} if $a < 0$. We further say that $a$ is \textbf{nonnegative} if $a \ge 0$.

The set of positive (resp., negative and nonnegative) hyperintegers is denoted by $\bZ^{\#}_{>0}$ (resp. $\bZ^{\#}_{< 0}$, and $\bZ^{\#}_{\ge 0}$).

\subsubsection{} We introduce the arithmetic of hyperintegers.

Let $a = \ulim_{s \in S}a_s$, $b = \ulim_{s \in S}b_s$ be hyperintegers for some integers $a_s, b_s \in \bZ$. We say that \textbf{$a$ divides $b$} (or \textbf{$a$ is a divisor of $b$}) if $b = ac$ for some hyperinteger $c = \ulim_{s \in S}c_s$ with $c_s \in \bZ$, i.e., $b_s = a_sc_s$ for $\cD$-almost all $s \in S$.

A hyperinteger $c = \ulim_{s \in S}c_s$ with integers $c_s$ is a \textbf{common divisor of $a, b$} if $c$ divides both $a$ and $b$.

The \textbf{greatest common divisor of $a, b$}, denoted by $\gcd(a, b)$, is a positive hyperinteger $d = \ulim_{s \in S}d_s$ with $d_s \in \bZ$ such that $d$ divides both $a$ and $b$ and if a hyperinteger $c = \ulim_{s \in S}c_s$ with integers $c_s$ is a common divisor of both $a$ and $b$, then $c \le d$. The following result proves the uniqueness of the greatest common divisor.

\begin{lemma}
\label{lem-gcd-of-hyperintegers}

Let $a = \ulim_{s \in S}a_s$, $b = \ulim_{s \in S}b_s$ be hyperintegers for some integers $a_s, b_s$ such that at least one of $a, b$ is nonzero. Then $\gcd(a, b) = \ulim_{s \in S}\gcd(a_s, b_s)$.

\end{lemma}

\begin{proof}

Let $d_s = \gcd(a_s, b_s)$ for $s \in S$, and $d = \ulim_{s \in S}d_s \in \bZ^{\#}$. We prove that $d$ is the greatest common divisor of $a, b$. Indeed, since $d_s$ divides both $a_s, b_s$ for $\cD$-almost all $s \in S$, $d$ is a common divisor of $a, b$. Since $d_s > 0$ for $\cD$-almost all $s \in S$, $d > 0$.

Let $c = \ulim_{s \in S}c_s$ be a hyperinteger for some integers $c_s$ such that $c$ divides both $a, b$. Then $c_s$ is a common divisor of both $a_s$ and $b_s$ for $\cD$-almost all $s \in S$. Thus $c_s \le d_s$ for $\cD$-almost all $s \in S$, and therefore $c \le d$. Thus $d$ is the greatest common divisor of $a$ and $b$, and thus the lemma follows immediately.

\end{proof}

We prove a nonstandard B\'ezout's identity.

\begin{theorem}
\label{thm-Bezout}

Let $a = \ulim_{s \in S}a_s$ and $b = \ulim_{s \in S}b_s$ with integers $a_s, b_s$ such that at least one of $a, b$ is nonzero. Then there exist hyperintegers $e, f \in \bZ^{\#}$ such that
\begin{align*}
\gcd(a, b) = ae + bf.
\end{align*}

\end{theorem}

\begin{proof}

By a standard B\'ezout's identity, there exist integers $e_s, f_s$ such that $\gcd(a_s, b_s) = a_se_s + b_sf_s$. Thus it follows from Lemma \ref{lem-gcd-of-hyperintegers} that
\begin{align*}
\gcd(a, b) = \ulim_{s \in S}\gcd(a_s, b_s) = ae + bf,
\end{align*}
where $e = \ulim_{s \in S}e_s$ and $f = \ulim_{s \in S}f_s$ are in $\bZ^{\#}$.

\end{proof}

\subsection{Internal sets and cardinality}
\label{subsec-internal-sets-and-cardinality}

In this subsection, we recall the notion of internal cardinality that plays a key role in this paper.

For a finite set $X$, we denote by $\card(X)$ the number of elememts in $X$.

\begin{definition}
\label{def-internal-cardinality}

Let $\alpha = \ulim_{s \in S}\alpha_s$ be a hypernatural in $\bN^{\#}$, where the $\alpha_s$ are natural numbers.

A set $A$ is said to have \textbf{internal cardinality $\alpha$} if there exists a collection of finite sets $A_s$ of cardinality $\alpha_s$ for $\cD$-almost all $s \in S$ such that $A = \prod_{s\in S} A_s/\cD$, i.e., $A$ is the ultraproduct of the sets $A_s$ with respect to $\cD$, and for $\cD$-almost $s \in S$, $A_s$ is a finite set of $\alpha_s$ elements. We also call $A$ an \textbf{internal set}.

In symbol, we write $\icard(A)$ for the internal cardinality of $A$. From the definition of internal cardinality, $\icard(A) = \ulim_{s \in S}\card(A_s)$.

\end{definition}

\begin{theorem}
\label{thm-internal-cardinality}
(See \cite[Theorem 12.5.1]{Goldblatt})

Let $A$ be an internal set whose internal cardinality is $n \in \bN^{\#}$. Then there is an internal bijection $f : \{1, \ldots, n\} \to A$, where $\{1, \ldots, n\} = \{m \in \bN^{\#} \; : \; m \le n\}$.

\end{theorem}

For internal sets, there is an analogue of least number principle as follows.

\begin{theorem}
\label{thm-internal-least-number-principle}

Any nonempty internal set in $\bN^{\#}$ has a least member.

\end{theorem}

For the proof of Theorem \ref{thm-internal-least-number-principle}, see Goldblatt \cite[Theorem 11.3.1]{Goldblatt}.

The following is straightforward from the notions of internal ses and cardinality.
\begin{proposition}
\label{prop-comparing-internal-sets-with-same-cardinality}

Let $A = \prod_{s \in S}A_s/\cD$, $B = \prod_{s \in S}B_s/\cD$ be internal sets such that the $A_s, B_s$ are finite sets for $\cD$-almost all $s \in S$, $A \subseteq B$ and $\icard(A) = \icard(B)$. Then $A = B$.

\end{proposition}

\subsection{The action of hyperintegers on ultraproducts}
\label{subsec-action-of-hyperintegers-on-ultraproducts}

Let $A = \prod_{s \in S}A_s/\cD$ be a (nonprincipal) ultraproduct of rings $A_s$, and let $n = \ulim_{s \in S}n_s$ be a nonnegative hyperinteger for some integers $n_s \ge 0$. The \textbf{ultrapower map on $A$ with exponent $n$} is defined as follows. For an element $a = \ulim_{s \in S}a_s \in A$ for some $a_s \in A_s$, define
\begin{align}
\label{e1-ultra-power-with-exponent-in-hypernaturals}
a^n := \ulim_{s \in S}a_s^{n_s} \in A.
\end{align}
It is easy to check that the above definition does not depend on the choice of $a_s$ for the equivalence class $a \in A$. By \L{}o\'s' theorem \ref{thm-Los}, we can show that the ultrapower map satisfies the same properties as the standard power map as follows. For any elements $a = \ulim_{s \in S}a_s, b = \ulim_{s \in S}b_s \in A$ and any nonnegative hyperintegers $m = \ulim_{s \in S}m_s, n = \ulim_{s \in S}n_s$,
\begin{itemize}

\item [(i)] $(ab)^n = a^n \cdot b^n$;

\item [(ii)] $a^n\cdot a^m = a^{m + n}$; and

\item [(iii)] $(a^n)^m = a^{mn}$.

\end{itemize}

We are most interested in the case where $A$ is the ultra-hull $\cU(F[t])$ of the polynomial ring $F[t]$, where $F = \prod_{s \in S}F_s/\cD$ is an ultraproduct of fields $F_s$. That is, $A = \cU(F[t]) = \prod_{s \in S}F_s[t]/\cD$. For any hypernatural $n = \ulim_{s \in S}n_s$ for some natural numbers $n_s$ such that $n \not\in \bN$, the element $t^n = \ulim_{s \in S}t^{n_s}$ belongs in the ultra-hull $\cU(F[t])$ but does not belong in $F[t]$. Such an element is an example of the so-called \textit{ultra-polynomial of ultra-degree $n$} that is not a polynomial will be defined in Section \ref{sec-ultrafinite-fields}.

If $F$ is an ultraproduct of fields $F_s$, we can also define the \textbf{ultrapower map on $F^{\times}$ with exponent as a hyperinteger $\alpha = \ulim_{s\in S}n_s \in \bZ^{\#}$} in the same way as (\ref{e1-ultra-power-with-exponent-in-hypernaturals}) above. That is, for a hyperinteger $\alpha = \ulim_{s \in S}n_s \in \bZ^{\#}$ for some integers $n_s \in \bZ$ and an arbitrary element $a = \ulim_{s \in S}a_s \in F^{\times}$ for some elements $a_s \in F_s^{\times}$, define
\begin{align}
\label{e2-ultrapower-with-hyperinteger-exponent}
a^{\alpha} := \ulim_{s \in S}a_s^{n_s}.
\end{align}
Since $a \in F^{\times}$, $a_s \ne 0$ for $\cD$-almost all $s \in S$, and thus the above definition is well-defined. Thus we obtain the action of $\bZ^{\#}$ on the multiplicative group $F^{\times}$ of the form
\begin{align*}
\bZ^{\#} \times F^{\times} \to F^{\times}
\end{align*}
that sends each pair $(\alpha, a)$ to $a^{\alpha}$ defined as in (\ref{e2-ultrapower-with-hyperinteger-exponent}).

For an arbitrary ultraproduct $R = \prod_{s \in S}R_s/\cD$ of rings $R_s$, there is a natural group action of $\bZ^{\#}$ on $R$ that is defined as follows. For each hyperinteger $n = \ulim_{s \in S}n_s$ for some integers $n_s$ and any element $r = \ulim_{s\in S}r_s$ in $R$ for some elements $r_s \in R_s$, we set $n \star r = \ulim_{s\in S}n_sr_s \in R$. It is straightforward that ``$\star$" defines a group action of $\bZ^{\#}$ on $R$. For simplicity, we only write $n r$ for $n \star r$ whenever the action $\star$ is clear in the context. Thus every ultraproduct of rings is an $\bZ^{\#}$-module.

\section{Ultrafinite fields}
\label{sec-ultrafinite-fields}

In this section, we develop the theory of ultra-finite fields in detail that is an analogue of that of finite fields. We begin by studying ultra-polynomials over ultra-fields.

\subsection{Roots of Irreducible ultra-polynomials and ultra-algebraic extensions}
\label{subsec-roots-of-irreducible-ultra-polynomials}

In this subsection, we study roots of irreducible ultra-polynomials over ultra-fields and ultra-algebraic extensions which are generalizations of roots of irreducible polynomials and algebraic extensions in field theory. We begin by defining the notions of \textit{ultra-fields} and \textit{ultra-hulls} that were mentioned in the introduction.

\begin{definition}
\label{def-ultra-fields}

An ultraproduct  of fields $F = \prod_{s \in S}F_s$ for some fields $F_s$ is called an \textbf{ultra-field}. The \textbf{ultra-hull} of the polynomial ring $F[t]$ is the ultraproduct $\prod_{s \in S}F_s[t]/\cD$ of polynomial rings $F_s[t]$. We denote by $\cU(F[t])$ the ultra-hull of $F[t]$.

\end{definition}

By Proposition \ref{prop-notion-section-algebraic-structures-of-ultraproducts}, every ultra-field is a field.

\begin{definition}
\label{def-ultra-polynomials}
(ultra-polynomials)

Let $F = \prod_{s \in S}F_s/\cD$ be an ultra-field, where the $F_s$ are fields, and let $\cU(F[x])$ be the ultra-hull of the polynomial ring $F[x]$. An element in $\cU(F[x])$ is called an \textbf{ultra-polynomial over $F$}.

The \textbf{ultra-degree} of an ultra-polynomial $f = \ulim_{s \in S}f_s$, where the $f_s$ are polynomials in $F_s[x]$ of degrees $d_s \in \bZ_{\ge 0}$, is the hyperinteger $d = \ulim_{s \in S}d_s \in \bZ^{\#}_{\ge 0}$.

The \textbf{leading coefficient} of an ultra-polynomial $f = \ulim_{s \in S}f_s$, where the $f_s$ are polynomials in $F_s[x]$ of leading coefficients $a_s \in F_s$ is the element $a = \ulim_{s \in S} a_s \in F$. When the $a_s$ equals $1$ for $\cD$-almost $s\in S$, the leading coefficient $a = 1$, and the ultra-polynomial $f$ is called \textbf{monic}.

\end{definition}

\begin{definition}
\label{def-irreducible-ultra-poly}
(irreducible ultra-polynomials)

Let $F = \prod_{s \in S}F_s/\cD$ be an ultra-field, where the $F_s$ are fields. An ultra-polynomial $f = \ulim_{s \in S}f_s$ over an ultra-field $F$, where the $f_s$ are polynomials in $F_s[x]$, is \textbf{irreducible over $F$} if the $f_s$ are irreducible over $F_s$ for $\cD$-almost all $s \in S$.

\end{definition}

\begin{definition}
\label{def-extension-of-ultra-fields}
(ultra-field extensions)

Let $F = \prod_{s \in S}F_s/\cD$ be an ultra-field. An \textbf{ultra-field extension} (or \textbf{ultra-extension $L$ of $F$}) is an ultra-field $L = \prod_{s \in S}L_s/\cD$, where the $L_s$ are field extensions of the $F_s$ for $\cD$-almost all $s \in S$.

We also say that $F$ is an \textbf{ultra-subfield of $L$}.

\end{definition}

\begin{definition}
\label{def-the-values-of-ultra-poly}

Let $F = \prod_{s \in S}F_s/\cD$ be an ultra-field, where the $F_s$ are fields. Let $L = \prod_{s \in S}L_s$ be an ultra-extension of $F$. Let $f = \ulim_{s \in S}f_s$ be an ultra-polynomal over $F$, where the $f_s$ are polynomials in $F_s[x]$. Let $a = \ulim_{s \in S}a_s \in L$ for $a_s \in L_s$. The \textbf{value of $f$ at $a$} is defined to be the element $\ulim_{s \in S}f_s(a_s) \in L$. In symbols, we write $f(a)$ for the value of $f$ at $a$.

\end{definition}

\begin{remark}

The value $f(a)$ does not depend on the representation of $a$ in the ultraproduct $L = \prod_{s \in S}L_s/\cD$.

\end{remark}

\begin{definition}
\label{def-roots-of-ultra-polynomials}
(roots of ultra-polynomials)

Let $F = \prod_{s \in S}F_s/\cD$ be an ultra-field, where the $F_s$ are fields, and let $f = \ulim_{s \in S}f_s$ be an ultra-polynomial over $F$. An element $a = \ulim_{s \in S}a_s$ in an ultra-field extension $L = \prod_{s \in S}L_s/\cD$ of $F$ with $a_s \in L_s$, where the $L_s$ are field extensions of the $F_s$, is called a \textbf{root of $f$} if $a_s$ is a root of $f_s$ for $\cD$-almost all $s \in S$.

Equivalently, $a$ is a root of $f$ if and only if $f(a) = 0$.

\end{definition}

\begin{lemma}
\label{lem-x-c-divides-f(x)-if-c-is-a-root-of-f(x)}

Let $F = \prod_{s \in S}F_s/\cD$ be an ultra-field, where the $F_s$ are fields, and let $f = \ulim_{s \in S}f_s$ be an ultra-polynomial over $F$. Let $L = \prod_{s\in S}L_s/\cD$ be an ultra-field extension of $F$ such that $L_s$ is a field extension of $F_s$ for $\cD$-almost all $s \in S$. For any element $a = \ulim_{s \in S}a_s \in L$,
\begin{align*}
f \equiv f(a) \pmod{(x - a)\cU(L[x])},
\end{align*}
that is, there exists an ultra-polynomial $Q \in \cU(L[x])$ over $L$ such that $f - f(a) = (x - a)Q$. Here $\cU(\L[x]) = \prod_{s \in S}L_s[x]/\cD$ is the ultra-hull of $L[x]$.

\end{lemma}

\begin{proof}

For each $s \in S$, we know from the Euclidean algorithm for the polynomial ring $L_s[x]$ that $f_s(x) - f_s(a_s) = (x - a_s)Q_s(x)$ for some polynomial $Q_s \in L_s[x]$. Thus it follows from \L{}o\'s' theorem that
\begin{align*}
f - f(a) &= \ulim_{s \in S}(f_s(x) - f_s(a_s))\\
 &= \ulim_{s \in S}(x - a_s)Q_s(x) \\
 &= \ulim_{s\in S}(x - a_s)\ulim_{s \in S}Q_s(x) \\
 &= (x - a)Q,
 \end{align*}
 where $Q = \ulim_{s \in S}Q_s(x) \in \cU(L[x])$. Thus the lemma follows immediately.

\end{proof}

The above two notions generalizes those of values and roots of polynomials, which can be verified as follows. Let $f(x)$ be a polynomial of degree $d \in \bZ_{> 0}$ in $F[x]$, where $F = \prod_{s \in S}F_s/\cD$ is an ultra-field for some fields $F_s$. Write
\begin{align*}
f(x) = a_dx^d + a_{d - 1}x^{d - 1} + \cdots + a_1x + a_0,
\end{align*}
where the $a_i$ are elements in $F$ and $a_d \ne 0$. For each $0 \le i \le d$, write $a_i = \ulim_{s \in S}a_{i, s}$, where the $a_{i, s}$ are elements in $F_s$. Since $a_d \ne 0$, it follows from \L{}o\'s' theorem that $a_{d, s} \ne 0$ for $\cD$-almost all $s \in S$. For each $s \in S$, set
\begin{align}
\label{eqn1-in-lem-roots-of-polynomials-over-ultrafields}
f_s(x) = a_{d, s}x^d + \cdots + a_{1, s}x + a_{0, s} \in F_s[x].
\end{align}

One can verify that $f = \ulim_{s \in S}f_s$. For an arbitrary element $c = \ulim_{s \in S}c_s \in F$ with $c_s \in F_s$, we see that
\begin{align*}
f(c) &= a_dc^d + a_{d - 1}c^{d - 1} + \cdots + a_1c + a_0 \\
&= \ulim_{s \in S}a_{d, s}\ulim_{s \in S}c_s^d  + \cdots + \ulim_{s \in S}a_{1, s}\ulim_{s \in S}c_s + \ulim{s \in S}a_{0, s} \\
&= \ulim_{s\in S}\left(a_{d, s}c_s^d + \cdots + a_{1, s}c_s + a_{0, s}\right),
\end{align*}
and thus
\begin{align}
\label{eqn2-roots-of-polynomials-over-ultrafields}
f(c) = \ulim_{s \in S}f_s(c_s),
\end{align}
which motivates Definition \ref{def-the-values-of-ultra-poly}.

The following result motivates Definition \ref{def-roots-of-ultra-polynomials}.

\begin{lemma}
\label{lem-roots-of-polynomials-over-ultrafields}

Let $F= \prod_{s \in S}F_s/\cD$ be an ultra-field, where the $F_s$ are fields. Let $f(x) = a_dx^d + a_{d - 1}x^{d - 1} + \cdots + a_1x + a_0$ be a polynomial of degree $d > 0$ in $F[x]$ such that $f$ is irreducible and separable over $F$, where the $a_i$ are elements in $F$ and $a_d \ne 0$. Write $f = \ulim_{s \in S}f_s$, where the $f_s$ are defined by (\ref{eqn1-in-lem-roots-of-polynomials-over-ultrafields}). For each $s \in S$, let $A_s$ denote the set of roots of $f_s$ in an algebraic closure of $F_s$, and let $A$ denote the set of roots of $f$ in an algebraic closure of $F$. Then $A = \prod_{s \in S}A_s/\cD$. More explicitly, every root of $f$ is the ultra-limit of certain roots of $f_s$.

\end{lemma}

\begin{proof}

Since being irreducible can be defined by a first-order formula in the language of rings, and $f(x)$ is irreducible over $F$, it follows from \L{}o\'s' theorem that $f_s(x)$ is irreducible over $F_s$ for $\cD$-almost all $s \in S$. We claim that $f_s$ is separable for $\cD$-almost all $s \in S$.

Suppose that $F$ has characteristic $0$. By \L{}o\'s' theorem, the $F_s$ have distinct characteristics $p_s$ for $\cD$-almost all $s \in S$, where $p_s$ is either a prime or zero. If $F_s$ has characteristic $0$, then $f_s$ is separable since it is irreducible over $F_s$. Note that for each $0 \le i \le d$, $a_i \ne 0$ if and only if $a_{i, s} \ne 0$ for $\cD$-almost all $s \in S$. Thus in setting $I :=\{i \; | \; \text{$0 \le i \le d$ and $a_i \ne 0$}\}$, we deduce that
\begin{align*}
I = I_s := \{i \; | \;  \text{$0 \le i \le d$ and $a_{i, s} \ne 0$}\}
\end{align*}
for $\cD$-almost all $s \in S$. In particular, since $I \ne \emptyset$, $I_s \ne \emptyset$ for $\cD$-almost all $s \in S$. Since there are only finitely many primes $q$ that divide every positive integer in $I$, and thus there are only finitely many primes $q$ that divide every positive integer in $I_s$ for $\cD$-almost all $s \in S$. Hence $f_s$ must be separable over $F_s$ for $\cD$-almost all $s \in S$.

 Suppose now that $F$ has a positive characteristic $p$. By \L{}o\'s' theorem, the $F_s$ also have positive characteristic $p$ for $\cD$-almost all $s \in S$. By assumption, we know that $f$ is separable, and thus at least one of the positive integers in $I$, say $j$ is not divisible by $p$. We also know that $a_j \ne 0$, and thus $a_{j, s} \ne 0$ for $\cD$-almost all $s \in S$. Thus there exists at least one nonzero term $a_{j, s}x^j$ of the polynomial $f_s(x)$ such that $p$ does not divide $j$ for $\cD$-almost all $s \in S$, and therefore $f_s$ is separable for $\cD$-almost all $s \in S$.

 We now prove that $\prod_{s \in S}A_s/\cD \subseteq A$. Indeed, let $c = \ulim_{s \in S}c_s$ be an element in $\prod_{s \in S}A_s/\cD$, where $c_s$ is a root of $f_s$ in an algebraic closure of $F_s$. Using the same arguments preceding (\ref{eqn2-roots-of-polynomials-over-ultrafields}), we see that
 \begin{align*}
 f(c) = \ulim_{s \in S}f_s(c_s) = 0,
 \end{align*}
 which proves that $c$ is a root of $f$, and thus $\prod_{s \in S}A_s/\cD \subseteq A$.

 We know that $\card(A_s) = d$ for $\cD$-almost all $s\in S$ since the $f_s$ are irreducible and separable of degree $d$ for $\cD$-almost all $s \in S$. Thus the internal cardinality of $\prod_{s \in S}A_s/\cD$ is $\ulim_{s\in S}\card(A_s) = \ulim_{s\in S}d = d$. By Theorem \ref{thm-internal-cardinality}, there exists a bijection between the finite set $\{1, 2, \ldots, d\}$ and $\prod_{s \in S}A_s/\cD$, which implies that $\prod_{s \in S}A_s/\cD$ is a finite set of exactly $d$ elements. Since $A$ is a finite set consisting of exactly $d$ roots of $f$ and it contains $\prod_{s \in S}A_s/\cD$ as its subset, we deduce that $\prod_{s \in S}A_s/\cD = A$ as required.

\end{proof}

\begin{definition}
(ultra-algebraic extensions)
\label{def-ultra-algebraic}

Let $F = \prod_{s \in S}F_s/\cD$ be an ultraproduct of fields $F_s$. An element $a$ in an ultra-field extension of $F$ is called \textbf{ultra-algebraic over $F$} if there exists an ultra-polynomial $f$ over $F$ such that $f(a) = 0$.

An ultra-field extension $L = \prod_{s \in S}L_s/\cD$ of $F$ is called \textbf{ultra-algebraic over $F$} if every element of $L$ is ultra-algebraic over $F$.

If the degree $[L_s : F_s]$ as a vector space over $F_s$ is finite for $\cD$-almost all $s \in S$, we define the \textbf{ultra-degree of $L$ over $F$} to be $\ulim_{s \in S}[L_s : F_s] \in \bN^{\#}$. In symbol, we write $[L : F]_u$ for the ultra-degree of $L$ over $F$.

\end{definition}

\begin{definition}
(ultra-algebraic closure)
\label{def-ultra-algebraic-closure}

An \textbf{ultra-algebraic closure of an ultra-field $F = \prod_{s \in S}F_s/\cD$} is an ultra-field extension of $F$ of the form $\prod_{s \in S}F_s^{\alg}/\cD$, where $F_s^{\alg}$ denotes an algebraic closure of $F_s$ for $\cD$-almost all $s \in S$.

We denote by $F^{\alg_u}$ an ultra-algebraic closure of $F$.

\end{definition}

\begin{remark}

For each positive integer $n$, we define a sentence in the language of rings of the form $\forall a_0 a_1\cdots a_{n - 1} \exists x (x^n + a_{n - 1}x^{n - 1} + \cdots + a_1x + a_0 = 0)$. The collection of such sentences with $n$ ranging over $\bZ_{>0}$ axiomatizes the algebraically closed property of a field. Thus by \L{}o\'s' theorem, an ultra-algebraic closure of $F$ is an algebraically closed field.

\end{remark}

\begin{lemma}
\label{lem-ultra-polynomials-have-roots-in-ultra-algebraic-closure}

For every ultra-polynomial $f = \ulim_{s \in S}f_s$ over an ultra-field $F = \prod_{s \in S}F_s/\cD$, all the roots of $f$ belong in an ultra-algebraic closure of $F$.

\end{lemma}

\begin{proof}

For all $s \in S$, all the roots of $f_s$ belong in $F_s^{\alg}$. Since an ultra-algebraic closure of $F$ is $\prod_{s\in S}F_s^{\alg}/\cD$, it follows that all roots of $f$ belong in the ultra-algebraic closure of $F$.

\end{proof}

\begin{definition}
\label{def-simple-roots-of-ultra-polynomials and simple roots}
(separable ultra-polynomials and simple roots)

An ultra-polynomial $f = \ulim_{s \in S}f_s$ over an ultra-field $F = \prod_{s \in S}F_s/\cD$ for some polynomials $f_s \in F_s[x]$ is said to be \textbf{separable over $F$} if the $f_s$ are separable over $F_s$ for $\cD$-almost all $s \in S$.

A root $c = \ulim_{s \in S}c_s$ of $f$ in an ultra-field extension of $F$ is said to be \textbf{simple} if $c_s$ is a simple root of $f_s$ for $\cD$-almost all $s \in S$.

\end{definition}

Using the same arguments as in Lemma \ref{lem-roots-of-polynomials-over-ultrafields}, we obtain the following.
\begin{theorem}
\label{thm-roots-of-ultra-polynomials-over-ultra-fields}

Let $F= \prod_{s \in S}F_s/\cD$ be an ultra-field, where the $F_s$ are fields. Let $f(x) = \ulim_{s \in S}f_s(x)$ be an ultra-polynomial of ultra-degree $d = \ulim_{s \in S}d_s$ over $F$ such that $f$ is irreducible and separable over $F$, where the $f_s$ are polynomials of degrees $d_s$ in $F_s[x]$ for $\cD$-almost all $s \in S$. For each $s \in S$, let $A_s$ denote the set of roots of $f_s$ in an algebraic closure $F_s^{\alg}$ of $F_s$, and let $A$ denote the set of roots of $f$ in an ultra-algebraic closure $F^{\alg_u} = \prod_{s \in S}F_s^{\alg}/\cD$ of $F$. Then
\begin{itemize}

\item [(i)] $A = \prod_{s \in S}A_s/\cD$.

\item [(ii)] $\icard(A) = \ulim_{s \in S}\card(A_s) = d$.

\end{itemize}

\end{theorem}

\begin{proof}

Using the same arguments as in Lemma \ref{lem-roots-of-polynomials-over-ultrafields}, Part (i) follows immediately. For part (ii), we know that $f_s$ is irreducible and separable for $\cD$-almost all $s \in S$, and thus $\card(A_s) = d_s$ for $\cD$-almost all $s\in S$. Therefore $\icard(A) = \ulim_{s \in S}\card(A_s) = \ulim_{s\in S}d_s = d$ as required.

\end{proof}

Without the irreducibility and separability assumption for polynomials $f_s$, we can obtain an upper bound in $\bN^{\#}$ of the zero set of an ultra-polynomial.

\begin{corollary}
\label{cor-upper-bound-for-zero-set-of-ultra-polynomial}

Let $F= \prod_{s \in S}F_s/\cD$ be an ultra-field, where the $F_s$ are fields. Let $f(x) = \ulim_{s \in S}$ be an ultra-polynomial of ultra-degree $d = \ulim_{s \in S}d_s$ over $F$, where the $f_s$ are polynomials of degrees $d_s$ in $F_s[x]$ for $\cD$-almost all $s \in S$. For each $s \in S$, let $A_s$ denote the set of roots of $f_s$ in an algebraic closure $F_s^{\alg}$ of $F_s$, and let $A$ denote the set of roots of $f$ in an ultra-algebraic closure $F^{\alg_u} = \prod_{s \in S}F_s^{\alg}/\cD$ of $F$. Then
\begin{itemize}

\item [(i)] $A = \prod_{s \in S}A_s/\cD$.

\item [(ii)] $\icard(A) = \ulim_{s \in S}\card(A_s) \le d$.

\end{itemize}

\end{corollary}

\begin{proof}

Part (i) follows immediately from Theorem \ref{thm-roots-of-ultra-polynomials-over-ultra-fields}. For part (ii), $\card(A_s) \le d_s$, and thus $\icard(A) = \ulim_{s \in S}\card(A_s) \le \ulim_{s \in S}d_s = d$.

\end{proof}

Suppose that an element $a = \ulim_{s \in S}a_s$ in an ultra-field extension of $F$ is ultra-algebraic over $F$. For $\cD$-almost all $s \in S$, it is known (see \cite{lang-algebra}) that the ideal $I_s = \{g \in F_s[x] \; | \; g(a_s) = 0\}$ is a principal ideal that is generated by some monic irreducible polynomial $g_s \in F_s[x]$ of least degree having $a_s$ as a root.

Set
$$g = \ulim_{s \in S}g_s \in \cU(F[x]),$$
and let
$$I = \{g \in \cU(F[x]) \; |\; g(a) = 0\}.$$

It is easily verified that $I$ is an ideal of the ultra-hull $\cU(F[x])$. We contend that $I$ is equal to the principal ideal $(g)$. Indeed, for any ultra-polynomial $h = \ulim_{s \in S}h_s$ over $F$ with $h(a) = 0$, $h_s(a_s) = 0$ for $\cD$-almost all $s \in S$. Thus $h_s \in I_s$, and therefore $h_s = g_sj_s$ for some polynomial $j_s \in F_s[x]$. We then see that $h = gj$, where $j = \ulim_{s \in S}j_s$ is an ultra-polynomial over $F$. Thus $I \subseteq (g)$. It is clear that $(g) \subseteq I$, and thus $I = (g)$.

By the definition of $g$, $g$ is irreducible over $F$ and monic. It is also easy to verify that $g$ is the monic ultra-polynomial over $F$ of least ultra-degree having $a$ as a root, i.e., the ultra-degree of $g$ is the least hypernatural among all the ultra-degrees of ultra-polynomials $h$ having $a$ as a root. In analogy with the classical notion of minimal polynomials in field theory, we call $g$ the \textbf{minimal ultra-polynomial of $a$ over $F$.} By the \textbf{ultra-degree of $a$ over $F$}, we mean the ultra-degree of $g$. In summary, we have proved the following.
\begin{theorem}
\label{thm-minimal-ultra-polynomials}

If an element $a = \ulim_{s \in S}a_s$ in an ultra-field extension of $F = \prod_{s \in S}F_s/\cD$ is ultra-algebraic over $F$, then its minimal ultra-polynomial $g$ over $F$ has the following properties:
\begin{itemize}

\item [(i)] $g$ is irreducible in $\cU(F[x])$ that can be represented as $g = \ulim_{s \in S}g_s$, where the $g_s(x) \in F_s[x]$ are the minimal polynomials of the $a_s$ for $\cD$-almost all $s \in S$.

\item [(ii)] For any ultra-polynomial $f \in \cU(F[x])$, $f(a) = 0$ if and only if $g$ divides $f$ in $\cU(F[x])$.

\item [(iii)] $g$ is the monic ultra-polynomial in $\cU(F[x])$ of least ultra-degree having $a$ as a root.

\end{itemize}

\end{theorem}

\begin{remark}

''Least ultra-degree" in Part (iii) in the above theorem means for any nonzero ultra-polynomial $f$ over $F$ such that $f(a) = 0$, the ultra-degree of $f$ is at least the ultra-degree of $g$, where the ordering here takes place in $\bN^{\#}$ (see Theorem \ref{thm-internal-least-number-principle}.)

\end{remark}

\begin{lemma}
\label{lem-divisibility-of-minimal-ultra-poly}

Let $f$ be an irreducible ultra-polynomial over an ultra-field $F = \prod_{s \in  S} F_{s}/\cD$, where the $F_{s}$ are fields for $\cD$-almost all $s \in S$.

Let $\alpha$ be a root of $f$ in an ultra-field extension of $F$. Then for an ultra-polynomial $h$ over $F$, $h(\alpha)=0$ if and only if $f$ divides $h$ in the ultra-hull $\cU(F[x])$.

\end{lemma}

\begin{proof}

Write $f = \ulim_{s \in S}f_s$, where $f_s \in F_{s}[x]$ are irreducible in $F_{s}[x]$ for $\cD$-almost $s \in S$.

Let $a$ be the leading coefficient of $f$, i.e., $a = \ulim_{s \in S}a_s \in F$, where $a_s \in F_s$ is the leading coefficient of $f_s$, and let $g= a^{-1}f$ so that $g$ is a monic irreducible ultra-polynomial with $g(\alpha)=0$. Thus $g$ is the minimal ultra-polynomial  of $\alpha$ over $F$, and therefore Lemma \ref{lem-divisibility-of-minimal-ultra-poly} follows immediately from Theorem \ref{thm-minimal-ultra-polynomials}.

\end{proof}

\begin{lemma}
\label{lem-multiplicative-of-ultra-degrees}

Let $F = \prod_{s \in S}F_s/\cD$, $L = \prod_{s\in S}L_s/\cD$, and $H = \prod_{s \in S}H_s/\cD$ be ultra-fields such that $F \subseteq L \subseteq H$. Suppose that $[H_s : F_s] < \infty$ for $\cD$-almost all $s \in S$. Then
\begin{align*}
[H : F]_u = [H : L]_u[L:F]_u
\end{align*}
holds in $\bN^{\#}$.

\end{lemma}

\begin{proof}

For $\cD$-almost all $s \in S$, $[H_s : F_s] = [H_s : L_s][L_s : F_s]$, and thus
\begin{align*}
[H : F]_u &=\ulim_{s \in S}[H_s: F_s] \\
&= \ulim_{s \in S}[H_s: L_s][L_s: F_s] \\
&= \left(\ulim_{s \in S}[H_s : L_s]\right)\left(\ulim_{s \in S}[L_s: F_s]\right) \\
&= [H : L]_u[L : F]_u,
\end{align*}
as required.

\end{proof}

\begin{definition}
\label{def-ultra-field-extenson-adjoining-elements}

Let $F = \prod_{s \in S}F_s/\cD$ be an ultra-subfield of the ultra-field $L$, and let $X$ be any subset of $L$. Then the field $F(X)_u$ is defined as the intersection of all ultra-subfields of $L$ containing both $F$ and $X$, and is called the \textbf{ultra-extension of $F$ obtained by adjoining the elements in $X$}. For a finite set $X = \{\alpha_1, \ldots, \alpha_n\}$, we write $F(X)_u = F(\alpha_1, \ldots, \alpha_n)_u$. If $X$ consists of a single element $\alpha \in L$, then $H = F(\alpha)_u$ is said to be a \textbf{simple ultra-extension of $F$}, and $\alpha$ is called a \textbf{defining element of $H$ over $F$}.

When the ultra-field $L$ is not mentioned in the context, we take $L$ to be the ultra-algebraic closure $F^{\alg_u} = \prod_{s \in S}F_s^{\alg}/\cD$, where $F_s^{\alg}$ denotes an algebraic closure of $F_s$ for $\cD$-almost all $s \in S$.

\end{definition}

We describe simple ultra-extensions of an ultra-field.

\begin{theorem}
\label{thm-ultra-degrees-of-simple-ultra-extensions}

Let $F = \prod_{s \in S}F_s/\cD$ be an ultra-field, where the $F_s$ are fields. Let $c = \ulim_{s \in S}c_s$ be an element in an ultra-extension of $F$ such that $c$ is ultra-algebraic over $F$, and $c_s$ is algebraic over $F_s$ for $\cD$-almost all $s \in S$. Let $p = \ulim_{s \in S}p_s$ be the minimal ultra-polynomial of $c$ of ultra-degree $d = \ulim_{s\in S}d_s \in \bN^{\#}$, where the $p_s$ are the minimal polynomials of $c_s$ of degrees $d_s \in \bZ_{>0}$ in $F_s[x]$ for $\cD$-almost all $s \in S$ (see Theorem \ref{thm-minimal-ultra-polynomials}). Then
\begin{itemize}

\item [(i)] $F(c)_u = \prod_{s \in S}F_s(c_s)/\cD$.

\item [(ii)] $[F(c)_u : F] = \ulim_{s \in S}[F_s(c_s) : F_s] = d.$

\end{itemize}

\end{theorem}

\begin{proof}

Since $c \in \prod_{s \in S}F_s(c_s)/\cD$, it follows from Definition \ref{def-ultra-field-extenson-adjoining-elements} that $F(c)_u \subseteq \prod_{s \in S}F_s(c_s)/\cD$.

Suppose that $L = \prod_{s \in S}L_s/\cD$ is an arbitrary ultra-field extension of $F$ that contains $c$. Then $L_s$ contains $F_s$ for $\cD$-almost all $s \in S$. Since $c = \ulim_{s \in S}c_s$, we deduce that $L_s$ contains $c_s$ for $\cD$-almost all $s \in S$. Thus $F_s(c_s)$ is a subfield of $L_s$ for $\cD$-almost all $s \in S$, and hence $L$ is an ultra-field extension of $\prod_{s \in S}F_s(c_s)/\cD$. Therefore $\prod_{s \in S}F_s(c_s)/\cD$ is the smallest ultra-field extension of $F$ that contains $c$, and thus part (i) follows immediately.

For part (ii), we know that $[F_s(c_s) : F_s] = \deg(p_s) = d_s$, and thus $[F(c)_u : F] = \ulim_{s \in S}[F_s(c_s) : F_s] = \ulim_{s \in S}d_s = d.$

\end{proof}

The notion of ultra-algebraic extensions generalizes that of algebraic extensions, which is signified in the following.

\begin{lemma}
\label{lem-algebraicity-implies-ultra-algebraicity}

Let $F = \prod_{s\in S}F_s/\cD$ be an ultra-field for some fields $F_s$. Let $c$ be an element of a field extension of $F$ such that $c$ is algebraic of degree $d \in \bZ_{>0}$ over $F$ and its minimal polynomial is separable. Then
\begin{itemize}

\item [(i)] $c$ is ultra-algebraic over $F$, and can be written as $c = \ulim_{s \in S}c_s$, where $c_s$ is algebraic over $F_s$ for $\cD$-almost all $s \in S$.

\item [(ii)] The simple extension $F(c)$ of $F$ obtained by adjoining $c$ to $F$ is an ultra-field such that
\begin{align*}
F(c) = F(c)_u = \prod_{s \in S}F_s(c_s)/\cD,
\end{align*}
where $F_s(c_s)$ is a simple extension of degree $d$ over $F_s$ for $\cD$-almost all $s \in S$ and the $c_s$ are defined in part (i) above.

\end{itemize}

\end{lemma}

\begin{proof}

Let $h(x) = a_dx^d + a_{d - 1}x^{d - 1} + \cdots + a_1x + a_0$ be the minimal polynomial of $c$, where the $a_i$ belong in $F$ and $a_d \ne 0$. Using the same arguments as in Lemma \ref{lem-roots-of-polynomials-over-ultrafields}, one can write
\begin{align*}
h(x) = \ulim_{s \in S}h_s(x),
\end{align*}
where
\begin{align*}
h_s(x) = a_{d, s}x^d + a_{d - 1, s}x^{d - 1} + \cdots + a_{1, s}x + a_{0, s} \in F_s[x]
\end{align*}
and $a_i = \ulim_{s \in S}a_{i, s}$ for every $0 \le i \le d$. Since $c$ is a root of $h(x)$, it follows from Lemma \ref{lem-roots-of-polynomials-over-ultrafields} that $c = \ulim_{s \in S}c_s$, where $c_s$ is some root of $h_s(x)$ in an algebraic closure of $F_s$ for $\cD$-almost all $s \in S$. Set
\begin{align*}
H = \prod_{s \in S}F_s(c_s)/\cD.
\end{align*}
Then $c \in H$ and $H$ is an ultra-extension of $F$. It is trivial that $h$ is also an ultra-polynomial, and thus $c$ is ultra-algebraic over $F$. So part (i) follows immediately.

Since $h(x)$ is an irreducible and separable polynomial of degree $d$, using the same arguments as in Lemma \ref{lem-roots-of-polynomials-over-ultrafields}, $h_s(x)$ is irreducible and separable of degree $d$ for $\cD$-almost all $s \in S$, and thus the extension $F_s(c_s)$ is simple of degree $d$ for $\cD$-almost all $s \in S$. Thus for $\cD$-almost all $s \in S$, every element in $F_s(c_s)$ can be written in the form $\sum_{i = 0}^{d - 1}b_{i, s}c_s^i$ for some $b_{i, s} \in F_s$. Thus every element $\alpha$ in $\prod_{s \in S}F_s(c_s)/\cD$ can be represented in the form
\begin{align*}
\alpha &= \ulim_{s\in S}\left(\sum_{i = 0}^{d - 1}b_{i, s}c_s^i\right)\\
&= \sum_{i = 0}^{d - 1}\left(\ulim_{s \in S}b_{i, s}\right)\left(\ulim_{s \in S}c_s\right)^i \\
&= \sum_{i = 0}^{d - 1}\left(\ulim_{s \in S}b_{i, s}\right)c^i,
\end{align*}
which proves that $\alpha \in F(c)$ since $\ulim_{s \in S}b_{i, s} \in F$ for every $0 \le i \le d - 1$. Therefore $\prod_{s \in S}F_s(c_s)/\cD \subseteq F(c)$. It is trivial that $F(c)$ is a subset of $\prod_{s \in S}F_s(c_s)/\cD$, and thus we deduce that
\begin{align*}
F(c) = \prod_{s \in S}F_s(c_s)/\cD.
\end{align*}

Theorem \ref{thm-ultra-degrees-of-simple-ultra-extensions} implies that $F(c)_u = \prod_{s \in S}F_s(c_s)/\cD$, and thus part (ii) follows immediately.

\end{proof}

\begin{remark}

In the proof of the above lemma, the only place we use the separability assumption of the minimal polynomial $h$ is the representation of $c$ as $\ulim_{s \in S}c_s$ for some root $c_s$ of the $h_s$ since we need to apply Lemma \ref{lem-roots-of-polynomials-over-ultrafields} for $h$.

\end{remark}

\begin{definition}
\label{def-divisibility-of-ultra-polynomials}
(divisibility of ultra-polynomials)

Let $F = \prod_{s \in S}F_s/\cD$ be an ultra-field where the $F_s$ are fields. Let $P = \ulim_{s \in S}P_s$, $Q = \ulim_{s \in S}Q_s$ be ultra-polynomials in the ultra-hull $\cU(F[x])$. We say that \textbf{$P$ divides $Q$} (or \textbf{$P$ is a divisor of $Q$}) in $\cU(F[x])$ if there exists an ultra-polynomial $H = \ulim_{s \in S}H_s$ such that $Q = PH$, i.e., $Q_s = P_sH_s$ for $\cD$-almost all $s \in S$.

In symbols, we write $Q \equiv 0 \pmod{P\cU(F[x])}$ whenever $P$ divides $Q$. For arbitrary ultra-polynomials $Q_1, Q_2$, we write $Q_1 \equiv Q_2 \pmod{P\cU(F[x])}$ whenever $P$ divides $Q_1 - Q_2$.

\end{definition}

\begin{definition}
\label{def-relatively-prime-ultra-poly}
(relatively prime ultra-polynomials)

Let $F = \prod_{s \in S}F_s/\cD$ be an ultra-field where the $F_s$ are fields. Ultra-polynomials $P = \ulim_{s \in S}P_s$, $Q = \ulim_{s \in S}Q_s$ in the ultra-hull $\cU(F[x])$ for some polynomials $P_s, Q_s \in F_s[x]$ are \textbf{relatively prime} if $P_s$, $Q_s$ are relatively prime in $F_s[x]$ for $\cD$-almost all $s \in S$.

\end{definition}

The following result is useful in the proof of the Reciprocity Law in Section \ref{sec-reciprocity-law}.
\begin{lemma}
\label{lem-relatively-prime-ultra-polynomials}

Let $F = \prod_{s \in S}F_s/\cD$ be an ultra-field where the $F_s$ are fields. Let $f = \ulim_{s \in S}f_s$ be an ultra-polynomial in $\cU(F[x])$, and let $P = \ulim_{s \in S}P_s$, $Q = \ulim_{s\in S}Q_s$ be relatively prime ultra-polynomials such that both $P$, $Q$ divide $f$. Then $PQ$ divides $f$, i.e., $f \equiv 0 \pmod{PQ\cU(F[x])}$.

\end{lemma}

\begin{proof}

By assumption, there exist $G = \ulim_{s \in S}G_s$, $H = \ulim_{s\in S}H_s$ in $\cU(F[x])$ for some polynomials $G_s, H_s \in F_s[x]$ such that $f = PG = QH$, i.e.,
\begin{align*}
f_s = P_sG_s = Q_sH_s
\end{align*}
for $\cD$-almost all $s \in S$. Since $P_s, Q_s$ are relatively prime in $F_s[x]$, $Q_s$ divides $G_s$, and thus $G_s = Q_sK_s$ for some polynomial $K_s \in F_s[x]$. Thus $f_s = P_sQ_SK_s$ for $\cD$-almost all $s \in S$, and therefore $f = PQK$, where $K = \ulim_{s\in S}K_s \in \cU(F[x])$.

\end{proof}

Recall (see \cite{Artin-Tate} or \cite{serre-local-fields}) that a \textbf{quasi-finite field} is a perfect field $F$ such that $F$ has a unique extension $F_n$ of degree $n$ for each integer $n \ge 1$, and that the union of these extensions is equal to an algebraic closure of $F$.

The following result is straightforward from Lemma \ref{lem-algebraicity-implies-ultra-algebraicity}.

\begin{corollary}
\label{cor-quasifinite-ultra-products}

Let $F_s$ be a quasi-finite field for $\cD$-almost all $s\in S$. Then the ultra-field $\prod_{s \in S}F_s/\cD$ is quasi-finite as a field.

\end{corollary}

\begin{remark}
\label{rem-quasifinite-ultra-fields}

In the language of fields, there is a sentence that axiomatizes the property ``the field has exactly one algebraic extension of degree $n$ for every positive integer $n$" (see \cite{Chatzidakis}). One can also axiomatize the property ``the field is perfect". Thus using \L{}o\'s' theorem, we deduce that an ultra-field $F = \prod_{s \in S}F_s/\cD$ is quasi-finite if and only if the $F_s$ are quasi-finite for $\cD$-almost all $s \in S$. Lemma \ref{lem-algebraicity-implies-ultra-algebraicity} provides another proof to the above corollary based on the theory of ultra-fields that is analogous to that of field theory.

\end{remark}

\begin{definition}
\label{def-ultra-field-homomorphism}
(ultra-field homomorphisms)

Let $F = \prod_{s \in S}F_s/\cD$ and $G = \prod_{s\in S}G_s/\cD$ be ultra-fields. A map $\phi : F \to G$ is an \textbf{ultra-field homomorphism from $F$ to $G$} if there exists a collection of field homomorphisms $\phi_s : F_s \to G_s$ for $\cD$-almost all $s \in S$ such that
\begin{align*}
\phi(\ulim_{s\in S}a_s) = \ulim_{s\in S}\phi_s(a_s)
\end{align*}
for all $\ulim_{s\in S}a_s \in F$ for some elements $a_s \in F_s$.

In symbol, we write $\phi = \ulim_{s\in S}\phi_s$.

An ultra-homomorphism $\phi : F \to G$ is \textbf{injective} (resp. \textbf{surjective}) if $\phi$ is an injective (resp. surjective) map. An ultra-homomorphism $\phi : F \to G$ is an \textbf{ultra-field isomorphism} if $\phi$ is both injective and surjective.

If $\phi : F \to F$ is an ultra-field isomorphism, we also call $\phi$ an \textbf{ultra-field automorphism of $F$}.

\end{definition}

The following shows that the notion of ultra-field homomorphisms generalize that of field homomorphisms.

\begin{lemma}
\label{lem-ultra-field-isomorphisms}

Let $F = \prod_{s \in S}F_s/\cD$ and $G = \prod_{s\in S}G_s/\cD$ be ultra-fields. Let $\phi = \ulim_{s\in S}\phi_s : F \to G$ be an ultra-field homomorphism. Then
\begin{itemize}

\item [(i)] $\phi : F \to G$ is a field homomorphism.

\item [(ii)] If $\phi_s : F_s \to G_s$ is injective (resp. surjective) for $\cD$-almost all $s \in S$, then $\phi : F \to G$ is an injective (resp. surjective) ultra-field homomorphism. In particular, if $\phi_s : F_s \to G_s$ is an isomorphism for $\cD$-almost all $s \in S$, then $\phi$ is an ultra-field isomorphism.

\end{itemize}

\end{lemma}

\begin{proof}

Since the $\phi_s$ are field homomorphisms for $\cD$-almost all $s \in S$, part (i) follows immediately from Definition \ref{def-ultra-field-homomorphism}.

Suppose that the $\phi_s$ are injective for $\cD$-almost all $s \in S$. As a field homomorphism, the kernel of $\phi$ satisfies
\begin{align*}
\ker(\phi) &=\{\ulim_{s\in S}a_s \; | \; \phi(\ulim_{s\in S}a_s) = 0\} \\
&= \{\ulim_{s\in S}a_s \; | \; \ulim_{s\in S}\phi_s(a_s) = 0\} \\
&= \{\ulim_{s\in S}a_s \; | \; \text{$a_s \in \ker(\phi_s)$ for $\cD$-almost all $s \in S$}\} \\
&= \{\ulim_{s\in S}a_s \; | \; \text{$a_s = 0$ for $\cD$-almost all $s \in S$}\}  \\
&= \{0\}.
\end{align*}
Thus $\phi$ is injective.

Suppose that the $\phi_s$ are surjective for $\cD$-almost all $s \in S$. Take an arbitrary $b = \ulim_{s\in S}b_s \in G$ for some elements $b_s \in G_s$. Then there exist elements $a_s \in F_s$ such that $\phi_s(a_s) = b_s$ for $\cD$-almost all $s \in S$. Letting $a = \ulim_{s \in S}a_s \in F$, we deduce that
\begin{align*}
\phi(a) = \ulim_{s\in S}\phi_s(a_s) = \ulim_{s\in S}b_s = b.
\end{align*}
Thus $\phi$ is surjective, and hence part (ii) follows immediately.

\end{proof}

\begin{remark}
\label{rem-ultra-isomorphism-stronger-than-field-isomorphism}

Lemma \ref{lem-ultra-field-isomorphisms} also implies that if $\phi$ is an ultra-field isomorphism, then $\phi$ is a field isomorphism.

\end{remark}

We now describe all ultra-algebraic extensions of an ultraproduct of quasi-finite fields.

\begin{theorem}
\label{thm-ultra-algebraic-extensions-of-quasifinite-fields}

Let $F =\prod_{s \in S}F_s/\cD$ be an ultra-field such that the $F_s$ are quasi-finite for almost all $s \in S$. For every hypernatural natural $n = \ulim_{s \in S}n_s$ with $n_s \in \bZ_{> 0}$, there exists a unique ultra-algebraic extension of ultra-degree $n$ over $F$ up to ultra-isomorphism.

\end{theorem}

\begin{proof}

Suppose that there is an ultra-algebraic extension $L = \prod_{s \in S}L_s/\cD$ of ultra-degree $n$ over $F$, where $L_s$ is an extension of degree $n_s$ over $F_s$ for $\cD$-almost all $s \in S$. Since $F_s$ is quasi-finite for $\cD$-almost all $s \in S$, and thus $L_s$ is a unique extension of degree $n_s$ over $F_s$ up to isomorphism $\phi_s$ for $\cD$-almost all $s \in S$. Therefore by Lemma \ref{lem-ultra-field-isomorphisms} and Remark \ref{rem-ultra-isomorphism-stronger-than-field-isomorphism}, $L = \prod_{s \in S}L_s/\cD$ is a unique ultra-extension of ultra-degree $n = \ulim_{s \in S}n_s$ over $F$ up to ultra-isomorphism $\ulim_{s\in S}\phi_s$.

For the existence of such an ultra-algebraic extension of ultra-degree $n$ over $F$, we see that for $\cD$-almost all $s \in S$, there exists a unique extension $L_s$ of degree $n_s$ over $F_s$. Hence their ultraproduct $\prod_{s \in S}L_s/\cD$ is an ultra-algebraic extension of ultra-degree $n = \ulim_{s \in S}n_s$ over $F$, as required.

\end{proof}

\subsection{Ultra-finite fields}

In this subsection, we study ultra-finite fields which are ultraproducts of finite fields. Our approach is based rather on the classical theory of finite fields than model theory as in literature. Ultra-finite fields are introduced and developed in the work of Ax (see, for example, \cite{ax-1968} and \cite{FJ}).

\begin{definition}
A set $\bF$ is called an \textbf{ultra-finite field} if $\bF=\prod_{s\in S}\bF_{q_s}/\cD$, where $\bF_{q_s}$ are finite fields of $q_s$ elements and $q_s$ are a power of a prime $p_s$ for $\cD$-almost all $s \in S$.

\end{definition}

By Proposition \ref{prop-notion-section-algebraic-structures-of-ultraproducts} or \L{}o\'s' theorem, $\bF$ must be a field.

\begin{remark}
\label{rem-ultra-Frobenius-map-definition}

\hfill
\begin{itemize}
\item[(i)] If $\bF$ is an ultra-finite field, then the internal cardinality of $\bF$ is
\begin{align*}
\alpha = \ulim_{s\in S}q_s \in \bN^{\#}.
\end{align*}

\item[(ii)] If $p_s = p$ for some prime $p > 0$ for $\cD$-almost all $s \in S$, then $\bF$ has characteristic $p$. Otherwise $\bF$ has characteristic $0$.

\item[(iii)] Let $\cF_s: \bF_{q_s} \to \bF_{q_s}$ be the Frobenius map that sends each $a \in \bF_s$ to $\cF_s(a)=a^{q_s}$.

Each $\cF_s$ is a field automorphism for $\cD$-almost $s \in S$. By Lemma \ref{lem-ultra-field-isomorphisms}, the map $\cF = \ulim_{s\in S}\cF_s : \bF \to \bF$ defined by $\cF(x)=x^\alpha = \ulim_{s \in S}x_s^{q_s} = \ulim_{s\in S}\cF_s(x_s)$ is an ultra-field automorphism of $\bF$ which we call the \textbf{ultra-Frobenius map}, where $x = \ulim_{s \in S}x_s$.

\item [(iv)] Since every finite field is perfect, it follows from \L{}o\'s' theorem that every ultra-finite field is perfect.

\end{itemize}
\end{remark}

An \textbf{ultra-prime} $\alpha$ in $\bN^{\#}$ is the ultraproduct of primes $p_s$, i.e., $\alpha = \ulim_{s \in S}p_s$, where the $p_s$ are primes for $\cD$-almost all $s \in S$.

If $p_s = p$ for some prime $p$ for $\cD$-almost all $s\in S$, $\alpha = p$ which is a prime in $\bZ$. Thus the notion of ultra-primes in $\bZ^{\#}$ generalizes that of primes in $\bZ$.

\begin{definition}
\label{def-ultra-Galois-fields}
(ultra-Galois fields)

Let $p_s$ be primes for $\cD$-almost all $s \in S$, and let $\bF_{p_s}$ be the finite field of $p_s$ elements for $\cD$-almost all $s \in S$. The ultra-finite field $\prod_{s \in S}\bF_{p_s}$ is called the \textbf{ultra-Galois field}, and denoted by $\GF(\alpha)$, where $\alpha = \ulim_{s \in S}p_s$ is an ultra-prime in $\bN^{\#}$ which is also the internal cardinality of $\GF(\alpha)$.

\end{definition}

\begin{remark}
\label{rem-char-of-ultra-Galois-field}

If the primes $p_s$ are distinct for $\cD$-almost all $s \in S$, then $\GF(\alpha)$ has characteristic $0$. Since the $p_s$ are arbitrarily large, by \cite{bell-slomson}, the cardinality of $\GF(\alpha)$ is $2^{\aleph_0}$ which implies that $\GF(\alpha)$ has the cardinality of the continuum, and so is uncountable.

If the primes $p_s = p$ for some prime $p$ for $\cD$-almost all $s \in S$, then $\alpha = p$, and thus $\GF(\alpha)$ has characteristic $p$ and is of internal cardinality $p$. By Theorem \ref{thm-internal-cardinality}, $\GF(\alpha)$ is in bijection with the finite set $\{1, \ldots, p\}$ of $p$ elements, and thus the cardinality of $\GF(\alpha)$ is exactly $p$. Thus $\GF(\alpha)$ is the Galois field $\bF_p$, which implies that the notion of ultra-Galois field generalizes that of Galois fields.

\end{remark}

\begin{remark}

Note that the notion of the ultra-Galois field $\GF(\alpha)$ does not depend on the representation of $\alpha$ as the ultraproduct of primes. Indeed, suppose that $\alpha = \ulim_{s \in S}p_s = \ulim_{s \in S}p_s'$, where the $p_s$ are distinct primes for $\cD$-almost $s \in S$, and the $p_s'$ are distinct primes for $\cD$-almost all $s \in S$. Then $p_s = p_s'$ for $\cD$-almost all $s \in S$, and $\bF_{p_s} = \bF_{p_s'}$ for $\cD$-almost all $s \in S$, and thus $\GF(\alpha)  = \prod_{s \in S}\bF_{p_s}/\cD = \prod_{s \in S}\bF_{p_s'}/\cD$.

\end{remark}

\begin{theorem}
\label{thm-uniqueness-of-extensions-of-ultra-finite-fields}

Let $\alpha = \ulim_{s \in S}p_s$ be an ultra-prime in $\bN^{\#}$, where the $p_s$ are primes for $\cD$-almost all $s \in S$. For each hypernatural number $d = \ulim_{s \in S}d_s$ for some positive integers $d_s$, there exists a unique ultra-algebraic ultra-extension  of ultra-degree $d$ over the ultra-Galois field $\GF(\alpha)$ up to ultra-isomorphism that is the ultra-finite field $\prod_{s \in S}\bF_{p_s^{d_s}}/\cD$.

Furthermore if $d$ is a positive integer, there exists a unique algebraic extension of degree $d$ over $\GF(\alpha)$ up to ultra-isomorphism \footnote{Note that in this case, a unique algebraic extension of degree $d$ over $\GF(\alpha)$ is also an ultra-algebraic ultra-extension of ultra-degree $d$, and thus the uniqueness is up to ultra-isomorphism instead of usual field isomorphism as in the classical field theory} that is of the form
$\prod_{s \in S}\bF_{p_s^{d}}/\cD$.

\end{theorem}

\begin{proof}

It is known that finite fields are quasi-finite (see Serre \cite{serre-local-fields}). Then the first assertion follows immediately from Theorem \ref{thm-ultra-algebraic-extensions-of-quasifinite-fields} and its proof.

By Remark \ref{rem-ultra-Frobenius-map-definition}(iv), $\GF(\alpha)$ is a perfect field, and thus by the primitive element theorem (see Lang \cite{lang-algebra}) any algebraic extension of degree $d \in \bZ_{>0}$ over $\GF(\alpha)$ is a simple extension. We know from Lemma \ref{lem-algebraicity-implies-ultra-algebraicity} that a simple ultra-algebraic ultra-extension of ultra-degree $d$ over $\GF(\alpha)$ is the same as a simple algebraic extension of degree $d$ over $\GF(\alpha)$. Since $d = \ulim_{s\in S}d$, the last assertion follows immediately from the first part.

\end{proof}

\begin{remark}
\label{rem-general-ultra-Galois-fields}

Motivated by the above theorem, we denote by $\GF(\alpha^d)$ the unique ultra-extension of $\GF(\alpha)$ of ultra-degree $d$. Whenever $\beta = \alpha^d$, we also write $\GF(\beta)$ instead of $\GF(\alpha^d)$ for simplicity. We also denote by $\GF(\alpha^d)^{\times}$ the multiplicative group consisting of all nonzero elements of $\GF(\alpha^d)$.

We are most interested in the ultra-finite field $\GF(\alpha^d)$ when $d$ is a positive integer. In that case, one can write $\GF(\alpha^d) = \prod_{s \in S}\bF_{p_s^d}/\cD$.

\end{remark}

The following result is immediate from Theorem \ref{thm-uniqueness-of-extensions-of-ultra-finite-fields}.
\begin{corollary}
\label{cor-uniqueness-of-ultra-Galois-fields}

Let $\alpha = \ulim_{s \in S}p_s$ be an ultra-prime in $\bN^{\#}$, where the $p_s$ are primes for $\cD$-almost all $s \in S$. Let $d = \ulim_{s \in S}d_s$ be a hypernatural in $\bN^{\#}$ for some natural numbers $d_s$, and let $q_s = p_s^{d_s}$ for each $s \in S$ so each $q_s$ is a power of prime $p_s$. Let $\beta = \alpha^d = \ulim_{s \in S}q_s$. For any hypernatual $m = \ulim_{s \in S}m_s$ for some natural numbers $m_s$, there exists a unique ultra-algebraic ultra-extension of ultra-degree $m$ over the ultra-Galois field $\GF(\beta) = \GF(\alpha^d)$ up to isomorphism that is the ultra-finite field $\prod_{s \in S}\bF_{q_s^{m_s}}/\cD$.

Furthermore if $m$ is a positive integer, there exists a unique algebraic extension of degree $m$ over $\GF(\beta) = \GF(\alpha^d)$ up to isomorphism that is of the form $\prod_{s \in S}\bF_{q_s^{m}}/\cD$.

\end{corollary}

\begin{lemma}
\label{lem-ultra-cyclic-structures}

Let $\alpha = \ulim_{s \in S}p_s$ be an ultra-prime in $\bN^{\#}$, where the $p_s$ are primes for $\cD$-almost all $s \in S$. Let $d = \ulim_{s \in S}d_s$ be a hypernatural number for some positive integers $d_s$, and set $q_s = p_s^{d_s}$ for each $s \in S$. Set $\beta = \alpha^d = \ulim_{s\in S}q_s \in \bN^{\#}$. For any element $a$ in the ultra-algebraic closure of $\GF(\beta)$, $a \in \GF(\beta)$ if and only if $a^{\beta} = a$. In particular $a \in \GF(\beta)^{\times}$ if and only if $a^{\beta - 1} = 1$.

\end{lemma}

\begin{proof}

By Remark \ref{rem-general-ultra-Galois-fields}, $\GF(\beta) = \GF(\alpha^d) = \prod_{s \in S}\bF_{q_s}/\cD$.

Let $a = \ulim_{s\in S}a_s$ be an element in the ultra-algebraic closure of $\GF(\beta)$, where the $a_s$ are elements in $\overline{\bF}_{q_s}$.

The identity $a^{\beta} =a$ is trivial for $a=0$.

Suppose that $a=\ulim_{s\in S} a_s$ is a nonzero element in $\GF(\beta)$, where $a_s \in \bF_{q_s}$  for $\cD$-almost $s \in S$. By \L{}o\'s' theorem, $a_s \neq 0$ for $\cD$-almost $s \in S$. Thus $a_s^{q_s -1} =1$ for $\cD$-almost $s \in S$ since $\bF_{q_s}^{\times}$ is a cyclic group of order $q_s -1$. Thus
\begin{align*}
a^{\beta-1} = \ulim_{s\in S}a_s^{q_s-1}=1,
\end{align*}
and therefore $a^{\beta}=a$.

Suppose that $a^{\beta} = a$. Thus $a_s^{q_s} = a_s$ for $\cD$-almost all $s\in S$, and therefore $a_s \in \bF_{q_s}$. Hence $a \in \GF(\beta)$, which yields the desired result.

\end{proof}

The following result and the notion of ultra-characteristics will not be used in another place in the paper, but it shows how analogous finite fields and ultra-finite fields are.

\begin{lemma}
\label{lem-ultra-characteristics}

Let $\alpha = \ulim_{s \in S}p_s$ be an ultra-prime in $\bN^{\#}$, where the $p_s$ are primes for $\cD$-almost all $s \in S$. Let $d = \ulim_{s \in S}d_s$ be a hypernatural number for some positive integers $d_s$. Then for every $a \in \GF(\alpha^d)$, $\alpha a = 0$.

\end{lemma}

\begin{proof}

Let $a = \ulim_{s \in S}a_s$ for some $a_s \in \bF_{p_s^{d_s}}$. Then $\alpha a = \ulim_{s \in S}p_s a_s = 0$ since the $\bF_{p_s^{d_s}}$ are finite fields of characteristics $p_s$ for $\cD$-almost $s \in S$.

\end{proof}

The above lemma motivates the following notion.

\begin{definition}
\label{def-ultra-characteristics}

$\alpha = \ulim_{s \in S}p_s$ is called the \textbf{ultra-characteristic} of the ultra-finite field $\GF(\alpha^d)$.

\end{definition}

We prove an analogue of finite fields for ultra-Galois fields.

\begin{lemma}
\label{lem-irreducible-ultra-polynomial-divides-x^alpha^d-x}

Let $\alpha = \ulim_{s \in S}p_s$ be an ultra-prime in $\bN^{\#}$, where the $p_s$ are primes for $\cD$-almost all $s \in S$. Let $d = \ulim_{s \in S}d_s$ be a hypernatural number for some positive integers $d_s$, and set $q_s = p_s^{d_s}$ for each $s \in S$. Set $\beta = \alpha^d = \ulim_{s\in S}q_s \in \bN^{\#}$. Let $f = \ulim_{s \in S}f_s$ be an irreducible ultra-polynomial in the ultra-hull $\cU(\GF(\beta)[x])$. Suppose that the ultra-degree of $f$ is $m$ in $\bN^{\#}$, and let $n$ be a hypernatural number. Then $f$ divides the ultra-polynomial $x^{\beta^{n}}-x$ if and only if $m$ divides $n$ in $\bN^{\#}$.

\end{lemma}

\begin{proof}

Suppose that $f$ divides $x^{\beta^{n}}-x$. Let $a$ be a root of $f$ in an ultra-algebraic closure of $\GF(\beta)$. Then $a^{\beta^{n}}=a$, so it follows from Lemma \ref{lem-ultra-cyclic-structures}, $a \in \GF(\beta^{n})$.

It follows that the simple ultra-extension $\GF(\beta)(a)_u$ is an ultra-subfield of $\GF(\beta^{n})$. By Theorem \ref{thm-ultra-degrees-of-simple-ultra-extensions}, we know that
\begin{align*}
[\GF(\beta)(a)_u: \GF(\beta)]_u = m.
\end{align*}
On the other hand, by definition, $[\GF(\beta^{n}): \GF(\beta)]_u = n$, and thus by Lemma \ref{lem-multiplicative-of-ultra-degrees}, $m$ divides $n$ in $\bN^{\#}$.

Conversely, if $m$ divides $n$ in $\bN^{\#}$, we deduce from Theorem \ref{thm-uniqueness-of-extensions-of-ultra-finite-fields} that $ \GF(\beta^{n})$ contains $ \GF(\beta^m)$  as a subfield. If $a$ is a root of $f$ in an ultra-algebraic closure of $\GF(\beta)$, then $[\GF(\beta)(a)_u: \GF(\beta)]_u = m$ and so it follows from Theorem \ref{thm-uniqueness-of-extensions-of-ultra-finite-fields} that $\GF(\beta)(a)_u = \GF(\beta^{m})$. Consequently $a \in \GF(\beta^m) \subseteq \GF(\beta^{n})$, and thus $a^{\beta^{n}}=a$ which implies that $a$ is a  root of  $x^{\beta^{n}}-x$. By  Lemma \ref{lem-divisibility-of-minimal-ultra-poly}, we deduce that $f$ divides $x^{\beta^{n}}-x$.

\end{proof}

\begin{theorem}
\label{thm-roots-of-ultra-polynomials-in-ultra-Galois-fields}

Let $\alpha = \ulim_{s \in S}p_s$ be an ultra-prime in $\bN^{\#}$, where the $p_s$ are primes for $\cD$-almost all $s \in S$. Let $d = \ulim_{s \in S}d_s$ be a hypernatural number for some positive integers $d_s$, and let $q_s = p_s^{d_s}$ for each $s \in S$. Set $\beta = \alpha^d = \ulim_{s \in S}q_s \in \bN^{\#}$. Let $f$ be an irreducible ultra-polynomial of ultra-degree $m = \ulim_{s \in S}m_s \in \bN^{\#}$ over $\GF(\beta)$, where the $m_s$ are positive integers. Then
\begin{itemize}

\item [(i)] $f$ has a root $a$ in $\GF(\beta^{m})$.

\item [(ii)] Furthermore, all the roots of $f$ consist of exactly all distinct elements $a$, $a^{\beta}$, $a^{\beta^{2}}, \ldots, a^{\beta^{m-1}}$ of $\GF(\beta^{m})$.
\end{itemize}

\end{theorem}

\begin{remark}

Note that the set $\{ a, a^{\beta}, a^{\beta^{2}}, \ldots, a^{\beta^{m-1}} \}$ is an internal set which has internal cardinality $m \in \bN^{\#}$.

\end{remark}

\begin{proof}

Let $a$ be a root of $f$ in an ultra-algebraic closure of $f$ over $\GF(\beta)$. Then the ultra-degree $[\GF(\beta)(a)_u: \GF(\beta)]_u=m$, and thus $[\GF(\beta)(a)_u : \GF(\alpha)]_u = dm$, since $[\GF(\beta) : \GF(\alpha)] = d$. By Corollary \ref{cor-uniqueness-of-ultra-Galois-fields}, $\GF(\beta)(a)_u=\GF(\beta^{m})$, which in particular implies that $a \in \GF(\beta^{m})$.

We now show that if $b = \ulim_{s \in S} b_s \in \GF(\beta^{m})$ for some $b_s \in \bF_{q_s^{m_s}}$ is a root of $f$, then $b^{\beta}$ is also a root of $f$. Write $f = \ulim_{s \in S}f_s$, where $f_s(x) = a_{m_s}^{(s)}x^{m_s}+\cdots+a_{1}^{(s)}x+a_{0}^{(s)}$ be an irreducible polynomial in $\bF_{q_s}[x]$ for $\cD$-almost all $s \in S$. By definition, we know that
\begin{align*}
f(b^{\beta}) &= \ulim_{s\in S}f_s(b_s^{q_s}) \\
&= \ulim_{s\in S} \left(a_{m_s}^{(s)}b_{s}^{q_sm_s}+\cdots+a_1^{(s)}b_{s}^{q_s}+a_0^{(s)} \right)\\
&= \ulim_{s\in S} \left(  (a_{m_s}^{(s)}b_{s}^{m_s})^{q_s}+\cdots+ (a_1^{(s)}b_{s})^{q_s} + (a_0^{(s)})^{q_s} \right), \; \left(\text{since}\; (a_i^{(s)})^{q_s} = a_i^{(s)}\right) \\
&= \ulim_{s\in S} \left(a_{m_s}^{(s)}b_{s}^{m_s}+\cdots + a_1^{(s)}b_{s}+a_0^{(s)} \right)^{q_s}\\
&= \ulim_{s\in S} f_s(b_s)^{q_s} \\
&= f(b)^{\beta} \\
&= 0.
\end{align*}
Therefore the elements $a, a^{\beta}, a^{\beta^{2}}, \ldots, a^{\beta^{m-1}}$ are roots of $f$.

It remains to prove that these elements are distinct. Since $m = \ulim_{s\in S}m_s \in \bZ^{\#}$, where $m_s \in \bN^{\#}$, we obtain an internal subset of $\bN^{\#}$ of the form
\begin{align*}
A=\{0,1,\ldots, m-1\} = \{i \in \bN^{\#}: i \le m - 1\} = \prod_{s \in S} \{0, \ldots, m_s-1\}/\cD,
\end{align*}
and thus for any $i \in A$, there exist nonnegative integers $0 \leq i_s \leq m_s-1$ such that $i = \ulim_{s \in S} i_s$.

Suppose that there exist distinct elements $i, j \in A$ such that $i = \ulim_{s \in S} i_s$ and $j = \ulim_{s \in S} j_s$ for some $0 \leq i_s, j_s \leq m_s-1$ for which
\begin{align}
\label{eq1-in-roots-of-ultra-polynomials-in-ultra-Galois}
a^{\beta^{i}} = a^{\beta^{j}}.
\end{align}

Recall from Section \ref{subsec-hyperintegers} that the order ``$\leq$" is a total order on $\bZ^{\#}$, and thus  the restriction of ``$\leq$" to $A$ is a total order on $A$.

Thus, without loss of generality, we assume that $i \leq j$ in (\ref{eq1-in-roots-of-ultra-polynomials-in-ultra-Galois}), and thus $i < j$ since $i \ne j$. By assumption, $j \leq m-1$. By raising both sides of (\ref{eq1-in-roots-of-ultra-polynomials-in-ultra-Galois}) to the power
\begin{align*}
\beta^{m-j} =  \ulim_{s\in S} q_s^{m_s-j_s} \in \bN^{\#},
\end{align*}
we get
\begin{align*}
a^{\beta^{m-j+i}} = a^{\beta^{m}} =a
\end{align*}
since $a \in \GF(\beta^{m})$ (see Lemma \ref{lem-ultra-cyclic-structures}). Since $a$ is a root of both ultra-polynomials $f$ and $x^{\beta^{m-j+i}}-x$, Lemma \ref{lem-divisibility-of-minimal-ultra-poly} implies that $f$ divides $x^{\beta^{m-j+i}} - x$, and it thus follows from Lemma \ref{lem-irreducible-ultra-polynomial-divides-x^alpha^d-x} that $m$ divides $m-j+i$ in $\bZ^{\#}_{> 0}$. Thus, there exists a positive hyperinteger $h = \ulim_{s \in S}h_s$ in $\bZ^{\#}_{> 0}$ such that $m-j+i = mh$, i.e., $m_s-j_s+i_s = m_sh_s$ for $\cD$-almost all $s \in S$. Thus $m_s \leq m_sh_s = m_s-j_s+i_s$. Since $i < j$, $i_s < j_s$ for $\cD$-almost all $s \in S$, and thus
\begin{align*}
m_s \leq m_s-(j_s - i_s) < m_s
\end{align*}
for $\cD$-almost $s \in S$, which is a contradiction.

Thus, the elements $a^{\beta^{i}}$ for $i \in A$ are all distinct. Therefore the set $\{ a^{\beta^{i}} : i \in A\}$ is an internal set of internal cardinality $m$.

Since $f = \ulim_{s \in S}f_s$ is irreducible over $\GF(\beta) = \prod_{s \in S}\bF_{q_s}/\cD$, $f_s$ is irreducible over the finite field $\bF_{q_s}$, and thus $f_s$ is separable. Therefore $f$ is also separable over $\GF(\beta)$. By Theorem \ref{thm-roots-of-ultra-polynomials-over-ultra-fields}, the set of roots of $f$ is an internal set of internal cardinality $m$. By Proposition \ref{prop-comparing-internal-sets-with-same-cardinality}, we deduce that the set of roots of $f$ is exactly $\{ a^{\beta^{i}} : i \in A\}$ as required.

\end{proof}

\begin{corollary}
\label{cor-finite-extensions-of-ultra-finite-fields-are-Galois}

Let $\alpha = \ulim_{s \in S}p_s$ be an ultra-prime in $\bN^{\#}$, where the $p_s$ are primes for $\cD$-almost all $s \in S$. Let $d = \ulim_{s \in S}d_s$ be a hypernatural number for some positive integers $d_s$, and let $q_s = p_s^{d_s}$ for each $s \in S$. Set $\beta = \alpha^d = \ulim_{s \in S}q_s \in \bN^{\#}$. For any positive integer $m \in \bZ_{>0}$, the ultra-finite field $\GF(\beta^m)$ is a Galois extension of degree $m$ over $\GF(\beta)$.

\end{corollary}

\begin{proof}

By Corollary \ref{cor-uniqueness-of-ultra-Galois-fields}, $\GF(\beta^m)$ is the unique algebraic extension of degree $m$ over $\GF(\beta)$. Since $\GF(\beta^m)$ is a perfect field, the primitive element theorem (see Lang \cite{lang-algebra}) implies that there exists an element $\lambda \in \GF(\beta^m)$ such that $\GF(\beta^m) = \GF(\beta)(\lambda)$. Let $f(x) \in \GF(\beta)[x]$ be a minimal polynomial of $\lambda$ over $\GF(\beta)$. Then the degree of $f(x)$ is $m$. By Theorem \ref{thm-roots-of-ultra-polynomials-in-ultra-Galois-fields}, all the roots of $f(x)$ belong in $\GF(\beta^m)$, and thus $\GF(\beta^m)$ is the splitting field of $f(x)$. Thus $GF(\beta^m)$ is a Galois extension of degree $m$ over $\GF(\beta)$.

\end{proof}

\section{Power Residue Symbol}
\label{sec-power-residue-symbol}

In this section, we introduce power residue symbol for the ring of polynomials over an ultra-finite field. Throughout this section, for $\cD$-almost all $s \in S$, we fix $q_s = p_s^{\epsilon_s}$ where the $p_s$ are primes, and the $\epsilon_s$ are positive integers.

Let $\bA_s = \bF_{q_s}[t]$ be the ring of polynomials over the finite field $\bF_{q_s}$ for $\cD$-almost all $s \in S$. Set $\kappa = \ulim_{s \in S}q_s = \ulim_{s \in S}p_s^{\epsilon_s} \in \bN^{\#}$.

Let $\GF(\kappa)$ be the ultra-finite field as in Section \ref{sec-ultrafinite-fields} which is the ultraproduct of finite fields $\bF_{q_s}$ with respect to $\cD$, that is,
\begin{align*}
\GF(\kappa) = \prod_{s\in S}\bF_{q_s}/\cD.
\end{align*}

Let $\bA=\GF(\kappa)[t]$ be the ring of polynomials over $\GF(\kappa)$. In this section, we introduce power residue symbols in $\bA$, and compare it with the notion of power residue symbols in $\bA$ in \cite{CP}. See also \cite{serre-local-fields} in which such power residue symbols are implicitly described. In fact, Clark and Pollack in \cite{CP} follows Serre's description of power residue symbols in $\bA$ closely.

Our approach is different from that of Serre's and Clark--Pollack's in which we mainly use model theory, especially the theory of ultraproducts to define power residue symbols in $\bA$. We not only define the $n$-th power residue symbols in $\bA$ for integers $n>0$, but also define the $n$-th power residue symbols in $\bA$ for hypernaturals $n>0$.

Let $\cU(\bA)$ be the ultrahull of $\bA$, i.e., $\cU(\bA) = \prod_{s\in S} \bA_s/\cD$. There is a natural embedding $\bA \to \cU(\bA)$ from which one can view $\bA$ as a subring of $\cU(\bA)$. Note that elements in $\cU(\bA)$ are ultra-polynomials over $\GF(\kappa)$.

We first prove an analogue of Fermat's little theorem for $\bA$.

\begin{theorem}
\label{thm-Fermat-little-theorem-for-A}

Let $P$ be an irreducible polynomial of degree $d \in \bZ_{>0}$ in $\bA$. For any polynomial $\alpha \in \bA$ such that $\alpha, P$ are relatively prime in $\bA$, the congruence
\begin{align*}
\alpha^{\kappa^d - 1} \equiv 1 \pmod{P\cU(\bA)}
\end{align*}
holds in $\cU(\bA)$, that is, there exists an ultra-polynomial $Q \in \cU(\bA)$ such that
\begin{align*}
\alpha^{\kappa^d - 1} - 1 = PQ.
\end{align*}

\end{theorem}

\begin{proof}

Write $\alpha = \ulim_{s \in S}\alpha_s$ for some polynomials $\alpha_s \in \bA_s$, and $P = \ulim_{s \in S}P_s$ for some polynomials $P_s \in \bA_s$. Since $P$ is irreducible, it follows from \L{}o\'s's theorem that $P_s$ is irreducible for $\cD$-almost all $s \in S$. Since $\alpha, P$ are relative prime in $\bA$ and $\bA$ is a principal ideal domain, there exist polynomials $f = \ulim_{s \in S}f_s$ and $g = \ulim_{s\in S}g_s$ in $\bA$ for some polynomials $f_s, g_s \in \bA_s$ such that
\begin{align*}
\alpha f + Pg = 1.
\end{align*}
Thus $\ulim_{s \in S}(\alpha_sf_s + P_sg_s) = 1$, which implies that $\alpha_sf_s + P_sg_s = 1$ for $\cD$-almost all $s \in S$. Therefore $\alpha_s, P_s$ are relatively prime in $\bA_s$ for $\cD$-almost all $s \in S$. Since $P$ is of degree $d$, $P_s$ is of degree $d$ for $\cD$-almost all $s \in S$. By Fermat's little theorem for $\bA_s$ (see \cite[Corollary, p.5]{rosen}), we see that
\begin{align*}
\alpha_s^{q_s^d - 1} \equiv 1 \pmod{P_s\bA_s},
\end{align*}
or equivalently, there exists a polynomial $Q_s \in \bA_s$ such that
\begin{align*}
\alpha_s^{q_s^d - 1} - 1 = P_sQ_s.
\end{align*}
Thus
\begin{align*}
\alpha^{\kappa^d - 1} - 1 = \ulim_{s \in S}(\alpha_s^{q_s^d - 1} - 1) = \ulim_{s\in S}(P_sQ_s) = PQ,
\end{align*}
where $Q = \ulim_{s \in S}Q_s \in \cU(\bA)$. Thus the theorem follows immediately.

\end{proof}

\subsection{Ultra-cyclic groups}
\label{subsec-ultra-cyclic-groups}

In this subsection, we introduce ultra-cyclic groups that we need for defining power residue symbols in $\bA$.

Let $C_s = \langle g_s \rangle$ be a finite cyclic group of order $m_s \in \bZ_{>0}$ with generator $g_s$ for $\cD$-almost all $s \in S$, and let $C = \prod_{s \in S} C_s/\cD$ be the ultraproduct of $C_s$. By \L{}o\'s'theoem, $C$ is a group. Let $g = \ulim_{s \in S} g_s$.

For each element $a \in C$, one can write $a = \ulim_{s \in S}a_s$ for some elements $a_s \in C_s$. Thus $a_s = g_s^{h_s}$ for some nonnegative integer $h_s$, and hence $a = \ulim_{s \in S}g_s^{h_s} = g^h$, where $h = \ulim_{s \in S} h_s \in \bZ^{\#}_{\ge 0}$. In setting
\begin{align*}
\langle g \rangle_u = \{g^h \; : \; h \in \bZ^{\#}\},
\end{align*}
it is clear from the above discussion that $C = \langle g \rangle_u$. In this case, we say that \textbf{$C$ is ultra-generated by $g$}, and \textbf{$g$ is an ultra-generator of $C$.}

\begin{lemma}
\label{lem-ultra-cyclic-groups}

Keep the same notation as above, and let $m = \ulim_{s \in S}m_s \in \bZ^{\#}_{> 0}$. Then
\begin{itemize}

\item [(i)] $m$ is the least positive hyperinteger such that $g^m = 0$.

\item[(ii)] For any hyperintegers $k, \ell \in \bZ^{\#}$, $g^k = g^{\ell}$ if and only if $k \equiv \ell \pmod{m}$ in $\bZ^{\#}$.

\item[(iii)] there exists a bijection from the internal set $\{1, 2, \ldots, m\} = \{i \in \bN^{\#} : i \le m\}$ to $C$.

\item [(iv)] $C = \langle g \rangle_u = \{g^i \; : \; \text{$i \in \bN^{\#}$ and $i \le m$}\}$.

\end{itemize}

\end{lemma}

\begin{proof}

Since $\card(C_s) = m_s$, $\icard(C) = \ulim_{s \in S}m_s = m$. Thus, part (iii) follows immediately from Theorem \ref{thm-internal-cardinality}.

By definition, $g^m = \ulim_{s \in S}g_s^{m_s} = 1$. Suppose that there is a hypernatual $h = \ulim_{s \in S}h_s$ for some natural numbers $h_s \in \bN$ such that $g^h = 1$ and $h < m$. Thus
\begin{align*}
g^h = \ulim_{s \in S}g_s^{h_s} = 1,
\end{align*}
which implies that $g_s^{h_s} = 1$ for $\cD$-almost all $s\in S$. Since $m_s$ is the order of $g_s$, we deduce that $m_s \le h_s$, and thus $m \le h$, a contradiction. Therefore part (i) follows immediately.

For part (ii), write $k = \ulim_{s \in S}k_s$ and $\ell = \ulim_{s \in S}\ell_s$ for some integers $k_s, \ell_s$. Suppose that
\begin{align*}
g^k = \ulim_{s \in S}g_s^{k_s} = \ulim_{s \in S}g_s^{\ell_s} = g^{\ell},
\end{align*}
and thus $g_s^{k_s} = g_s^{\ell_s}$ for $\cD$-almost all $s \in S$. Since $C_s$ is a finite cyclic group of order $m_s$, we deduce that $k_s \equiv \ell_s \pmod{m_s}$ for $\cD$-almost all $s \in S$. Thus \L{}o\'s' theorem implies that $k \equiv \ell \pmod{m}$ in $\bZ^{\#}$.

The reverse implication of part (ii) is obvious, and thus part (ii) follows immediately.

Part (iv) follows immediately from parts (i) and (iii).

\end{proof}

Lemma \ref{lem-ultra-cyclic-groups} motivates the following notion.
\begin{definition}
\label{def-hyper-cyclic}
(ultra-cyclic groups)

Keep the same notation as above. The group $C$ is called an \textbf{ultra-cyclic group of ultra-order $m$}.

A subgroup $H$ of $C$ is called an \textbf{ultra-cyclic subgroup of $C$} if $H$ is itself ultra-cyclic.

\end{definition}

Note that the ultra-order of $C$ is precisely its internal cardinality.

\begin{lemma}
\label{lem-ultra-orders-of-elements}

Let $C = \prod_{s \in S}C_s/\cD$ be an ultra-cyclic group of ultra-order $m = \ulim_{s \in S}m_s \in \bN^{\#}$, where $C_s = \langle g_s \rangle$ is a finite cyclic group with generator $g_s$ of order $m_s \in \bN$ for $\cD$-almost all $s \in S$. Let $g = \ulim_{s\in S}g_s \in C$ be an ultra-generator of $C$, i.e., $C = \langle g \rangle_u$. Then

Then
\begin{itemize}

\item [(i)] For any element of $h = g^a$ of $C$ for some positive hyperinteger $a = \ulim_{s\in S}a_s\in \bZ^{\#}_{>0}$ with positive integers $a_s$, $H := \langle h \rangle_u = \{h^b \; : \; b \in \bZ^{\#}\}$ is an ultra-cyclic subgroup of $C$ whose ultra-order is $m/\gcd(a, m) \in \bN^{\#}$. In particular, the ultra-order of $H$ divides $m$ in $\bZ^{\#}$.

\item [(ii)] Conversely, for any ultra-cyclic subgroup $H$ of $C$, there exists an element $h$ of $C$ such that $H = \langle h \rangle_u$.

\end{itemize}

\end{lemma}

\begin{proof}

For part (i), we see that $H_s := \langle g_s^{a_s} \rangle$ is a finite cyclic subgroup of $C_s$ of order $m_s/\gcd(a_s, m_s)$. We prove that $H = \prod_{s \in S}H_s/\cD$. Indeed, since $h = g^a = \ulim_{s \in S}g_s^{a_s}$, we see immediately that $H \subseteq \prod_{s \in S}H_s/\cD$.

Take an arbitrary element $\ulim_{s \in S}\ell_s \in \prod_{s \in S}H_s/\cD$. Then $\ell_s = g_s^{b_sa_s}$ for some $b_s\in \bZ$. Thus
\begin{align*}
\ulim_{s\in S}\ell_s = \ulim_{s \in S}g_s^{b_sa_s} = g^{ab} = (g^a)^b = h^b \in \langle h\rangle_u = H,
\end{align*}
where $b = \ulim_{s \in S}b_s \in \bZ^{\#}$. Thus $H = \prod_{s \in S}H_s/\cD$, which implies that $H$ is an ultra-cyclic subgroup of $C$. Furthermore we know that $H_s$ is a finite cyclic group of order $m_s/\gcd(a_s, m_s)$, and it thus follows from Lemma \ref{lem-gcd-of-hyperintegers} that $H$ is of ultra-order $\ulim_{s \in S}m_s/\gcd(a_s, m_s) = m/\gcd(a, m)$, which proves part (i).

For part (ii), write $H = \prod_{s \in S}H_s/\cD$, where $H_s$ are finite cyclic groups. Since $H$ is a subgroup of $C$, $H_s$ is a subgroup of $C_s$, and thus $H_s = \langle g_s^{a_s}\rangle$ for some integer $a_s$. Repeating the same arguments as above, we deduce that $H = \langle g^a \rangle_u$, where $a = \ulim_{s \in S}a_s \in \bZ^{\#}$.

\end{proof}

The following notion is motivated by the above proposition.

\begin{definition}
\label{def-ultra-order-of-elements}

Let $C = \ulim_{s \in S} C_s$ be an ultra-cyclic group of ultra-order $m = \ulim_{s \in S}m_s$, where $m_s$ is the order of finite cyclic group $C_s$. Let $h = \ulim_{s \in S}h_s$ be an element in $C$ for some elements $h_s \in C_s$. The \textbf{ultra-order} of $h$ is defined to be the ultra-order of the ultra-cyclic subgroup $\langle h \rangle_u$ of $C$ that is ultra-generated by $h$.

\end{definition}

The following is immediate from Lemma \ref{lem-ultra-orders-of-elements}.

\begin{corollary}
\label{cor-ultra-order-of-element}

Let $C$ be an ultra-cyclic group ultra-generated by an element $g$ of ultra-order $m = \ulim_{s \in S}m_s$, i.e., $C = \langle g \rangle_u$, where $g = \ulim_{s \in S}g_s$ and $C_s = \langle g_s \rangle$ is a finite cyclic group with generator $g_s$ of order $m_s$ for $\cD$-almost all $s \in S$. For any hyperinteger $a \in \bZ^{\#}$, the ultra-order of $g^a$ is $m/\gcd(a, m)$.

\end{corollary}

\begin{lemma}
\label{lem-divisibility-of-ultra-order}

Let $C$ be an ultra-cyclic group ultra-generated by an element $g$ of ultra-order $m = \ulim_{s \in S}m_s$, i.e., $C = \langle g \rangle_u$, where $g = \ulim_{s \in S}g_s$ and $C_s = \langle g_s \rangle$ is a finite cyclic group with generator $g_s$ of order $m_s$ for $\cD$-almost all $s \in S$. Let $h$ be an element of $C$ whose ultra-order is $d = \ulim_{s\in S}d_s\in \bN^{\#}$ for some natural numbers $d_s$. If $h^n = 1$ for some hyperinteger $n = \ulim_{s\in S}n_s \in \bZ^{\#}$ with integers $n_s$, then $d$ divides $n$ in $\bZ^{\#}$.

\end{lemma}

\begin{proof}

We can write $h = g^a$ for some hypernatural number $a = \ulim_{s\in S}a_s$, where the $a_s$ are certain natural numbers. By assumption, the ultra-cyclic subgroup $H$ of $C$ ultra-generated by $h = g^a$, say $H = \langle g^a \rangle_u$ is of the form
\begin{align*}
H = \prod_{s \in S}H_s/\cD,
\end{align*}
where $H_s = \langle g_s^{a_s}\rangle$ is a finite cyclic subgroup of $C_s$ of order $d_s$ for $\cD$-almost all $s \in S$. Thus the order of $g_s^{a_s}$ in $C_s$ is $d_s$ for $\cD$-almost all $s \in S$. Since $h^n = g^{an} = \ulim_{s \in S}(g_s^{a_s})^{n_s} = 1$, we deduce that $(g_s^{a_s})^{n_s} = 1$ for $\cD$-almost all $s \in S$, and thus since $g_s^{a_s}$ has order $d_s$, $d_s$ divides $n_s$ in $\bZ$ for $\cD$-almost all $s \in S$. Therefore $d = \ulim_{s \in S}d_s$ divides $n = \ulim_{s \in S}n_s$ in $\bZ^{\#}$.

\end{proof}

\begin{theorem}
\label{thm-uniqueness-of-ultra-cyclic-subgroups}

Let $C = \ulim_{s \in S}C_s = \langle g \rangle_u$ be an ultra-cyclic group of ultra-order $m = \ulim_{s \in S}m_s \in \bN_{\#}$, where $C_s =\langle g_s\rangle$ is a finite cyclic group of order $m_s$ generated by $g_s$ for each $s \in S$, where $g = \ulim_{s \in S}g_s$. For any hyperinteger $k \in \bZ^{\#}$, $\langle g^k \rangle_u = \langle g^{\gcd(k, m)} \rangle_u$. In particular, any ultra-cyclic subgroup of $C$ is of the form $\langle g^d\rangle_u$ for some positive divisor $d \in \bZ^{\#}$ of $m$ in $\bZ^{\#}$. Different values of $d$ give ultra-cyclic subgroups with different ultra-orders, and thus there exists a unique ultra-cyclic subgroup of $C$ having a given ultra-order.

\end{theorem}

\begin{proof}

 Take an arbitrary hyperinteger $k \in \bZ^{\#}$. We first show that $g^k$ and $g^{\gcd(k, m)}$ are both ultra-powers of each other.

Since $\gcd(k, m)$ divides $k$, $g^k$ is certainly a ultra-power of $g^{\gcd(k, m)}$, and thus $\langle g^k \rangle_u \subseteq \langle g^{\gcd(k, m)}\rangle_u$. By the nonstandard B\'ezout's identity (see Theorem \ref{thm-Bezout}), one can write
\begin{align*}
\gcd(k, m) = xk + ym
\end{align*}
for some hyperintegers $x, y \in \bZ^{\#}$. By part (i) of Lemma \ref{lem-ultra-cyclic-groups},
\begin{align*}
g^{\gcd(k, m)} = (g^k)^x(g^m)^y = (g^k)^x,
\end{align*}
which implies that $\langle g^k\rangle_u = \langle g^{\gcd(k, m)}\rangle_u$.

We have shown in Lemma \ref{lem-ultra-orders-of-elements} above that any ultra-cyclic subgroup $H$ of $C$ is ultra-generated by $g^k$ for some hyperinteger $k$. From the above discussion, $H = \langle g^d \rangle_u$, where $d = \gcd(k, m)$ is a positive divisor of $m$ in $\bZ^{\#}$. By Corollary \ref{cor-ultra-order-of-element}, $g^d$ has ultra-order $m/d$, and thus the ultra-order of $H = \langle g^d \rangle_u$ is $m/d$. Therefore if $d, d'$ are distinct positive divisors of $m$ in $\bZ^{\#}$, then since the two ultra-cyclic subgroups $\langle g^d\rangle_u$, $\langle g^{d'}\rangle_u$ have different ultra-orders $m/d$, $m/d'$, respectively, we deduce that $\langle g^d \rangle_u \ne \langle g^{d'} \rangle_u$.

\end{proof}

\begin{example}
\label{exam-multiplicative-group-of-ultra-Galois-fields}

Let $\kappa = \ulim_{s \in S}q_s$ be an ultra-power of an ultra-prime in $\bN^{\#}$, where the $q_s = p_s^{d_s}$ for some primes $p_s$ and some natural numbers $d_s$ for $\cD$-almost all $s \in S$. By Corollary \ref{cor-uniqueness-of-ultra-Galois-fields}, for each hypernatural number $m = \ulim_{s \in S}m_s$, $\GF(\kappa)$ has a unique ultra-field extension $\GF(\kappa^m)$ of ultre-degree $m$ over $\GF(\kappa)$ that can be written in the form $\prod_{s \in S}\bF_{{q_s}^{m_s}}/\cD$.

By \L{}o\'s' theorem, the multiplicative group $\GF(\kappa^m)^{\times}$ consisting of all nonzero elements of $\GF(\kappa^m)$, is the ultraproduct $\prod_{s\in S}\bF_{q_s^{m_s}}^{\times}/\cD$. We know that $\bF_{q_s^{m_s}}^{\times}$ is a finite cyclic group of order $q_s^{m_s} - 1$. Thus $\GF(\kappa^m)^{\times}$ is an ultra-cyclic group of ultra-order $\ulim_{s \in S}(q_s^{m_s} - 1) = \beta^m - 1$.

\end{example}

\begin{example}
\label{exam-ultra-roots-of-unity}

Let $F = \prod_{s \in S}F_s/\cD$ be an ultra-field for some fields $F_s$. Let $n = \ulim_{s \in S}n_s$ be a hypernatural number for some natural numbers $n_s$ such that $n_s$ is relatively prime to the characteristic of $F_s$ for $\cD$-almost all $s \in S$. For each $s \in S$, let $\mu_{n_s}(F_s)$ denote the group of $n_s$-th roots of unity in an algebraic closure $F_s^{\alg}$ of $F_s$, i.e., the group of zeros of the polynomial $f_s(x) = x^{n_s} - 1 \in F_s[x]$ in $F_s^{\alg}$. It is known that $\mu_{n_s}(F_s)$ is a finite cyclic group of order $n_s$.

Let $f = \ulim_{s \in S}f_s$ be an ultra-polynomial in the ultra-hull $\cU(F[x])$. One can write
\begin{align*}
f(x) = \ulim_{s \in S}(x^{n_s} - 1) = x^n - 1.
\end{align*}

An \textbf{$n$-th ultra-root of unity in $F^{\alg_u}$} is defined to be a zero of $f$ in an ultra-algebraic closure $F^{\alg_u} = \prod_{s \in S}F_s^{\alg}/\cD$, i.e., an element $a = \ulim_{s \in S}a_s \in F^{\alg_u}$ for some elements $a_s \in F_s^{\alg}$ such that $a_s$ is an $n_s$-th root of unity in $F_s^{\alg}$. Let $\mu_n(F)$ denote the set of $n$-th ultra-roots of unity in $F^{\alg_u}$. Then one can immediately see from Corollary \ref{cor-upper-bound-for-zero-set-of-ultra-polynomial} that
\begin{align*}
\mu_n(F) = \prod_{s \in S}\mu_{n_s}(F_s)/\cD.
\end{align*}

Since $\mu_{n_s}(F_s)$ is a finite cyclic group of order $n_s$, $\mu_n(F)$ is an ultra-cyclic group of ultra-order $n = \ulim_{s \in S}n_s$.

Note that if $n$ is a positive integer in $\bZ$, $\mu_n(F)$ is the usual group of $n$-th roots of unity.

\end{example}

We maintain the same notation as in Examples \ref{exam-multiplicative-group-of-ultra-Galois-fields} and \ref{exam-ultra-roots-of-unity}. Suppose now that $n$ divides $\kappa^m -1 = \ulim_{s\in S}(q_s^m - 1)$ in $\bZ^{\#}$. Then $n_s$ divides $q_s^m - 1$, and thus $n_s$ is relatively prime to $p_s$, the characteristic of $\bF_{q_s^{m_s}}$. Thus $\mu_n(\GF(\kappa^m))$ is an ultra-cyclic group of ultra-order $n$.

Since $n$ divides $\kappa^m - 1$ which is the ultra-order of the ultra-cyclic group $\GF(\kappa^m)^{\times}$, it follows from Theorem \ref{thm-uniqueness-of-ultra-cyclic-subgroups}, there exists a unique ultra-cyclic subgroup, say $H$ of $\GF(\kappa^m)^{\times}$ such that $H$ has ultra-order $n$. For every element $z$ of $H$, we know from part (i) of Lemma \ref{lem-ultra-cyclic-groups} that $z^n - 1 = 0$, and thus $H$ is precisely the ultra-cyclic group $\mu_n(\GF(\kappa^m))$ of $n$-th ultra-roots of unity. Therefore we obtain the following.

\begin{corollary}
\label{cor-nth-ultra-roots-of-unity-in-Galois-ultra-fields}

In order that $\mu_n(\GF(\kappa^m))$ is an ultra-cyclic subgroup of $\GF(\kappa^m)^{\times}$ of ultra-order exactly $n$, it is necessary and sufficient that $n$ divides $\kappa^m - 1$.

\end{corollary}

\subsection{Power residue symbol}

In this section, we describe power residue symbol for $\bA$.

\begin{lemma}
\label{lem-U/PU=A/PA-iso}

Let $P$ be an irreducible polynomial of degree $d \in \bZ_{>0}$. Then
\begin{itemize}

\item [(i)] For any ultra-polynomial $\alpha \in \cU(\bA)$, there exists a unique polynomial $R$ of degree less than $\deg(P) = d$ in $\bA$ such that
\begin{align*}
\alpha \equiv R \pmod{P\cU(\bA)}.
\end{align*}

\item [(ii)] In particular, $P\cU(\bA)$ is a maximal ideal of $\cU(\bA)$, and the map $\psi : \bA/P\bA \to \cU(\bA)/P\cU(\bA)$ that sends $R \pmod{P\bA}$ to $R \pmod{P\cU(\bA)}$ is a field isomorphism. Furthermore, the inverse isomorphism $\rho$ of $\psi$ is described as follows: for each ultra-polynomial $\alpha \in \cU(\bA)$, let $\epsilon(\alpha) := R$ be the unique polynomial of degree less than $d$ in part (i) such that $\alpha \equiv R \pmod{P\cU(\bA)}$. Then $\rho(\alpha \pmod{P\cU(\bA)}) = (\epsilon(\alpha) \pmod{P\bA})$.

\end{itemize}

\end{lemma}

\begin{proof}

Let $\alpha = \ulim_{s\in S}\alpha_s$ be an ultra-polynomial of ultra-degree $d = \ulim_{s\in S}d_s \in \bZ^{\#}_{>0}$ in $\cU(\bA)$, where the $\alpha_s$ are polynomials of degrees $d_s \in \bZ_{>0}$ for $\cD$-almost all $s \in S$. Since $P$ is a polynomial of degree $d$ in $\bA$, one can write $P = \ulim_{s\in S}P_s$, where $P_s$ is a polynomial of degree $d$ for $\cD$-almost all $s \in S$. By \L{}o\'s' theorem, the $P_s$ are irreducible for $\cD$-almost all $s \in S$.

For $\cD$-almost all $s \in S$, it follows from the Euclidean's algorithm that there exist polynomials $Q_s, R_s$ in $\bA$ such that
\begin{align}
\label{e-U/PU=A/PA-lem}
\alpha_s = P_sQ_s + R_s,
\end{align}
where $R_s$ is either $0$ or $\deg(R_s) < \deg(P_s) = d$. Set $R = \ulim_{s\in S}R_s \in \cU(\bA)$. Since $\deg(R_s) < d$ for $\cD$-almost all $s \in S$, $R$ must be a polynomial in $\bA$ of degree less than $d$.

We know from (\ref{e-U/PU=A/PA-lem}) that $\alpha \equiv R \pmod{P\cU(\bA)}$. If there is another polynomial $R'$ of degree less than $d$ such that $\alpha \equiv R'\pmod{P\cU(\bA)}$, then $R \equiv R' \pmod{P\cU(\bA)}$, and thus $R - R' = PQ$ for some ultra-polynomial $Q \in \cU(\bA)$. Since $R-R'\in \bA$ and $P\in \bA$, $Q$ must belong in $\bA$. If $R - R' \ne 0$, then
\begin{align*}
\deg(P) = d \le \deg(R - R') < d,
\end{align*}
which is a contradiction. Thus $R = R'$, and so the first assertion follows.

For the last assertion, it is not difficult to see that $\psi$ is a homomorphism. Since $P$ is a prime in $\bA$, $\bA/P\bA$ is a field. Furthermore by Proposition \ref{prop-notion-section-algebraic-structures-of-ultraproducts} one can write $\cU(\bA)/P\cU(\bA) = \prod_{s\in S}(\bA_s/P_s\bA_s)/\cD$. Since the rings $\bA_s/P_s\bA_s$ are fields for $\cD$-almost all $s \in S$, we deduce that $\cU(\bA)/P\cU(\bA)$ is an ultra-field which in particular implies that it is a field. Thus $P\cU(\bA)$ is a maximal ideal of $\cU(\bA)$, and hence $\psi$ is an injective homomorphism of fields. The surjectivity of $\psi$ follows immediately from the first part, and thus $\psi$ is an isomorphism.

We see that
\begin{align*}
\psi \circ \rho(\alpha \pmod{P\cU(\bA)}) &= \psi(\rho(\alpha \pmod{P\cU(\bA)})) \\
&= \psi(\epsilon(\alpha) \pmod{P\bA})\\
&= \epsilon(\alpha) \pmod{P\cU(\bA)} \\
&= \alpha \pmod{P\cU(\bA)}
\end{align*}
since $\epsilon(\alpha) \equiv \alpha \pmod{P\cU(\bA)}$. Thus the last assertion follows immediately.

\end{proof}

For the rest of this section, we fix an irreducible polynomial $P$ of degree $d \in \bZ_{>0}$ in $\bA$.By Lemma \ref{lem-U/PU=A/PA-iso}, $\bA/P\bA$ is isomorphic to $\cU(\bA)/P\cU(\bA)$ as fields.

We obtain the following sequence of homomorphisms
\begin{align*}
\GF(\kappa) \xrightarrow{(1)} \bA \xrightarrow{(2)} \cU(\bA) \xrightarrow{(3)} \cU(\bA)/ P\cU(\bA),
\end{align*}
where $(1)$ is the natural embedding of $\GF(\kappa)$ into $\bA$, $(2)$ is also an embedding of $\bA$ into $\cU(\bA)$, following \L{}o\'s' theorem, and $(3)$ is the canonical homomorphism that sends each element $f\in \cU(\bA)$ to $f$ modulo $P\cU(\bA)$.

We also have a sequence of homomorphisms
\begin{align*}
\GF(\kappa) \to \bA \to \bA/P\bA \cong \cU(\bA)/ P\cU(\bA).
\end{align*}

Thus one obtains an injective homomorphism of fields
\begin{align*}
\GF(\kappa) \rightarrow \cU(\bA)/ P\cU(\bA) \cong \bA/P\bA,
\end{align*}
from which we can view $\cU(\bA)/P\cU(\bA)$ as an algebraic extension of $\GF(\kappa)$. Since $P$ is an irreducible polynomial of degree $d$ in $\bA$, $\bA/P\bA$ is of degree $d$ over $\GF(\kappa)$, and thus $\cU(\bA)/ P\cU(\bA)$ is an algebraic extension of $\GF(\kappa)$ of degree $d$.

We know from Corollary \ref{cor-uniqueness-of-ultra-Galois-fields} that $\GF(\kappa)$ has a unique extension $\GF(\kappa^d)$ of degree $d$ over $\GF(\kappa)$, up to ultra-isomorphism, of the form
\begin{align*}
\GF(\kappa^d) = \prod_{s \in S}\bF_{q_s^d}/\cD.
\end{align*}
Thus up to ultra-isomorphism, we can identify
\begin{align*}
\cU(\bA)/ P\cU(\bA) =\GF(\kappa^d).
\end{align*}
Therefore the multiplicative group of nonzero element in $\cU(\bA)/ P\cU(\bA)$ is $\GF(\kappa^d)^\times$ which is the multiplicative group of nonzero elements in $\GF(\kappa^d)$. In fact, following the proof of Lemma \ref{lem-U/PU=A/PA-iso}, we can directly deduce that $(\cU(\bA)/ P\cU(\bA))^{\times}$ can be identified as $\GF(\kappa^d)^\times$, up to ultra-isomorphism.

We know from Corollary \ref{cor-uniqueness-of-ultra-Galois-fields} that
\begin{align*}
\GF(\kappa^d) = \prod_{s \in S} \bF_{q_s^d}/\cD,
\end{align*}
thus we deduce from \L{}o\'s' theorem that
\begin{align*}
\GF(\kappa^d)^\times = \prod_{s \in S} \bF_{q_s^d}^\times / \cD.
\end{align*}
Since $\bF_{q_s^d}^\times$ is a finite cyclic group of order $q_s^d-1$ and $\kappa^d = \ulim_{s\in S}q_s^d$, we know from Subsection \ref{subsec-ultra-cyclic-groups} that $\GF(\kappa^d)^\times$ is an ultra-cyclic group of ultra order $\kappa^d-1$.

Thus $\left( \cU(\bA)/ P\cU(\bA) \right)^\times = \GF(\kappa^d)^{\times}$ is an ultra-cyclic group of ultra-order $\kappa^d-1$.

Let $n=\ulim_{s\in S}n_s$ be a hypernatural number in $\bN^{\#}$ such that $n$ divides $\kappa-1$ in $\bN^{\#}$, i.e., $n_s$ divides $q_s-1$ for $\cD$-almost all $s \in S$.

Take any polynomial $\alpha \in \bA$ such that $\alpha$ and $P$ are relatively prime in $\bA$. We know that the element  $\alpha^{\frac{\kappa^d-1}{n}}$ is an ultra-polynomial $\cU(\bA)$ such that
\begin{align*}
\left( \alpha^{\frac{\kappa^d-1}{n}} \right)^n = \alpha^{\kappa^d-1} \equiv 1 \pmod{P\cU(\bA)},
\end{align*}
following the analogue of Fermat's little theorem \ref{thm-Fermat-little-theorem-for-A}. By Lemma \ref{lem-divisibility-of-ultra-order}, the ultra-order of $\alpha^{\frac{\kappa^d-1}{n}} \pmod{P\cU(\bA)}$ divides $n$ in $\bN^{\#}$, which implies that the ultra-order of $ \alpha^{\frac{\kappa^d-1}{n}}\pmod{P\cU(\bA)}$ divides $\kappa-1$. Since the multiplicative group $\GF(\kappa)^{\times}$ is an ultra-cyclic subgroup of $\GF(\kappa^d)^{\times}$ order $\kappa - 1$, we see from Theorem \ref{thm-uniqueness-of-ultra-cyclic-subgroups}, there exists a unique ultra-cyclic subgroup of $\GF(\kappa)^{\times}$ whose ultra-order is equal to the ultra-order of $\alpha^{\frac{\kappa^d-1}{n}} \pmod{P\cU(\bA)}$. In other words, by Lemma \ref{lem-ultra-orders-of-elements}, there exists a unique element $\beta \in \GF(\kappa)^\times$ such that
\begin{align*}
\alpha^{\frac{\kappa^d-1}{n}} \equiv \beta \pmod{P\cU(\bA)}.
\end{align*}

 The above discussion motivates the following.
 \begin{definition} ($n$-th power residue symbol)
\label{def-nth-power-residue-symbol}

 We maintain the notation as above. Let $\left( \dfrac{\alpha}{P} \right)_n$ be the unique element of the ultra-cyclic group $\GF(\kappa)^\times$ such that
\begin{align*}
\alpha^{\frac{\kappa^d-1}{n}} \equiv \left( \dfrac{\alpha}{P} \right)_n \pmod{P\cU(\bA)}.
\end{align*}

 If $P$ divides $\alpha$, we simply define $\left( \dfrac{\alpha}{P} \right)_n =0$. We call $\left( \dfrac{\alpha}{P} \right)_n$ the \textbf{$n$-th power residue symbol}.

 \end{definition}

\begin{remark}
\label{rem-power-residue-symbol-are-nth-roots-of-unity}

The above discussion implies that the ultra-order of $\left( \dfrac{\alpha}{P} \right)_n$ divides $n$, and so
\begin{align*}
\left( \dfrac{\alpha}{P} \right)_n^n = 1.
\end{align*}

If $\mu_n(\GF(\kappa))$ denotes the $n$-th ultra-roots of unity in $\GF(\kappa)$, then Corollary \ref{cor-nth-ultra-roots-of-unity-in-Galois-ultra-fields} implies that $\mu_n(\GF(\kappa))$ is an ultra-cyclic subgroup of $\GF(\kappa)^{\times}$ of ultra-order exactly $n$. By the above equation, $\left( \dfrac{\alpha}{P} \right)_n$ belongs in $\mu_n(\GF(\kappa))$, i.e., it is an $n$-th ultra-root of unity in $\GF(\kappa)$.

\end{remark}

\begin{remark}
\label{rem-power-residue-symbol-recovers-CP-PRS}

\begin{itemize}

\item[]

\item [(i)] In this remark, we prove that our notion of power residue symbol for $F[t]$ recovers that of Clark--Pollack \cite{CP} when $F$ is an ultra-finite field and $n$ is a positive integer that divides $\kappa - 1$. We will use the notation $\left(\dfrac{\cdot}{\cdot}\right)_{n, CP}$ for the $n$-power residue symbol in \cite{CP} that we recall in Section \ref{sec-Introduction}.

By Corollary \ref{cor-quasifinite-ultra-products}, $\GF(\kappa) = \prod_{s \in S}\bF_{q_s}/\cD$ is a quasi-finite field. Since the Frobenius map $\cF_s : \bF_{q_s}^{\alg} \to \bF_{q_s}^{\alg}$ that sends each element $x$ to $x^{q_s}$, is a topological generator of the Galois group $\Gal(\bF_{q_s}^{\alg}/\bF_{q_s})$, we see from Theorem \ref{thm-roots-of-ultra-polynomials-in-ultra-Galois-fields} that the ultra-Frobenius map $\cF : \GF(\kappa)^{\alg} \to \GF(\kappa)^{\alg}$ defined by $\cF(x) = x^{\kappa} = \ulim_{s\in S}x^{q_s}$ is a topological generator of the Galois group $\Gal(\GF(\kappa)^{\alg}/\GF(\kappa))$ (see Remark \ref{rem-ultra-Frobenius-map-definition} and Nguyen \cite{nguyen-2023-2nd-paper} for a detailed proof). Thus the pair $\GF(\kappa^d), \cF^d)$ is a quasi-finite field, following the notation used in Section \ref{sec-Introduction}.

Take any polynomials $\alpha, P \in \bA = \GF(\kappa)[t]$ such that $P$ is an irreducible polynomial of degree $d$ that does not divide $\alpha$. Identifying $\cU(\bA)/P\cU(\bA)$ with $\GF(\kappa^d)$, we see that
\begin{align*}
\left(\dfrac{\alpha}{P}\right)_{n, CP} = \Gamma_{\GF(\kappa^d), n}(\alpha \pmod{P\cU(\bA)}),
\end{align*}
where $\Gamma_{\GF(\kappa^d), n}$ is the map defined as in (\ref{eqn-Gamma-map-introduction-in-Serre}).

We have
\begin{align*}
\Gamma_{\GF(\kappa^d), n}(\alpha \pmod{P\cU(\bA)}) &= \dfrac{\cF^d((\alpha \pmod{P\cU(\bA)})^{1/n})}{(\alpha \pmod{P\cU(\bA)})^{1/n}} \\
&= \dfrac{(\alpha \pmod{P\cU(\bA)})^{\kappa^d/n}}{(\alpha \pmod{P\cU(\bA)})^{1/n}} \\
&= (\alpha \pmod{P\cU(\bA)})^{(\kappa^d - 1)/n} \\
&= \alpha^{(\kappa^d - 1)/n} \pmod{P\cU(\bA)} \; \; (\text{since $n$ divides $\kappa - 1$}),
\end{align*}
and it thus follows from Definition \ref{def-nth-power-residue-symbol} that
\begin{align*}
\left(\dfrac{\alpha}{P}\right)_{n, CP} = \alpha^{(\kappa^d - 1)/n} \pmod{P\cU(\bA)} = \left(\dfrac{\alpha}{P}\right)_n.
\end{align*}

Thus our notion of power residue symbol recovers that of Clark--Pollack \cite{CP} whenever $F$ is an ultra-finite field and $n$ is a positive integer dividing $\kappa - 1$.

\item[(ii)] In this remark, we give another proof to the first part of Theorem \ref{thm-Clark-Pollack}.

From the first remark, we see that
\begin{align}
\label{eq1-Gamma-map-in-1st-part-of-CP-thm}
\Gamma_{\GF(\kappa^d), n}(x) = x^{\frac{\kappa^d - 1}{n}}
\end{align}
for each $x \in \GF(\kappa^d)$. In order to give another proof to part (i) of Theorem \ref{thm-Clark-Pollack}, we make use of Lemma 11 in Clark--Pollack \cite{CP} that expresses the norm $\Norm_{\GF(\kappa^d)/\GF(\kappa)}$ in terms of resultant, and thus it suffices to prove that
\begin{align}
\label{eqn2-1st-part-of-CP-thm}
\Gamma_{\GF(\kappa^d), n} = \Gamma_{\GF(\kappa), n} \circ \Norm_{\GF(\kappa^d)/\GF(\kappa)},
\end{align}
where $\Norm_{\GF(\kappa^d)/\GF(\kappa)}$ is the norm map of $\GF(\kappa^d)$ over $\GF(\kappa)$.

Using Theorem \ref{thm-roots-of-ultra-polynomials-in-ultra-Galois-fields}, we deduce that for each $x \in \GF(\kappa^d)$,
\begin{align*}
\Norm_{\GF(\kappa^d)/\GF(\kappa)}(x) = x\cdot x^{\kappa} \cdots x^{\kappa^{d - 1}} = x^{\frac{\kappa^d - 1}{\kappa - 1}},
\end{align*}
and thus
\begin{align*}
\Gamma_{\GF(\kappa), n} \circ \Norm_{\GF(\kappa^d)/\GF(\kappa)}(x) &= \Gamma_{\GF(\kappa), n}(x^{\frac{\kappa^d - 1}{\kappa - 1}})\\
&= \left(x^{\frac{\kappa^d - 1}{\kappa - 1}}\right)^{\frac{\kappa - 1}{n}} \\
&= x^{\frac{\kappa^d - 1}{n}} \\
&= \Gamma_{\GF(\kappa^d), n}(x).
\end{align*}
Therefore (\ref{eqn2-1st-part-of-CP-thm}) holds, and so the first part of Theorem \ref{thm-Clark-Pollack} follows.

\end{itemize}

\end{remark}

\begin{lemma}
\label{lem-ultra-powers-mod-PU}

Let $m = \ulim_{s\in S}m_s$ be a positive hyperinteger in $\bZ^{\#}_{>0}$, where the $m_s$ are positive integers in $\bZ$ for $\cD$-almost all $s \in S$. For any ultra-polynomial $\lambda \in \cU(\bA)$,
\begin{align*}
(\lambda \pmod{P\cU(\bA)})^n = (\lambda^n \pmod{P\cU(\bA)}).
\end{align*}
holds in the field $\cU(\bA)/P\cU(\bA)$.

\end{lemma}

\begin{remark}

Note that the left-hand side of the equation in Lemma \ref{lem-ultra-powers-mod-PU} is well-defined because $\cU(\bA)/P\cU(bA)$ is an ultra-field of the form $\prod_{s\in S}\bA_s/P_s\bA_s$, where $P = \ulim_{s\in S}P_s$ and the $P_s$ are irreducible polynomials of degree $d$ for $\cD$-almost all $s \in S$ (see Lemma \ref{lem-U/PU=A/PA-iso}). In fact, if $\lambda = \ulim_{s\in S}\lambda_s$ for some polynomials $\lambda_s \in \bA_s$, then
\begin{align*}
(\lambda \pmod{P\cU(\bA)})^n = \ulim_{s\in S} \left(\lambda_s \pmod{P_s\bA_s}\right)^{n_s}.
\end{align*}

\end{remark}

\begin{proof}

We maintain the same notation as in the remark above. We know that
\begin{align*}
\left(\lambda_s \pmod{P_s\bA_s}\right)^{n_s} = \left(\lambda_s^{n_s} \pmod{P_s\bA_s}\right)
\end{align*}
holds in the field $\bA_s/P_s\bA_s$ for $\cD$-almost all $s \in S$. Thus
\begin{align*}
(\lambda \pmod{P\cU(\bA)})^n &= \ulim_{s\in S} \left(\lambda_s \pmod{P_s\bA_s}\right)^{n_s} \\
&= \ulim_{s\in S}\left(\lambda_s^{n_s} \pmod{P_s\bA_s}\right) \\
&= \left(\lambda^n \pmod{P\cU(\bA)}\right),
\end{align*}
which proves the lemma.

\end{proof}

\begin{proposition}
\label{prop-properties-of-power-residue-symbol}

The $n$-th power residue symbol has the following properties: for any $\alpha, \lambda \in \bA$,
\begin{itemize}

\item[(i)] $\left( \dfrac{\alpha}{P} \right)_n = \left( \dfrac{\lambda}{P} \right)_n$ iff $\alpha \equiv \lambda \pmod{P\bA}$.

\item[(ii)] $\left( \dfrac{\alpha\lambda}{P} \right)_n = \left( \dfrac{\alpha}{P} \right)_n \left( \dfrac{\lambda}{P}\right)_n$.

\item[(iii)] $\left( \dfrac{\alpha}{P} \right)_n = 1$ if and only if the equation $x^n \equiv \alpha \pmod{P\cU(\bA)} $ is solvable in $\bA$. If $n$ is a positive integer, then $\left( \dfrac{\alpha}{P} \right)_n = 1$ if and only if the equation $x^n \equiv \alpha \pmod{P\bA}$ is solvable.

\item[(iv)] Let $\zeta$ be an element in $\GF(\kappa)^\times$ of ultra-order dividing $n$. Then there exists an element $\alpha \in \bA$ such that
\begin{align*}
\left( \dfrac{\alpha}{P} \right)_n  = \zeta.
\end{align*}

\end{itemize}

\end{proposition}

\begin{proof}

Part (i) follows from the definition, and the fact that for any $\alpha, \lambda \in \bA$, $\alpha \equiv \lambda \pmod{P\bA}$ if and only if  $\alpha \equiv \lambda \pmod{P \cU (\bA)}$.

Part (ii) also follows from the definition, and the fact that if two constants in $\GF(\kappa)^\times$ are congruent modulo $P\bA$, then they are equal since $\deg(P)$ is positive.

We now prove part (iii). If $x^n \equiv \alpha \pmod{P\cU(\bA)} $ is solvable in $\bA$, then $b^n \equiv \alpha \pmod{P\cU(\bA)}$ for some $b \in \bA$. Thus by the analogue of Fermat little theorem \ref{thm-Fermat-little-theorem-for-A},
\begin{align*}
\alpha^{\frac{\kappa^d-1}{n}} \equiv (b^n)^{\frac{\kappa^d-1}{n}} \equiv b^{\kappa^d-1} \equiv 1 \pmod{P\cU(\bA)},
\end{align*}
and therefore $\left( \dfrac{\alpha}{P} \right)_n = 1$.

Suppose that $\alpha^{\frac{\kappa^d-1}{n}}  \equiv \left( \dfrac{\alpha}{P} \right)_n = 1 \pmod{P\cU(\bA)}$. Since $n$ divides $\kappa - 1$, we know from Corollary \ref{cor-nth-ultra-roots-of-unity-in-Galois-ultra-fields} that the group of $n$-th ultra-roots of unity, say $\mu_n(\GF(\kappa^d))$ is an ultra-cyclic subgroup of $\GF(\kappa^d)^{\times}$ of ultra-order exactly $n$. Let $\mu_{n_s}(\bF_{q_s^d})$ denote the unique cyclic subgroup of $\bF_{q_s^d}^{\times}$ of order $n_s$. Using Corollary \ref{cor-nth-ultra-roots-of-unity-in-Galois-ultra-fields}, we know that
\begin{align*}
\mu_n(\GF(\kappa^d)) = \prod_{s\in S}\mu_{n_s}(\bF_{q_s^d})/\cD.
\end{align*}

Consider the group homomorphism $\Psi$ from $(\cU(\bA) / P\cU(\bA))^\times \cong (\bA / P\bA)^\times \cong \GF(\kappa^d)^\times$ to itself defined by $\Psi(a)=a^n$ for all $a \in \GF(\kappa^d)^\times$. The kernel of $\Psi$ consists exactly all $n$-th ultra-roots of unity in $\GF(\kappa^d)^\times$, and thus $\text{ker}(\Psi) = \mu_n(\GF(\kappa^d))$.

The image of $\Psi$, denoted by $\Im(\Psi)$ consists all $n$-th ultra-powers in $\GF(\kappa^d)^\times$. By the isomorphism theorem (see \cite{lang-algebra}),
\begin{align*}
\GF(\kappa^d)^\times / \text{ker}(\Psi) \cong \Im(\Psi).
\end{align*}

We see that
\begin{align*}
\GF(\kappa^d)^\times / \text{ker}(\Psi) &= \GF(\kappa^d)^{\times}/\mu_n(\GF(\kappa^d)) \\
&=  \left(\prod_{s\in S} \bF_{q_s^d}^\times / \cD\right)/\left(\prod_{s\in S}\mu_{n_s}(\bF_{q_s^d})/\cD\right) \\
&\cong \prod_{s\in S} \left( \bF_{q_s^d}^\times / \mu_{n_s}(\bF_{q_s^d}) \right) / \cD \; \text{(see Proposition \ref{prop-notion-section-algebraic-structures-of-ultraproducts} or \cite[2.1.6, page 10]{schoutens})} \\
&\cong \Im(\Psi).
\end{align*}
So $\Im(\Psi)$ is an ultra-cyclic group of ultra-order $\ulim_{s\in S} \frac{q_s^d -1}{n_s} = \frac{\kappa^d-1}{n}$ since the cardinality of $\bF_{q_s^d}^\times / \mu_{n_s}(\bF_{q_s^d}) $ is $\frac{q_s^d -1}{n_s}$. In particular, the set of $n$-th ultra-powers in $\GF(\kappa^d)^\times \cong (\cU(\bA) / P\cU(\bA))^\times$ is an internal set of internal cardinality $\frac{\kappa^n-1}{n}$. For an $n$-th ultra-power $b^n$ in $\GF(\kappa^d)^\times$ for some $b \in \GF(\kappa^d)^\times$, it is easy to see that $x=b^n$ satisfies the equation
\begin{align*}
x^{\frac{\kappa^d-1}{n}} -1 =0 \in \GF(\kappa^d)^{\times},
\end{align*}
following the analogue of Fermat's little theorem \ref{thm-Fermat-little-theorem-for-A}. Since the ultra-degree of the ultra-polynomial $x^{\frac{\kappa^d-1}{n}} -1$ is $\frac{\kappa^d-1}{n}$, we know from Corollary \ref{cor-upper-bound-for-zero-set-of-ultra-polynomial} that the internal cardinality of its zero set in $\GF(\kappa^d)^\times$ is at most $\frac{\kappa^d-1}{n}$. We have shown that the internal cardinality of $n$-th ultra-powers in $\GF(\kappa^d)^\times$ is $\frac{\kappa^d-1}{n}$, and all $n$-th ultra-powers in $\GF(\kappa^d)\times$ belong in the zero set of $x^{\frac{\kappa^d-1}{n}} -1$ whose internal cardinality is at most $\frac{\kappa^d-1}{n}$. Thus it follows from Proposition \ref{prop-comparing-internal-sets-with-same-cardinality} that the zero set of $x^{\frac{\kappa^d-1}{n}} -1$ consists precisely of all $n$-th ultra-powers in $\GF(\kappa^d)^\times$.

Since $\alpha^{\frac{\kappa^d-1}{n}} -1 \equiv 0 \pmod{P\cU(\bA)}$, i.e., $\alpha \pmod{P\cU(\bA)}$ is a zero to the equation
\begin{align*}
x^{\frac{\kappa^d-1}{n}} -1=0 \; \text{in} \; \cU(\bA)/ P\cU(\bA) \cong \bA/P\bA \cong \GF(\kappa^d),
\end{align*}
$\alpha \pmod{P\cU(\bA)}$ is an $n$-th ultra-power in the field
 \begin{align*}
\cU(\bA)/ P\cU(\bA) \cong \GF(\kappa^d).
\end{align*}
Thus there exists $\beta \pmod{P\cU(\bA)} \in \cU(\bA)/P\cU(\bA)$ for some ultra-polynomial $\beta \in \cU(\bA)$ such that
\begin{align*}
\left(\alpha \pmod{P\cU(\bA)}\right) = \left(\beta \pmod{P\cU(\bA)}\right)^n
\end{align*}
holds in the field $\cU(\bA)/P\cU(\bA)$. By Lemma \ref{lem-ultra-powers-mod-PU}, we deduce that
\begin{align*}
\left(\beta \pmod{P\cU(\bA)}\right)^n = \left(\beta^n \pmod{P\cU(\bA)}\right)
\end{align*}
holds in the field $\cU(\bA)/P\cU(\bA)$. By Lemma \ref{lem-U/PU=A/PA-iso}, one can find a polynomial $\lambda$ of degree less than $d$ such that
\begin{align*}
\beta^n \equiv \lambda^n \pmod{P\cU(\bA)},
\end{align*}
and hence
\begin{align*}
\lambda^n \equiv \alpha \pmod{P\cU(\bA)},
\end{align*}
which proves the first assertion of part (iii).

If $n \in \bZ_{>0}$, then $\lambda^n-\alpha \in \bA$ and $\lambda^n-\alpha = PQ$ for some $Q \in \cU(\bA)$. Since $\lambda^n-\alpha \in \bA$ and $P \in \bA$, we deduce that $Q$ must belong in $\bA$, which implies that
\begin{align*}
\lambda^n \equiv \alpha \pmod{P\bA}.
\end{align*}
Thus part (iii) follows.

We now prove part (iv). Let $\rho : \cU(\bA)/P\cU(\bA) \to \bA/P\bA$ be the field isomorphism in Lemma \ref{lem-U/PU=A/PA-iso} that sends each element $\lambda \pmod{P\cU(\bA)}$ to $\epsilon(\lambda) \pmod{P\bA}$, where for each $\lambda \in \cU(\bA)$, $\epsilon(\lambda) \in \bA$ is the unique polynomial of degree less than $d$ in $\bA$ such that $\lambda \equiv \epsilon(\lambda) \pmod{P\cU(\bA)}$.

We define the map $\Sigma : (\cU(\bA)/P\cU(\bA))^{\times} \to \GF(\kappa)^{\times}$ by
\begin{align*}
\Sigma(\lambda \pmod{P\cU(\bA)}) = \left( \dfrac{\epsilon(\lambda)}{P}\right)_n
\end{align*}
for each $\lambda \in \cU(\bA)$ such that $P$ does not divide $\lambda$ in $\cU(\bA)$. We prove that $\Sigma$ is a group homomorphism. Let $\lambda_1, \lambda_2 \in \cU(\bA)$ such that $P$ does not divide both $\lambda_1, \lambda_2$. Then
\begin{align*}
\Sigma(\lambda_1\lambda_2 \pmod{P\cU(\bA)}) = \left( \dfrac{\epsilon(\lambda_1\lambda_2)}{P}\right)_n.
\end{align*}

By definition, $\lambda_1\lambda_2 \equiv \epsilon(\lambda_1\lambda_2)\pmod{P\cU(\bA)}$, $\lambda_1 \equiv \epsilon(\lambda_1)\pmod{P\cU(\bA)}$, and $\lambda_2 \equiv \epsilon(\lambda_2)\pmod{P\cU(\bA)}$, and thus
\begin{align*}
\epsilon(\lambda_1\lambda_2) \equiv \epsilon(\lambda_1)\epsilon(\lambda_2)\pmod{P\cU(\bA)}.
\end{align*}
Therefore there is an ultra-polynomial $H \in \cU(\bA)$ such that
\begin{align*}
\epsilon(\lambda_1\lambda_2) - \epsilon(\lambda_1)\epsilon(\lambda_2) = PH.
\end{align*}
Since both $\epsilon(\lambda_1\lambda_2) - \epsilon(\lambda_1)\epsilon(\lambda_2)$ and $P$ are polynomials in $\bA$, we deduce that $H$ must be a polynomial in $\bA$, and thus
\begin{align*}
\epsilon(\lambda_1\lambda_2) \equiv \epsilon(\lambda_1)\epsilon(\lambda_2) \pmod{P\bA}.
\end{align*}

By parts (i) and (ii), we deduce that
\begin{align*}
\left( \dfrac{\epsilon(\lambda_1\lambda_2)}{P}\right)_n = \left( \dfrac{\epsilon(\lambda_1)\epsilon(\lambda_2)}{P}\right)_n = \left( \dfrac{\epsilon(\lambda_1)}{P}\right)_n\left(\dfrac{\epsilon(\lambda_2)}{P}\right)_n,
\end{align*}
which deduces that
\begin{align*}
\Sigma(\lambda_1\lambda_2 \pmod{P\cU(\bA)}) = \Sigma(\lambda_1 \pmod{P\cU(\bA)})\Sigma(\lambda_2\pmod{P\cU(\bA)}).
\end{align*}
Therefore $\Sigma$ is a group homomorphism.

We contend that the kernel of $\Sigma$ is exactly the $n$-th ultra-powers in $(\cU(\bA)/ P\cU(\bA))^\times$. Indeed, suppose that
\begin{align*}
\Sigma(\lambda \pmod{P\cU(\bA)}) = \left( \dfrac{\epsilon(\lambda)}{P}\right)_n = 1
\end{align*}
for some ultra-polynomial $\lambda \in \cU(\bA)$ such that $P$ does not divide $\lambda$. Part (iii) implies that there exists a polynomial $\beta$ in $\bA$ such that $\beta^n \equiv \epsilon(\lambda)\pmod{P\cU(\bA)}$. Thus
\begin{align*}
\lambda \equiv \epsilon(\lambda)\equiv \beta^n \pmod{P\cU(\bA)},
\end{align*}
which, by Lemma \ref{lem-ultra-powers-mod-PU}, deduces  that $\lambda$ is an $n$-ultra-power in $\cU(\bA)$.

Conversely, if
\begin{align*}
(\lambda \pmod{P\cU(\bA)}) = (\beta \pmod{P\cU(\bA)})^n
\end{align*}
for some $\beta \in \cU(\bA)$, then Lemma \ref{lem-U/PU=A/PA-iso} implies that
\begin{align*}
\epsilon(\lambda)\pmod{P\cU(\bA)} = \lambda \pmod{P\cU(\bA)}  = (\beta \pmod{P\cU(\bA)})^n = (\epsilon(\beta)\pmod{P\cU(\bA)})^n,
\end{align*}
and it thus follows from Lemma \ref{lem-ultra-powers-mod-PU} that
\begin{align*}
\epsilon(\lambda) \equiv \epsilon(\beta)^n \pmod{P\cU(\bA)}.
\end{align*}
Therefore part (iii) implies that
\begin{align*}
\Sigma(\lambda \pmod{P\cU(\bA)}) = \left(\dfrac{\epsilon(\lambda)}{P}\right)_n = 1,
\end{align*}
and hence $\lambda \pmod{P\cU(\bA)}$ belongs in the kernel of $\Sigma$.

Thus the kernel of $\Sigma$ is exactly the $n$-th ultra-powers in $(\cU(\bA)/ P\cU(\bA))^\times$.

Following the proof of Lemma \ref{lem-U/PU=A/PA-iso}, we know that $(\cU(\bA)/ P\cU(\bA))^\times$ is an ultra-cyclic group of ultra-order $\kappa^d-1$. Since $\cU(\bA)/P\cU(\bA) \cong \GF(\kappa^d)$, we also know from the proof of part (iii) above that the kernel of $\Sigma$ consisting of exactly all the $n$-th ultra-powers in $(\cU(\bA)/ P\cU(\bA))^\times$ is an ultra-cyclic group of ultra-order $\frac{\kappa^d-1}{n}$. Thus it follows from the first isomorphism theorem (see \cite{lang-algebra}) and \L{}o\'s' theorem that the image of $\Sigma$ is an ultra-cyclic subgroup of $\GF(\kappa)^\times$ of ultra-degree $n$. By Theorem \ref{thm-uniqueness-of-ultra-cyclic-subgroups}, the image of $\Sigma$ is a unique ultra-cyclic subgroup of $\GF(\kappa)^{\times}$ of ultra-order $n$. Since $\zeta$ is an element in $\GF(\kappa)^\times$ of ultra-order dividing $n$, it follows from Theorem \ref{thm-uniqueness-of-ultra-cyclic-subgroups} that the ultra-cyclic subgroup ultra-generated by $\zeta$, say $\langle \zeta \rangle_u$ is a unique ultra-cyclic subgroup of the image of $\Sigma$, and thus there exists an element $\alpha \in \bA$ such that $\left( \frac{\alpha}{P}\right)_n = \zeta$.

\end{proof}

\section{Reciprocity Law}
\label{sec-reciprocity-law}

We maintain the same notation as in Section \ref{sec-power-residue-symbol}. In this section, we prove an analogue of the reciprocity law for $\bA = \GF(\kappa)[t]$. We begin with a simple but useful lemma that we will need in the proof of the reciprocity law.

\begin{lemma}
\label{lem-ultra-Frobenius-action-on-polynomials}

For every polynomial $Q(t) \in \bA$ and any positive integer $j$, $Q(t)^{\kappa^j} = Q(t^{\kappa^j})$.

\end{lemma}

\begin{proof}

If $Q(t)$ is a constant polynomial, say $Q(t) \equiv a$ for some $a \in \GF(\kappa)$, then the lemma follows trivially.

Suppose that $m = \deg(Q) \in \bZ_{> 0}$. We can write
\begin{align*}
Q(t) = a_mt^m + \cdots + a_1t + a_0,
\end{align*}
where for each $0 \le i \le m$, $a_i = \ulim_{s \in S}a_{i, s} \in \GF(\kappa)$ for some elements $a_{i, s} \in \bF_{q_s}$ such that $a_m \ne 0$. Let
\begin{align*}
Q_s(t) = a_{m, s}t^m + \cdots + a_{1, s}t + a_{0, s} \in \bF_{q_s}[t].
\end{align*}
Then $Q(t) = \ulim_{s \in S}Q_s(t)$.

Since $a_{i, s}^{q_s^j} = a_{i, s}$ for any $0 \le i \le m$, we deduce that $Q_s(t)^{q_s^j} = Q_s(t^{q_s^j})$, and thus
\begin{align*}
Q(t)^{\kappa^j} &= \ulim_{s \in S}Q_s(t)^{q_s^j} \\
&= \ulim_{s \in S}Q_s(t^{q_s^j})\\
&= \ulim_{s \in S}\left(a_{m, s}(t^{q_s^j})^m + \cdots + a_{1, s}(t^{q_s^j}) + a_{0, s}\right)\\
&= \sum_{i = 0}^m(\ulim_{s \in S}a_{i, s})(\ulim_{s \in S}(t^{q_s^j})^i \\
&= \sum_{i = 0}^ma_i(t^{\kappa^j})^i \\
&= Q(t^{\kappa^j}).
\end{align*}

\end{proof}

\begin{theorem}
(Reciprocity Law)
\label{thm-Reciprocity-Law}

Let $n$ be a hypernatural number in $\bN^{\#}$ such that $n$ divides $\kappa-1$. For any monic irreducible polynomials $P, Q$ in $\bA$,
\begin{align*}
\left( \frac{Q}{P} \right)_n = (-1)^{\frac{\kappa-1}{n}\deg(P)\deg(Q)}\left( \frac{P}{Q} \right)_n.
\end{align*}

\end{theorem}

\begin{proof}

If $P = Q$, then both $\left( \frac{Q}{P} \right)_n$, $\left( \frac{P}{Q} \right)_n$ are zero, and the theorem follows immediately. Suppose that $P \ne Q$. Since both $P$, $Q$ are monic irreducible, $P, Q$ are relatively prime in $\bA$.

Set $d = \deg(P)$ and $\ell = \deg(Q)$.

From the definition of power residue symbols, we see that
\begin{align*}
\left(\dfrac{f}{g}\right)_n = \left(\dfrac{f}{g}\right)^{\frac{\kappa-1}{n}}_{\kappa-1}
\end{align*}
for any monic irreducible polynomials $f, g$ in $\bA$, and thus the theorem follows immediately if one could show
\begin{align*}
\left(\dfrac{Q}{P}\right)_{\kappa-1} = \left(-1\right)^{d\ell}\left(\dfrac{P}{Q}\right)_{\kappa-1}
\end{align*}
since the assertion in the theorem would follow by raising both sides to the $\frac{\kappa-1}{n}$-th ultra-power.

By Theorem \ref{thm-roots-of-ultra-polynomials-in-ultra-Galois-fields}, $P$ has a root $a$ in the ultra-Galois field $\GF(\kappa^d)$, and the set of all roots of $P$ consists precisely of $d$ distinct elements $a, a^{\kappa}, \ldots, a^{\kappa^{d - 1}}$ in $\GF(\kappa^d)$. Thus one can write
\begin{align}
\label{eq1-thm-reciprocity-law-for-A}
P(t) = \prod_{i = 0}^{d - 1}(t - a^{\kappa^i}).
\end{align}

Similarly, $Q$ has a root $b$ in the ultra-Galois field $\GF(\kappa^{\ell})$, and the set of all roots of $Q$ consists precisely of $\ell$ distinct elements $b, b^{\kappa}, \ldots, b^{\kappa^{\ell - 1}}$ in $\GF(\kappa^\ell)$. Thus one can write
\begin{align}
\label{eq2-thm-reciprocity-law-for-A}
Q(t) = \prod_{i = 0}^{\ell - 1}(t - b^{\kappa^i}).
\end{align}

Let $\bL = \GF(\kappa^{d\ell})$ be the ultra-Galois field that contains both $\GF(\kappa^d)$ and $\GF(\kappa^{\ell})$ as ultra-subfields. Note that $\bL = \prod_{s\in S}\bF_{q_s^{d\ell}}/\cD$. Let $\cU(\bL[t])$ denote the ultra-hull of $\bL[x]$, i.e., $\cU(\bL[t]) = \prod_{s\in S}\bF_{q_s^{d\ell}}[t]/\cD$.

By Lemma \ref{lem-ultra-Frobenius-action-on-polynomials}, we deduce that
\begin{align}
\label{eq3-thm-reciprocity-law-for-A}
Q(t)^{\kappa^j} = Q(t^{\kappa^j})
\end{align}
for any $0 \le j \le d$. Using Lemma \ref{lem-x-c-divides-f(x)-if-c-is-a-root-of-f(x)} and (\ref{eq3-thm-reciprocity-law-for-A}), we see that
\begin{align}
\label{eq3-1/2-thm-reciprocity-law-for-A}
Q(t)^{1+\kappa+\kappa^2+\cdots+\kappa^{d-1}} &= Q(t)Q(t^\kappa)\cdots Q(t^{\kappa^{d-1}}) \nonumber \\
&\equiv Q(a)Q(a^\kappa)\cdots Q(a^{\kappa^{d-1}}) \pmod{(t-a)\cU(\bL[t])}.
\end{align}

By symmetry, the above congruence also holds modulo $(t-a^{\kappa^i})\cU(\bL[t])$ for all $0 \le i \le d - 1$, and since polynomials $t-a^{\kappa^i}$ and $t-a^{\kappa^j}$  are relatively prime in $\bL[t]$ for any $0 \le i \ne j \le d - 1$, we deduce from Lemma \ref{lem-relatively-prime-ultra-polynomials}, (\ref{eq1-thm-reciprocity-law-for-A}), (\ref{eq2-thm-reciprocity-law-for-A}), and (\ref{eq3-1/2-thm-reciprocity-law-for-A}) that
\begin{align*}
Q^{\frac{\kappa^d-1}{\kappa-1}} =  Q^{1+\kappa+\cdots+\kappa^{d-1}} \equiv \prod_{i =0}^{d - 1}  \prod_{j = 0}^{\ell - 1} (a^{\kappa^i}- b^{\kappa^j})\pmod{P\cU(\bL[t])}.
\end{align*}

By Definition \ref{def-nth-power-residue-symbol},
\begin{align*}
\left(\frac{Q}{P}\right)_{\kappa-1} \equiv Q^{\frac{\kappa^d-1}{\kappa-1}} \pmod{P\cU(\bA)}.
\end{align*}

Since $\cU(\bA) = \cU(\GF(\kappa)[t])$ is an ultra-subfield of $\cU(\GF(\kappa^{d\ell})[t]) = \cU(\bL[t])$, and thus
\begin{align*}
\left(\frac{Q}{P}\right)_{\kappa-1} \equiv Q^{\frac{\kappa^d-1}{\kappa-1}} \pmod{P\cU(\bL[t])}.
\end{align*}
Therefore
\begin{align*}
\left(\frac{Q}{P}\right)_{\kappa-1} \equiv \prod_{i =0}^{d - 1}  \prod_{j = 0}^{\ell - 1} (a^{\kappa^i}- b^{\kappa^j}) \pmod{P\cU(\bL[t])}.
\end{align*}

Since both sides of the above congruence are in $\bL$, and the degree of $P$ is a positive integer,
\begin{align}
\label{eq4-them-reciprocity-law-for-A}
\left(\dfrac{Q}{P}\right)_{\kappa-1} &= \prod_{i =0}^{d - 1}  \prod_{j = 0}^{\ell - 1} (a^{\kappa^i}- b^{\kappa^j})) \nonumber \\
&=\left(-1\right)^{d\ell} \prod_{i =0}^{d - 1}  \prod_{j = 0}^{\ell - 1} (b^{\kappa^i}- a^{\kappa^j}).
\end{align}

Using the same arguments as above with the roles of $P, Q$ exchanged, we deduce that
\begin{align*}
\left(\dfrac{P}{Q}\right)_{\kappa-1} =  \prod_{i =0}^{d - 1}  \prod_{j = 0}^{\ell - 1} (b^{\kappa^i}- a^{\kappa^j}),
\end{align*}
and it thus follows from (\ref{eq4-them-reciprocity-law-for-A} that
\begin{align*}
\left(\dfrac{Q}{P}\right)_{\kappa-1} = \left(-1\right)^{d\ell}\left(\dfrac{P}{Q}\right)_{\kappa-1}.
\end{align*}
Thus the theorem follows immediately.

\end{proof}

We can extend the definition of the $n$-th power residue symbol to the case where the prime $P$ is replaced with an arbitrary nonzero element in $\bA$ that is not necessarily prime.

\begin{definition}
\label{def-general-definition-of-power-residue-symbol}

Let $n$ be a hypernatural number in $\bN^{\#}$ such that $n$ divides $\kappa-1$. Let $\beta \in \bA$ such that $\deg(\beta) > 0$. Write $\beta =  a P_1^{d_1}P_2^{d_2}\cdots P_m^{d_m}$ for the prime factorization of $\beta$ into monic irreducible polynomials $P_i$ in $\bA$ where $a \in \GF(\kappa)^\times$ and the $d_i$ are positive integers. For each $\alpha \in \bA$, define
\begin{align*}
\left( \dfrac{\alpha}{\beta} \right)_n := \prod_{i=1}^m \left( \frac{\alpha}{P_i} \right)_n^{d_i} \in \GF(\kappa)^\times.
\end{align*}

If $\beta \in \GF(\kappa)^{\times}$, i.e., $\deg(\beta) = 0$ and $\beta \ne 0$, we simply define
\begin{align*}
\left( \dfrac{\alpha}{\beta} \right)_n := 1.
\end{align*}

\end{definition}

\begin{remark}
\label{rem-generl-power-residue-symbol-are-nth-ultra-roots}

\begin{itemize}

\item []

\item [(i)] In the prime factorization of $\beta$, $a$ is called the \textbf{sign of $\beta$} which we denote by $a = \text{sign}(\beta)$. In the above definition, we ignore $\text{sign}(\beta)$, and thus the symbol only depends on the principal ideal $\beta\bA$ generated by $\beta$.

\item [(ii)] It follows from Remark \ref{rem-power-residue-symbol-are-nth-roots-of-unity} that for any relatively prime polynomials $\alpha$, $\beta$ in $\bA$, the symbol $\left(\dfrac{\alpha}{\beta}\right)_n$ also belong in the group $\mu_n(\GF(\kappa))$ of $n$-th ultra-roots of unity in the ultra-cyclic group $\GF(\kappa)^{\times}$, i.e.,
    \begin{align*}
    \left(\dfrac{\alpha}{\beta}\right)_n^n = 1.
    \end{align*}

\end{itemize}
\end{remark}

The following result generalizes Proposition \ref{prop-properties-of-power-residue-symbol}.

\begin{proposition}
\label{prop-properties-of-general-power-residue-symbol}

The symbol $\left( \dfrac{\alpha}{\beta} \right)_n$ has the following properties: for any $\alpha, \alpha_1, \alpha_2, \beta, \beta_1, \beta_2 \in \bA$ such that $\beta, \beta_1, \beta_2$ are nonzero,

\begin{itemize}

\item[(i)] If $\alpha_1 \equiv \alpha_2 \pmod{\beta\bA}$, then $\left( \dfrac{\alpha_1}{\beta} \right)_n = \left( \dfrac{\alpha_2}{\beta} \right)_n$.

\item[(ii)] $\left(\dfrac{\alpha_1\alpha_2}{\beta} \right)_n =\left(\dfrac{\alpha_1}{\beta} \right)_n  \left( \dfrac{\alpha_2}{\beta} \right)_n$.

\item[(iii)] $  \left(\dfrac{\alpha}{\beta_1\beta_2} \right)_n =\left( \dfrac{\alpha}{\beta_1} \right)_n  \left( \dfrac{\alpha}{\beta_2} \right)_n$.

\item[(iv)]  $\left( \dfrac{\alpha}{\beta} \right)_n \neq 0$ if and only if $\alpha$ and $\beta$ are relatively prime in $\bA$.

\item[(v)] If $x^n \equiv \alpha \pmod{\beta \cU(\bA)}$ is solvable in $\bA$, then $\left(\dfrac{\alpha}{\beta} \right)_n = 1$, provided that  $\alpha, \beta$ are relatively prime.

\end{itemize}
\end{proposition}

\begin{proof}

Parts (i)--(iv) follow immediately from the definition, Proposition \ref{prop-properties-of-power-residue-symbol}, and the properties of the symbol $\left( \frac{\alpha}{P} \right)_n$ for primes $P \in \bA$.

We now prove part (v). Suppose $\lambda^n \equiv \alpha \pmod{\beta \cU(\bA)}$ for some $\lambda \in \bA$. Then parts (i), (ii), and Remark \ref{rem-generl-power-residue-symbol-are-nth-ultra-roots} imply that
\begin{align*}
\left( \dfrac{\alpha}{\beta} \right)_n = \left( \dfrac{\lambda^n}{\beta} \right)_n = \left( \dfrac{\lambda}{\beta} \right)_n^n = 1.
\end{align*}

\end{proof}

For each polynomial $f \in \bA$, we denote by $\text{sign}_n(f)$ the leading coefficient of the ultra-polynomial $f^{\frac{\kappa - 1}{n}}$. If the leading coefficient of $f$ is $a \in \GF(\kappa)^{\times}$, then
\begin{align}
\label{def-signature-of-kappa-1/n-ultra-power}
\text{sign}_n(f) = a^{\frac{\kappa - 1}{n}}.
\end{align}

We need the following lemma in the proof of the general reciprocity law below.

\begin{lemma}
\label{lem-power-residue-symbol-for-constants}

Let $P$ be a monic irreducible polynomial in $\bA$ of degree $d \in \bZ_{>0}$. For any $a \in \GF(\kappa)$, .
\begin{align*}
\left( \frac{a}{P} \right)_n = a^{\frac{\kappa-1}{n}\deg(P)}.
\end{align*}

\end{lemma}

\begin{proof}

The lemma is trivial for $a = 0$. Suppose that $a \ne 0$. We see that
\begin{align*}
\frac{\kappa^d-1}{n} = \frac{\kappa-1}{n} ( 1+\kappa+\kappa^2+\cdots + \kappa^{d-1}).
\end{align*}
Since $a\in \GF(\kappa)^{\times}$, $a^{\kappa^i}=a$ for all $0 \leq \kappa \leq i-1$, and thus
\begin{align*}
a^{\frac{\kappa^d-1}{n}} =a^{\frac{\kappa-1}{n} d} \in \GF(\kappa)^{\times}.
\end{align*}
Therefore $\left(\dfrac{a}{P} \right)_n = a^{\frac{\kappa-1}{n}\deg(P)}$.

\end{proof}

We now prove the general reciprocity law for arbitrary relative prime nonzero elements in $\bA$.

\begin{theorem}
(The general reciprocity law)
\label{thm-General-Reciprocity-Law}

Let $\alpha, \beta$ be relatively prime nonzero elements in $\bA$. Then
\begin{align*}
\left( \dfrac{\alpha}{\beta} \right)_n \left(\dfrac{\beta}{\alpha} \right)_n^{-1} = (-1)^{\frac{\kappa-1}{n} \deg(\alpha)\deg(\beta)} \text{sign}_n(\alpha)^{\deg(\beta)}  \text{sign}_n(\beta)^{-\deg(\alpha)}.
\end{align*}

\end{theorem}

\begin{remark}

In Nguyen \cite{nguyen-2023-2nd-paper}, the author proves an analogue of the Hilbert reciprocity law for $\GF(\kappa)(t)$ whose proof essentially relies on the General Reciprocity Law above.

\end{remark}

\begin{proof}

If at least one of $\alpha$, $\beta$ belongs in $\GF(\kappa)^{\times}$, the theorem follows trivially from Definition \ref{def-general-definition-of-power-residue-symbol} and Lemma \ref{lem-power-residue-symbol-for-constants}. Hence without loss of generality, suppose that both $\alpha$ and $\beta$ are of positive degrees in $\bA$.

Let $\alpha =  a P_1^{d_1}P_2^{d_2}\cdots P_r^{d_r}$  and $\beta =  b Q_1^{\ell_1}Q_2^{\ell_2}\cdots Q_m^{\ell_m}$ be the prime factorizations of $\alpha, \beta$, respectively, where $a, b \in \GF(\kappa)^\times$, the $P_i, Q_j$ are monic primes in $\bA$ and the $d_i, \ell_j$ are positive integers.

So
\begin{align}
\label{eq1-the-general-reciprocity-law}
\left( \dfrac{\alpha}{\beta} \right)_n =\prod_{i=1}^m \prod_{j=1}^r \left[ \left( \dfrac{P_j}{Q_i} \right)_n^{d_j \ell_i} \left( \dfrac{a}{Q_i} \right)_n^{\ell_i} \right].
\end{align}

We see from \eqref{eq1-the-general-reciprocity-law} and Lemma \ref{lem-power-residue-symbol-for-constants} that
\begin{equation*}
\begin{aligned}
\left( \frac{\alpha}{\beta} \right)_n &= \prod_{i=1}^m a^{\frac{\kappa-1}{n} \ell_i \deg(Q_i)} \prod_{j=1}^r \left( \frac{P_j}{Q_i} \right)_n^{d_j \ell_i} \\
&= a^{\frac{\kappa-1}{n} \sum_{i=1}^m \deg(Q_i^{\ell_i})} \prod_{i=1}^m \prod_{j=1}^r \left( \frac{P_j}{Q_i} \right)_n^{d_j \ell_i}
\end{aligned}
\end{equation*}
Note that
\begin{align*}
a^{\frac{\kappa-1}{n} \sum_{i=1}^m \deg(Q_i^{\ell_i})} = \text{sign}_n(\alpha)^{\deg(\beta)}.
\end{align*}
By Theorem \ref{thm-Reciprocity-Law},
\begin{align*}
\left( \dfrac{P_j}{Q_i} \right)_n = \left(\dfrac{Q_i}{P_j} \right)_n (-1)^{ \frac{\kappa-1}{n} \deg(P_j) \deg(Q_i)},
\end{align*}
and thus
\begin{align*}
\left( \frac{\alpha}{\beta} \right)_n = \text{sign}_n(\alpha)^{\deg(\beta)} \prod_{i=1}^m \prod_{j=1}^r \left( \frac{Q_i}{P_j} \right)_n^{d_j \ell_i} (-1)^{d_j \ell_i \frac{\kappa-1}{n} \deg(P_j) \deg(Q_i)}.
\end{align*}
Since
\begin{align*}
\deg(\alpha) \deg(\beta) &= \left( \sum_{j=1}^r \deg(P_j^{d_j}) \right) \left( \sum_{i=1}^m \deg(Q_i^{\ell_i}) \right) \\
&=  \sum_{i=1}^m  \sum_{j=1}^r \deg(P_j^{d_j}) \deg(Q_i^{\ell_i})\\
&=  \sum_{i=1}^m  \sum_{j=1}^r d_j \ell_i \deg(P_j) \deg(Q_i),
\end{align*}
we deduce that
\begin{align}
\label{eq2-the-general-reciprocity-law}
\left( \frac{\alpha}{\beta} \right)_n = (-1)^{\frac{\kappa-1}{n} \deg(\alpha) \deg(\beta)} \text{sign}_n(\alpha)^{\deg(\beta)} \prod_{i=1}^m \prod_{j=1}^r \left( \frac{Q_i}{P_j} \right)_n^{d_j \ell_i}.
\end{align}

On the other hand,
\begin{equation}
 \label{eq3-the-general-reciprocity-law}
\begin{aligned}
\left( \frac{\beta}{\alpha} \right)_n  &=  \prod_{j=1}^r  \prod_{i=1}^m \left[ \left( \frac{Q_i}{P_j} \right)_n^{d_j \ell_i} \left( \frac{b}{P_j} \right)_n^{d_j} \right] \\
&= \prod_{j=1}^r  \prod_{i=1}^m \left[ b^{ \frac{\kappa-1}{n} \deg(P_j)d_j} \left( \frac{Q_i}{P_j} \right)_n^{d_j \ell_i}  \right]
\\
&= b^{ \frac{\kappa-1}{n} \sum_{j=1}^r \deg(P_j^{d_j})} \prod_{j=1}^r  \prod_{i=1}^m \left( \frac{Q_i}{P_j} \right)_n^{d_j \ell_i} \\
&= \text{sign}_n (\beta)^{\deg(\alpha)} \prod_{j=1}^r  \prod_{i=1}^m\left( \frac{Q_i}{P_j} \right)_n^{d_j \ell_i} .
\end{aligned}
\end{equation}

Combining \eqref{eq2-the-general-reciprocity-law} and \eqref{eq3-the-general-reciprocity-law}, the theorem follows immediately.

\end{proof}

\begin{remark}
\label{rem-2nd-part-of-CP-theorem}

In this remark, we give another proof to the second part of Theorem \ref{thm-Clark-Pollack} for the polynomial ring over an ultra-finite field which is the second part of Clark--Pollack \cite[part (b), Theorem 9]{CP}. For the rest of this remark, assume that $n$ is a positive integer dividing $\kappa - 1$, and that $\alpha, \beta$ are relatively prime monic polynomials in $\bA = \GF(\kappa)[t]$. By Remark \ref{rem-power-residue-symbol-recovers-CP-PRS},
\begin{align*}
\Gamma_{\GF(\kappa), n}(x) = x^{\frac{\kappa - 1}{n}}
\end{align*}
for each $x \in \GF(\kappa)$, where the map $\Gamma_{\GF(\kappa), n}$ is defined by (\ref{eqn-Gamma-map-introduction-in-Serre}). Thus
\begin{align*}
\Gamma_{\GF(\kappa), n}(-1)^{\deg(\alpha)\deg(\beta)} = (-1)^{\frac{\kappa - 1}{n}\deg(\alpha)\deg(\beta)},
\end{align*}
and therefore the second part of Theorem \ref{thm-Clark-Pollack} follows immediately from Remark \ref{rem-power-residue-symbol-recovers-CP-PRS} and Theorem \ref{thm-General-Reciprocity-Law}.

\end{remark}

We also give another proof to \cite[Corollary 10]{CP} using Theorem \ref{thm-General-Reciprocity-Law}.

\begin{corollary}
(See Corollary 10 in Clark--Pollack \cite{CP})
\label{cor-Corollary-10-in-CP}

We maintain the same notation as above for $\GF(\kappa)$. Assume that $\GF(\kappa)$ contains the $n$-th roots of unity for all positive integers $n$. Let $\alpha, \beta$ be relatively prime monic polynomials in $\bA = \GF(\kappa)[t]$. Then, for any positive integer $n$,

\begin{itemize}

\item[(i)] $\left(\dfrac{\alpha}{\beta}\right)_n = \left(\dfrac{\beta}{\alpha}\right)_n$.

\item[(ii)] If $\alpha$ and $\beta$ are moreover irreducible, then $\alpha$ is an $n$-th power modulo $\beta$ if and only if $\beta$ is an $n$-th power modulo $\alpha$.

\end{itemize}

\end{corollary}

\begin{proof}

By Lemma \ref{lem-divisibility-of-ultra-order} and Corollary \ref{cor-nth-ultra-roots-of-unity-in-Galois-ultra-fields}, $\GF(\kappa)$ contains the $n$-th roots of unity for every positive integer $n$ if and only if $n$ divides $\kappa -1$ in $\bZ^{\#}$ for every positive integer $n$. Since $\alpha, \beta$ are monic, Theorem \ref{thm-General-Reciprocity-Law} implies that
\begin{align*}
\left(\dfrac{\alpha}{\beta}\right)_n = (-1)^{\frac{\kappa - 1}{n}\deg(\alpha)\deg(\beta)}\left(\dfrac{\beta}{\alpha}\right)_n.
\end{align*}

Since $2n$ divides $\kappa - 1$, $\kappa - 1 = 2n\lambda$ for some hyperinteger $\lambda \in \bZ^{\#}$. Thus
\begin{align*}
(-1)^{\frac{\kappa - 1}{n}\deg(\alpha)\deg(\beta)} = (-1)^{2\lambda\deg(\alpha)\deg(\beta)} = 1,
\end{align*}
and therefore we deduce from the above equation that
\begin{align*}
\left(\dfrac{\alpha}{\beta}\right)_n = \left(\dfrac{\beta}{\alpha}\right)_n.
\end{align*}
So part (i) follows.

Part (ii) follows from part (i), the definition of $n$-th power residue symbol, and Proposition \ref{prop-properties-of-power-residue-symbol}.

\end{proof}

\begin{remark}

We construct infinitely many hyperintegers $\kappa$ that satisfy the conditions of Corollary \ref{cor-Corollary-10-in-CP}. Indeed, let $S = \bZ_{>0}$, and let $\cD$ be a nonprincipal ultrafilter on $S$. For each $m \in \bZ_{>0}$, let $p_m = m!r_m + 1 = (1\cdot2 \cdots m)r_m + 1$ for some integer $r_m$. By Dirichlet's theorem on primes in arithmetic progressions (see Serre \cite{serre-arithmetic}), one can choose infinitely many integers $r_m$ such that the $p_m$ are distinct primes. Let $(\epsilon_m)_{m \in \bZ_{> 0}}$ be a sequence of positive integers, and for each $m \in \bZ_{>0}$, set $q_m = p_m^{\epsilon_m}$. Set
\begin{align*}
\kappa = \ulim_{m \in \bZ_{>0}}q_m \in \bZ^{\#}.
\end{align*}

Take any positive integer $n$. We see that $n$ divides $p_m - 1$ for all $m \ge n$, and thus $q_m - 1 = p_m^{\epsilon_m} - 1$ is divisible by $n$ for all $m \ge n$. Since $\cD$ is a nonprincipal ultrafilter, we deduce that $n$ divides $q_m - 1$ for $\cD$-almost all $m \in S$, and thus $n$ divides $\kappa - 1 = \ulim_{m \in \bZ_{>0}}(q_m - 1)$ in $\bZ^{\#}$ for every positive integer $n$. Thus the ultra-finite field $\GF(\kappa) = \prod_{m \in S}\bF_{q_m}/\cD$ with $\kappa$ above satisfies the conditions of Corollary \ref{cor-Corollary-10-in-CP}.

\end{remark}

\section{An analogue of the Grunwald-Wang theorem}
\label{sec-analogue-of-GW-thm}

We maintain the same notation as in Section \ref{sec-power-residue-symbol}. In this section, we prove an analogue of the Grunwald--Wang theorem for the ring of polynomials over an ultra-finite field. For the classical Grunwald--Wang theorem, see Artin--Tate \cite{Artin-Tate}.

\begin{lemma}
\label{lem-infinitely-many-primes-in-A}

Let $n$ be a positive integer in $\bZ$. There are infinitely many monic primes in $\bA$ of degree $n$. More precisely, let $\cP(n)$ denote the set of monic primes of degree $n$ in $\bA$. Then
\begin{align*}
\cP(n) = \prod_{s\in S} \cP_{q_s}(n)/ \cD,
\end{align*}
where for each $s \in S$, $\cP_{q_s}(n)$ denotes the set of monic primes of degree $n$ in $\bA_s = \bF_{q_s}[t]$. In particular, the internal cardinality of $\cP(n)$ satisfies $\icard(\cP(n)) \ge \dfrac{\kappa^{n/2}}{n}$.

\end{lemma}

\begin{proof}

Let $f(t) = t^n+ \sum_{i=1}^{n-1} a_i t^i$ be a monic prime of degree $n$ in $\bA$, where the $a_i$ are in $\GF(\kappa)$. For each $i$, let
\begin{align*}
a_i = \ulim_{s\in S} a_{i,s},
\end{align*}
where $a_{i,s} \in \bF_{q_s}$. Then it is easy to verify that
\begin{align*}
f(t)= \ulim_{s\in S} f_s(t),
\end{align*}
where $f_s(t) = t^n+ \sum_{i=1}^{n-1} a_{i, s} t^i \in \bA_s$.

Since $f(t)$ is prime, it follows from \L{}o\'s' theorem that $f_s(t)$ is prime for $\cD$-almost all $s \in S$, and thus $f_s(t) \in \cP_s(n)$ for $\cD$-almost all $s \in S$. Thus
\begin{align*}
\cP(n) \subseteq \prod_{s\in S} \cP_s(n)/\cD.
\end{align*}

Let $f(t) = \ulim_{s \in S}f_s(t) \in \prod_{s \in S} \cP_s(n)/ \cD$, where $f_s(t) = t^n + \sum_{i=0}^{n-1} a_{i,s}t^i \in \cP_s(n)$ for $\cD$-almost all $s \in S$. For each $s \in S$, set
\begin{align*}
a_i = \ulim_{s \in S} a_{i,s} \in \GF(\kappa).
\end{align*}
Then we see that
\begin{align*}
f(t)= t^n+ \sum_{i=0}^{n-1} a_i t^i \in \bA,
\end{align*}
and it follows from \L{}o\'s' theorem that $f$ is prime in $\bA$. Thus $f(t) \in \cP(n)$, and so $\cP(n) = \prod_{s \in S} \cP_s(n) / \cD$.

By \cite[Theorem 2.2]{rosen}, we know that the cardinality of $\cP_s(n)$ is greater than $\frac{q_s^{n/2}}{n}$ for all sufficiently large prime powers $q_s$. Thus the internal cardinality of $\cP(n)$ satisfies
\begin{align*}
\icard(\cP(n)) &= \ulim_{s \in S} \card(\cP_s(n)) \\
&\geq \ulim_{s \in S} \frac{q_s^{n/2}}{n} = \frac{k^{n/2}}{n} ,
\end{align*}
which in particular implies that $\cP(n)$ is uncountable.

\end{proof}

The following is an analogue of the Grunwald--Wang theorem for $\bA$ with the restriction that $n$ divides $\kappa - 1$. We will remove this condition in Theorem \ref{thm-Grunwald-Wang}.

\begin{lemma}
\label{lem-for-thm-Grunwald-Wang}

Let $\alpha$ be a polynomial in $\bA$ of positive degree, and let $n$ be a positive integer that divides $\kappa-1$ in $\bZ^{\#}$. If $x^n \equiv \alpha \pmod{P\bA}$ is solvable in $\bA$ for all but finitely many monic primes $P$ of $\bA$, then $\alpha = \beta^n$ for some $\beta \in \bA$.

\end{lemma}

\begin{proof}

Let $\alpha = a Q_1^{\ell_1}Q_2^{\ell_2} \cdots Q_m^{\ell_m}$ be the prime factorization of $\alpha$, where $a \in \GF(\kappa)^\times$, the $Q_i$ are monic primes in $\bA$, and the $\ell_i$ are positive integers.

Suppose that there is some $\ell_i$ that is not divisible by $n$. We can claim that there are infinitely many primes $\fp$ in $\bA$ such that $\left( \dfrac{\alpha}{\fp} \right)_n \neq 1$.

By reindexing the $Q_i$, if necessary, we can assume that $\ell_1$ is not divisible by $n$. By part (iii) of Proposition  \ref{prop-properties-of-power-residue-symbol} and since $x^n \equiv \alpha \pmod{P\bA}$  is solvable in $\bA$ for all but finitely many monic primes $P$, there are only finitely many monic primes $\fp_1, \fp_2, \ldots, \fp_h$ in $\bA$ such that $\left( \frac{\alpha}{\fp_i} \right)_n \neq 1$ for all $1 \leq i \leq h$, and for any monic prime $P$ in $\bA$ that are distinct from $\fp_1, \ldots, \fp_h$, $\left( \frac{\alpha}{P} \right)_n = 1$.

Let $\zeta_n$ be a primitive $n$-th root of unity in $\GF(\kappa)^{\times}$ that exists since $n$ divides $\kappa - 1$. By part (iv) of Proposition  \ref{prop-properties-of-power-residue-symbol}, there exists an element $\epsilon \in \bA$ such that
\begin{align*}
\left( \frac{\epsilon}{Q_1} \right)_n = \zeta_n.
\end{align*}

Note that since $0^n \equiv \alpha \pmod{Q_i\bA}$ for all $1 \leq i \leq m$, it follows from part (iii) of  Proposition  \ref{prop-properties-of-power-residue-symbol}, that $\left( \dfrac{\alpha}{Q_i} \right)_n = 1$ for all $1 \leq i \leq m$.

Thus all primes $Q_i, \fp_j$ are distinct. Therefore by the Chinese Remainder Theorem for $\bA$ (see Lang \cite{lang-algebra}), there exists a polynomial $\lambda_0 \in \bA$ such that $\lambda_0 \equiv \epsilon \pmod{Q_1\bA}$, $\lambda_0 \equiv 1 \pmod{Q_i \bA}$ for all $2 \leq i \leq m$, and $\lambda_0 \equiv 1 \pmod{\fp_j \bA}$ for all $1 \leq j \leq h$.

Consider the polynomial
\begin{align}
\label{eq6-lem-Grunwald-Wang}
\lambda = \lambda_0 + c \prod_{i=1}^m Q_i \prod_{j=1}^h \fp_j \in \bA,
\end{align}
for some polynomial $c \in \bA$ that will be determined later. Let
\begin{align*}
d_0 =  \deg( \prod_{i=1}^m  Q_i \prod_{j=1}^h \fp_j ) \in \bZ_{>0}.
\end{align*}
Let $u$ be any positive integer such that
\begin{align*}
2un-d_0 > \deg(\lambda_0).
\end{align*}
Then for any polynomial $c \in \bA$ of degree $2un-d_0$, we see that the degree of $\lambda$ is $2un$.

We can choose $c$ to be a monic polynomial so that $\lambda$ is a monic polynomial. By \eqref{eq6-lem-Grunwald-Wang}, we know that
\begin{align*}
\left \{ \begin{array}{l}
                    \lambda \equiv \lambda_0 \equiv \epsilon \pmod{Q_1 \bA}, \\
                    \lambda \equiv \lambda_0 \equiv 1 \pmod{Q_i \bA}, 2 \leq i \leq m, \\
                    \lambda \equiv \lambda_0 \equiv 1 \pmod{\fp_j \bA}, 1 \leq j \leq h.
         \end{array} \right.
\end{align*}

By Proposition \ref{prop-properties-of-general-power-residue-symbol},
\begin{align}
\label{eq7-lem-Grunwald-Wang}
\left( \dfrac{\lambda }{\alpha} \right)_n = \left( \dfrac{\epsilon }{Q_1} \right)^{\ell_1} = \zeta_n^{\ell_1} \neq 1,
\end{align}
since $\ell_1$ is not divisible by $n$.

By the General Reciprocity Law \ref{thm-General-Reciprocity-Law},
\begin{align}
\label{eq8-lem-Grunwald-Wang}
\left( \frac{\alpha}{\lambda} \right)_n =\left( \frac{\lambda }{\alpha} \right)_n (-1)^{\frac{\kappa-1}{n} \deg(\alpha)\deg(\lambda)} \text{sign}_n(\alpha)^{\deg(\lambda)} \text{sign}_n(\lambda)^{-\deg(\alpha)}.
\end{align}
We see that
\begin{align*}
(-1)^{\frac{\kappa-1}{n} \deg(\alpha)\deg(\lambda)} = (-1)^{2(\kappa-1) \deg(\alpha)u} =1
\end{align*}
since $\deg(\lambda) = 2un$.

Recall that $\text{sign}_n(\alpha)$ is the leading coefficient of the ultra-polynomial $\alpha^{\frac{\kappa-1}{n} }$, and thus $\text{sign}_n(\alpha) = a^{\frac{\kappa-1}{n}}$. Therefore, it follows from Lemma \ref{lem-ultra-cyclic-structures} that
\begin{align*}
\text{sign}_n(\alpha)^{\deg(\lambda)} = a^{2(\kappa-1)u}=1
\end{align*}
since $a \in \GF(\kappa)^\times$, $\deg(\lambda)=2un$, and $\kappa = \ulim_{s\in S}q_s$.

Since $\lambda$ is monic, the leading coefficient of $\lambda^{\frac{\kappa-1}{n} }$ is $1$, and thus
\begin{align*}
\text{sign}_n(\lambda)^{-\deg(\alpha)} =1.
\end{align*}
Thus by \eqref{eq7-lem-Grunwald-Wang}, \eqref{eq8-lem-Grunwald-Wang}, we see that
\begin{align*}
\left( \dfrac{\alpha}{\lambda} \right)_n =\left(\dfrac{\lambda }{\alpha} \right)_n = \zeta_n^{\ell_1} \neq 1.
\end{align*}
It follows that there exists a monic prime $P$ in $\bA$ such that $P$ divides $\lambda$ and $\left(\dfrac{\alpha}{P} \right)_n \neq 1$, which implies that $P = \fp_z$ for some $1 \le z \le h$.

We know that $\lambda \equiv 1\pmod{\fp_j\bA}$ for every $1 \le j \le h$, we deduce that $P = \fp_z \neq \fp_j$ for all $1 \le j \le h$, which is a contradiction.

Thus $\ell_i$ is divisible by $n$ for all $1 \leq i \leq m$ and therefore
\begin{align*}
\alpha = a \alpha_0^n
\end{align*}
for some $a \in \GF(\kappa)^\times$ and $\alpha_0 \in \bA$. It remains to show that $a$ is a $n$-th power in $\GF(\kappa)^{\times}$.

We know from Proposition \ref{prop-properties-of-power-residue-symbol} and Lemma \ref{lem-power-residue-symbol-for-constants} that
\begin{align*}
\left (\dfrac{\alpha}{P} \right)_n &= \left(\dfrac{a\alpha_0^n}{P} \right)_n = \left (\dfrac{a}{P} \right)_n \left(\dfrac{\alpha_0}{P} \right)_n^n = \left(\dfrac{a}{P} \right) \\
&= a^{\frac{\kappa-1}{n} \deg(\P)}
\end{align*}
for any prime $P$ in $\bA$.

Since $x^n \equiv \alpha \pmod{P\bA}$ is solvable in $\bA$ for all but finitely many monic primes $P$, it follows from Lemma \ref{lem-infinitely-many-primes-in-A} that there exists a monic prime $P$ in $\bA$ of a sufficiently large degree that is prime to $n$ such that $x^n \equiv \alpha \pmod{P\bA}$ is solvable in $\bA$.

Equivalently, using Part (iii) of Proposition \ref{prop-properties-of-power-residue-symbol}, there exists a monic prime $P$ of degree prime to $n$ such that $\left(\dfrac{\alpha}{P} \right)_n =1$. Thus
\begin{align*}
a^{\frac{\kappa-1}{n} \deg(\P)} =1.
\end{align*}

Since $\GF(\kappa) = \prod_{s \in S} \bF_{q_s}/ \cD$, one can write $a = \ulim_{s \in S}a_s$, where $a_s \in \bF_{q_s}$. By the equation above, we see that
\begin{align*}
a^{\frac{\kappa-1}{n} \deg(\P)} = \ulim_{s \in S} a_s^{\frac{q_s-1}{n} \deg(P)} =1,
\end{align*}
and thus it follows from \L{}o\'s' theorem that $a_s^{\frac{q_s-1}{n} \deg(P)} =1$ for $\cD$-almost all $s \in S$.

Since $\deg(P)$ is prime to $n$, $a_s^{\frac{q_s-1}{n} } =1$ in $\bF_{q_s}$, and it then follows from the theory of finite fields (see \cite{Lidl-Niederreiter}) that $a_s$ is an $n$-th power in $\bF_{q_s}$ for $\cD$-almost all $s \in S$, i.e., there is an element $b_s \in \bF_{q_s}$ such that $a_s = b_s^n$ for $\cD$-almost all $s \in S$. Thus
\begin{align*}
a = \ulim_{s \in S}a_s = \ulim_{s \in S}b_s^n = (\ulim_{s \in S}b_s)^n = b^n,
\end{align*}
where $b =\ulim_{s\in S}b_s \in \GF(\kappa)$. Thus $\alpha = a \alpha_0^n = (b\alpha_0)^n$, which proves the lemma.

\end{proof}

\begin{lemma}
\label{lem-powers-of-elements-in-A}

Let $n$ be a positive integer, and let $d = \gcd(n, \kappa-1) \in \bZ^{\#}_{>0}$. Then
\begin{align*}
\GF(\kappa)^{\times^n} = \GF(\kappa)^{\times^d},
\end{align*}
where $\GF(\kappa)^{\times^n}  = \{ a^n: a \in \GF(\kappa)^\times \}$ and $\GF(\kappa)^{\times^d}  = \{ a^d: a \in \GF(\kappa)^\times\}$.

\end{lemma}

\begin{remark}
\label{rem-preceding-thm-Grunwald-Wang}

Note that since $d$ is a divisor of $n$ in $\bZ^{\#}$ and $n \in \bZ_{>0}$, $d$ must be an integer, and thus $d \in \bZ_{>0}$.

\end{remark}

\begin{proof}

By the nonstandard B\'ezout's identity \ref{thm-Bezout}, there exit hyperintegers $u, v \in \bZ^{\#}$ such that
\begin{align*}
un + (\kappa-1)v=d.
\end{align*}

For all $a \in \GF(\kappa)^\times$, we know from Lemma \ref{lem-ultra-cyclic-structures} that $(a^v)^{\kappa-1} =1$, and thus
\begin{align*}
a^d = (a^u)^n (a^v)^{\kappa-1} = (a^u)^n.
\end{align*}
Therefore
\begin{align*}
\GF(\kappa)^{\times^d} \subseteq \GF(\kappa)^{\times^n}.
\end{align*}

For the inverse implication, we see that
\begin{align*}
a^n = (a^{\frac{n}{d}})^d
\end{align*}
since $\frac{n}{d} \in \bZ$. Thus $\GF(\kappa)^{\times^n} \subseteq \GF(\kappa)^{\times^d}$, and hence the lemma follows immediately.

\end{proof}

We are now ready to prove an analogue of the Grunwald--Wang theorem for $\bA$ for any positive integer $n$ that is prime to the characteristic of $\GF(\kappa)$.

\begin{theorem}
\label{thm-Grunwald-Wang}

Let $\alpha$ be a polynomial in $\bA$ of positive degree. Let $n$ be a positive integer such that either $\GF(\kappa)$ has characteristic $0$ or $n$ is prime to the characteristic of $\GF(\kappa)$. If $x^n \equiv \alpha \pmod{P\bA}$ for all but finitely many monic primes $P$ in $\bA$, then $\alpha = \beta^n$ for some $\beta \in \bA$.

\end{theorem}

\begin{remark}
\label{rem-thm-Grunwald-Wang}

Theorem \ref{thm-Grunwald-Wang} refines Lemma \ref{lem-for-thm-Grunwald-Wang} by removing the assumption that $n$ divides $\kappa-1$ in $\bN^{\#}$.

\end{remark}

\begin{proof}

Throughout the proof, we denote by $\phi(\cdot)$ the Euler's totient function, i.e., $\phi(n)$ is the number of integers $k$ with $1 \leq k \leq n$ for which $\text{gcd}(\kappa, n)=1$.

There are only finitely many primes in $\bZ$ that divide $n$. Thus since either $\GF(\kappa)$ has characteristic $0$ in which case the $\bF_{q_s}$ are of distinct characteristics for $\cD$-almost all $s \in S$, or $n$ is prime to the characteristic of $\GF(\kappa)$, we deduce that the set $\{s \in S\; : \; \gcd(q_s, n) = 1\} \in \cD$, i.e., $q_s$ and $n$ are relatively prime for $\cD$-almost all $s \in S$. Thus we know from Euler's theorem (see Serre \cite{serre-arithmetic}) that $q_s^{\phi(n)}-1 \equiv 0 \pmod{n}$ for $\cD$-almost all $s \in S$. Thus $n$ divides $\kappa^{\phi(n)} - 1 = \ulim_{s \in S}(q_s^{\phi(n)} - 1)$ in $\bZ^{\#}$.

Let $\bL := \GF(\kappa^{\phi(n)})$ be a unique field extension of $\GF(\kappa)$ of degree $\phi(n)$ (see Corollary \ref{cor-uniqueness-of-ultra-Galois-fields} ). The extension $\bL$ can be described as $\bL = \prod_{s \in S} \bF_{q_s^{\phi(n)}}/ \cD$, and thus $\bL$ has internal cardinality $\ulim_{s \in S}q_s^{\phi(n)} = \kappa^{\phi(n)}$. The group $\bL^\times$ of nonzero elements of $\bL$ is an ultra-cyclic group that is the ultraproduct $\prod_{s \in S} \bF_{q_s^{\phi(n)}}^{\times}/ \cD$, and is of internal cardinality $\kappa^{\phi(n)}-1$.

Let $\bB = \bL[t]$ be the ring of polynomials over $\bL$.

For any monic prime $Q$ of $\bB$, $Q\bB \cap \bA$ is a prime ideal of $\bA$, and since $\bA$ is a principal ideal domain, $Q\bB \cap \bA = P\bA$ for some monic prime $P$ of $\bA$.

Note that for each monic prime $P$ of $\bA$, there are only finitely many monic primes $Q$ of $\bB$ such that $Q \bB \cap \bA = P\bA$ (see \cite{Zariski-Samuel}).

Since $x^n \equiv \alpha \pmod{P\bA}$ is solvable in $\bA$ for all but finitely many monic primes $P$ of $\bA$, we contend that $x^n \equiv \alpha \pmod{Q\bB}$ is solvable in $\bB$ for all but finitely many monic primes $Q$ of $\bB$. Indeed, let $\{\fp_1, \ldots, \fp_r\}$ be the set of all monic primes in $\bA$ such that the equation $x^n \equiv \alpha \pmod{\fp_i\bA}$ is not solvable in $\bA$ for all $1 \le i \le r$. Let $\cU$ denote the set of monic primes $Q$ in $\bB$ such that $Q\bB \cap \bA = \fp_i\bA$ for some $1 \le i \le r$. Since $\bA$ is a Dedekind domain with fraction field $\GF(\kappa)(t)$ and $\bB$ is the integral closure of $\bA$ in $\bL(t) = \GF(\kappa^{\phi(n)})(t)$ \footnote{Using Corollary \ref{cor-finite-extensions-of-ultra-finite-fields-are-Galois}, we can also prove that $\bL(t)$ is a Galois extension of degree $\phi(n)$ over $\GF(\kappa)(t)$.}, $\bB$ is itself a Dedekind domain. Thus the set $\cU$ is a finite set.

Take any monic prime $Q$ that does not belong in $\cU$, and let $P$ be a monic prime in $\bA$ such that $Q\bB \cap \bA = P\bA$. Since $Q$ does not belong in $\cU$, $P$ is distinct from $P_1, \ldots, P_h$, and thus the equation $x^n \equiv \alpha \pmod{P\bA}$ is solvable in $\bA$. Thus it follows immediately that the equation $x^n \equiv \alpha \pmod{Q\bB}$ is solvable in $\bB$. Therefore the equation $x^n \equiv \alpha \pmod{Q\bB}$ is solvable in $\bB$ for all but finitely many monic primes $Q$ of $\bB$.

Since $n$ divides $\kappa^{\phi(n)}-1 = \ulim_{s \in S}(q_s^{\phi(n)} -1)$, applying Lemma \ref{lem-for-thm-Grunwald-Wang} for the ring $\bB = \bL[t]$, we deduce that $\alpha = \alpha_0^n$ for some polynomial $\alpha_0 \in \bB$. We claim that $\alpha_0$ can be chosen to be in $\bA$. Indeed, the prime factorization of $\alpha_0$ in $\bB$ is of the form: $\alpha_0 =  a Q_1^{\ell_1} Q_2^{\ell_2} \cdots Q_m^{\ell_m}$, where $a \in \bL^\times$, the $Q_i$ are distinct monic primes in $\bB$, and the $\ell_i$ are positive integers. Thus
\begin{align}
\label{eq9-thm-Grunwald-Wang}
\alpha =  a^n Q_1^{n\ell_1} Q_2^{n\ell_2} \cdots Q_m^{n\ell_m}.
\end{align}
Since $\alpha \in \bA$, we can write
\begin{align}
\label{eq10-thm-Grunwald-Wang}
\alpha =  b P_1^{f_1}  \cdots P_h^{f_h},
\end{align}
where $b \in \GF(\kappa)^\times$, the $P_j$ are distinct monic primes of $\bA$, and the $f_j$ are positive integers.

For each $1 \leq j \leq h$, the prime factorization of $P_j$ in $\bB$ is of the form
\begin{align}
\label{eq11-thm-Grunwald-Wang}
P_j = Q_{j,1} \cdots Q_{j, m_j},
\end{align}
where the $Q_{j,k}$ are distinct monic primes in $\bB$ since any irreducible polynomial over $\GF(\kappa)$ is separable, and thus it has no repeated roots in an algebraic closure of $\GF(\kappa)$.

By \eqref{eq9-thm-Grunwald-Wang}, \eqref{eq10-thm-Grunwald-Wang}, \eqref{eq11-thm-Grunwald-Wang},
\begin{align}
\label{eq12-thm-Grunwald-Wang}
\alpha =  a^n Q_1^{n\ell_1}\cdots Q_m^{n\ell_m} = b\left( \prod_{i=1}^{m_1} Q_{1, i} \right)^{f_1} \cdots \left( \prod_{i=1}^{m_h} Q_{h, i} \right)^{f_h}.
\end{align}
We claim that for any $1 \leq i \neq j \leq h$,
\begin{align*}
Q_{i,k} \neq Q_{j,\ell}
\end{align*}
for any $k, \ell$. Indeed, assume the contrary, i.e.,
\begin{align*}
Q_{i,k} = Q_{j,\ell}
\end{align*}
for some $k, \ell$. Since $P_i$ divides $Q_{i,k}$ and $P_j$ divides $Q_{j,\ell}$, we deduce that
\begin{align*}
P_i \bA = Q_{i,k} \bB \cap \bA = Q_{j,\ell}\bB \cap \bA = P_j \bA,
\end{align*}
and thus since $P_i, P_j$ are monic primes in $\bA$,   $P_i = P_j$, a contradiction.

By the uniqueness of the prime factorization of $\alpha$ in $\bB$, we deduce from \eqref{eq12-thm-Grunwald-Wang} that for any $1 \leq i \leq h$ and $1 \leq j \leq m_i$,
\begin{align*}
Q_{i,j}^{f_i} = Q^{n\ell_{\kappa_{i, j}}}_{\kappa_{i, j}}
\end{align*}
for some $1 \leq \kappa_{i, j} \leq m$.

Since both $Q_{\kappa_{i, j}}, Q_{i,j}$ are monic primes in $\bB$, $Q_{i,j} = Q_{\kappa_{i, j}}$ and $f_i = n\ell_{\kappa_{i, j}}$ for any $1 \leq i \leq h$ and $1 \le j \le m_i$.

Then it follows from \eqref{eq10-thm-Grunwald-Wang} that
\begin{align*}
\alpha &=  b ( P_1^{\ell_{\kappa_{1, 1}}} P_2^{\ell_{\kappa_{2, 1}}}  \cdots P_h^{\ell_{\kappa_{h, 1}}} )^n  \\
&= b\beta^n,
\end{align*}
where $\beta = P_1^{\ell_{\kappa_{1, 1}}} P_2^{\ell_{\kappa_{2, 1}}}  \cdots P_h^{\ell_{\kappa_{h, 1}}} \in \bA$.

It remains to show that $b$ is an $n$-th power in $\bA$.

Since $x^n \equiv \alpha = b\beta^n \pmod{P\bA}$ is solvable in $\bA$ for all but finitely many monic primes $P$ of $\bA$, we deduce that $x^n \equiv b \pmod{P\bA}$ is solvable in $\bA$ for all but finitely many monic primes $P$ of $\bA$, excluding also those monic primes $P$ dividing $\beta$.

Let $d = \text{gcd}(n, \kappa-1)$ in $\bZ^{\#}$. By Remark \ref{rem-preceding-thm-Grunwald-Wang}, $d$ must be a positive integer. Since $d$ divides $n$ in $\bZ$, we deduce that $x^d \equiv b \pmod{P\bA}$ is solvable in $\bA$ for all but finitely many monic primes $P$ of $\bA$. Since $d$ divides $\kappa-1$ in $\bZ^{\#}$, Lemma \ref{lem-for-thm-Grunwald-Wang} implies that $b =  c^d$ for some $c \in \bA$.

Since $b \in \GF(\kappa)^\times$, $c$ must be in $\GF(\kappa)^\times$, and thus $b \in \GF(\kappa)^{\times^d}$.

By Lemma \ref{lem-powers-of-elements-in-A}, $\GF(\kappa)^{\times^d} = \GF(\kappa)^{\times^n}$, and thus $b \in  \GF(\kappa)^{\times^n}$, which implies that $b=\lambda^n$ for some $\lambda \in \GF(\kappa)^\times$. Therefore
\begin{align*}
\alpha = b \beta^n = (\lambda \beta)^n
\end{align*}
for some $\lambda \in \GF(\kappa)^\times$ and $\beta \in \bA$, which proves the theorem.

\end{proof}

We now recast Theorem \ref{thm-Grunwald-Wang} in the form that can be viewed as an analogue of the Grunwald--Wang theorem for the ring of polynomials over an ultra-finite field.

For a monic prime $P$ of $\bA=\GF(\kappa)[t]$, we can define the following $P$-adic valuation $v_P$ on the rational function field $\Sigma = \GF(\kappa)(t)$ that is the quotient field of $\bA$: let $v_P(0) = \infty$ and $v_P(P^n \frac{f}{g})=n$, where $n \in \bZ$, $f, g \in \Sigma\setminus\{0\}$ are not divisible by $P$.

We denote by $\Sigma_P$ the completion of $\Sigma$ with respect to the $P$-adic valuation $v_P$. Let $\cO_P$ be the valuation ring of $\Sigma_P$, i.e.,
\begin{align*}
\cO_P = \{x \in \Sigma_P | v_P(x) \geq 0\},
\end{align*}
where by abuse of notation, we also denote by $v_P$ the unique extension of the $P$-adic valuation to $\Sigma_P$.

The residue class field $\overline{\Sigma}_P = \cO_P / \cM_P$ with $\cM_P = \{ x \in \Sigma_P | v_P(x)>0 \}$, is canonically isomorphic to the field $\bA/P\bA = \GF(\kappa)[t]/P(\GF(\kappa)[t])$ which can be identified with the unique ultra-extension of degree $\deg(P)$ over $\GF(\kappa)$, that is, the ultra-finite field $\GF(\kappa^{\deg(P)})$ (see Corollary \ref{cor-uniqueness-of-ultra-Galois-fields}).

Let $\alpha$ be a polynomial in $\bA$ of positive degree and let $n$ be a positive integer such that either $\GF(\kappa)$ has characteristic $0$ or $n$ is prime to the characteristic of $\GF(\kappa)$. Let $P$ be a monic prime in $\bA$ such that $P$ dose not divide $\alpha$. Suppose that the equation $x^n \equiv \alpha \pmod{P\bA}$ is solvable in $\bA$. Let $f(x) = x^n-\alpha$ be a polynomial over $\bA$. Then $f(x) \in \bA[x] \subseteq \cO_P[x]$. By assumption, there exists an element $a_0 \in \bA \subseteq \cO_P$ such that $f(a_0) \equiv 0 \pmod{P\bA}$, and thus $v_P(f(a_0)) \geq 1$.

We know that $a_0 \not \equiv 0 \pmod{P\bA}$; otherwise, $a_0  \equiv 0 \pmod{P\bA}$, and thus $\alpha \equiv a_0^n  \equiv 0 \pmod{P\bA}$, a contradiction.

Since $\overline{\Sigma}_P \cong \bA / P\bA \cong \GF(\kappa^{\deg(P)})$ is an ultra-extension of $\GF(\kappa)$, and either $\GF(\kappa)$ has characteristic $0$ or $n$ is prime to the characteristic of $\GF(\kappa)$, the formal derivative $f'(x) =  nx^{n-1}$ satisfies $f'(a_0) =  na_0^{n-1} \not \equiv 0 \pmod{P\bA}$. So $v_P(f'(a_0)) = 0$, and hence
\begin{align*}
v_P(f(a_0)) > 2v_P(f'(a_0)).
\end{align*}
By Hensel's lemma (see \cite[Theorem 1.3.1]{prestel}), there exists some element $a \in \cO_P$ such that $f(a)=0$, i.e., $x^n = \alpha$ is solvable in $\cO_P$ if and only if $x^n = \alpha$ is solvable in $\overline{\Sigma}_P$.

By Theorem \ref{thm-Grunwald-Wang} and the discussion above, one obtains an equivalent formulation of Theorem \ref{thm-Grunwald-Wang} in the local-to-global principle form.

\begin{theorem}
(analogue of the Grunwald--Wang theorem)
\label{thm-local-to-global-GW}

Let $\alpha$ be a polynomial in $\bA$ of positive degree, and let $n$ be a positive integer such that either $\GF(\kappa)$ has characteristic $0$ or $n$ is prime to the characteristic of $\GF(\kappa)$. Let $\Sigma= \GF(\kappa)(t)$ be the quotient field of $\bA$. Then $x^n = \alpha$ is solvable in $\bA$ if and only if $x^n = \alpha$ is solvable in $\cO_P$ for all but finitely many monic primes $P$ of $\bA$.

\end{theorem}

\begin{remark}
\label{rem-after-local-to-global-GW}

Since $\alpha \in \bA$, $x^n = \alpha$ has a solution in $\bA$ iff $x^n = \alpha$ has a solution in $\Sigma$. Thus the assertion in Theorem \ref{thm-Grunwald-Wang} is equivalent to the following: $x^n = \alpha$ is solvable in $\Sigma$ if and only if $x^n = \alpha$ is solvable in $\Sigma_P$ for all but finitely many monic primes $P$ of $\bA$.

\end{remark}

\end{document}